\documentclass[a4paper,11pt]{article}
\usepackage{mystyle}

\newcommand{\vz}[1]{\ensuremath{\mathbb{#1}}}
\newcommand{\R}{{\vz R}}
\newcommand{\N}{{\vz N}}

\newcommand{\Leb}{\mathrm{Leb}}

\long\def\authornote#1{%
        \leavevmode\unskip\raisebox{-3.5pt}{\rlap{$\scriptstyle\diamond$}}%
        \marginpar{\raggedright\hbadness=10000
        \def\baselinestretch{0.8}\tiny
        \it #1\par}}

\long\def\authornote#1{\relax}

\title{Deep Limits of Residual Neural Networks}
\author[1,2]{Matthew Thorpe\thanks{Email: \url{matthew.thorpe-2@manchester.ac.uk}}}
\author[3]{Yves van Gennip}
\affil[1]{Department of Mathematics,\protect\\ University of Manchester,\protect\\ Manchester, M13 9PL\vspace{\baselineskip}}
\affil[2]{The Alan Turing Institute,\protect\\ London, NW1 2DB, UK\vspace{\baselineskip}}
\affil[3]{Delft Institute of Applied Mathematics,\protect\\ Delft University of Technology,\protect\\ 2628 CD  Delft, The Netherlands}
\date{November 2022}

\usepackage{etoolbox}
\usepackage{xcolor}
\tracingpatches
\makeatletter
\patchcmd{\@setref}{\bfseries ??}{\bfseries\color{red} undefined Label}{}{}
\patchcmd{\@citex}{\bfseries ?}{\bfseries\color{red} undefined Citation}{}{}
\makeatother

\usepackage{xcolor}

\begin{document}

\maketitle

\begin{abstract}
Neural networks have been very successful in many applications; we often, however, lack a theoretical understanding of what the neural networks are actually learning.
This problem emerges when trying to generalise to new data sets.
The contribution of this paper is to show that, for the residual neural network model, the deep layer limit coincides with a parameter estimation problem for a nonlinear ordinary differential equation.
In particular, whilst it is known that the residual neural network model is a discretisation of an ordinary differential equation, we show convergence in a variational sense.
This implies that optimal parameters converge in the deep layer limit.
This is a stronger statement than saying for a fixed parameter the residual neural network model converges (the latter does not in general imply the former).
Our variational analysis provides a discrete-to-continuum $\Gamma$-convergence result for the objective function of the residual neural network training step to a variational problem constrained by a system of ordinary differential equations; this rigorously connects the discrete setting to a continuum problem.
\end{abstract}

\noindent
\keywords{deep neural networks, ordinary differential equations, deep layer limits, variational convergence, Gamma-convergence, regularity}

\noindent
\subjclass{34E05, 39A30, 39A60, 49J45, 49J15}
% 	39A60  	Difference and functional equations 	Difference equations 	Applications
% 	39A30  	Difference and functional equations 	Difference equations	Stability theory
% 	34E05  	Ordinary Differential Equations 		Asymptotic theory		Asymptotic expansions
%	34E99  	Ordinary Differential Equations			Asymptotic theory		None of the above, but in this section
%	49J45  	Calculus of variations and optimal control; optimization [See also 34H05, 34K35, 65Kxx, 90Cxx, 93-XX] 	Existence theories		Methods involving semicontinuity and convergence; relaxation
% 	49J15  	Calculus of variations and optimal control; optimization [See also 34H05, 34K35, 65Kxx, 90Cxx, 93-XX] 	Existence theories		Optimal control problems involving ordinary differential equations
% 	49J55  	Calculus of variations and optimal control; optimization [See also 34H05, 34K35, 65Kxx, 90Cxx, 93-XX] 	Existence theories		Problems involving randomness [See also 93E20]

% 	68R10  	Discrete mathematics in relation to computer science: Graph theory
% 	62G20  	Nonparametric inference: Asymptotic properties
%   60D05   Geometric probability and stochastic geometry
%  	35J20   Variational methods for second-order elliptic equations
%   65C60   Computational problems in statistics
% 	65N12   Stability and convergence of numerical methods

\section{Introduction \label{sec:Intro}}

Recent advances in neural networks have proven immensely successful for classification and imaging tasks \cite{schmidhuber2015deep}. These practical sucessess have inspired many theoretical studies that try to understand why certain network architectures work better than others and what role the various parameters of the networks play. Over the years, these studies have come from such diverse areas as computational science \cite{orponen1992neural,cybenko1996neural,siegelmann2012neural}, discrete mathematics \cite{anthony2001discrete}, control theory and dynamical systems \cite{narendra1990identification,hunt2012neural,haber17,E2017}, approximation theory \cite{cybenko1989approximation,hornik1989multilayer,hornik1991approximation}, frame theory \cite{wiatowski2018mathematical}, and statistical consistency~\cite{oberman2018lipschitz}.
To the best of our knowledge this and~\cite{oberman2018lipschitz} are the only papers to study variational limits of neural networks.

Stemming from the work of Haber and Ruthotto~\cite{haber17} and E~\cite{E2017} there has been recent interest in interpreting neural networks as dynamical systems.
The connection with dynamical systems follows from an idealised infinitely deep interpretation of a neural network where one treats the depth as a time variable.
There is a well-developed theory, such as Hamilton--Jacobi--Bellman equations and Pontryagin's maximum principle, which can be applied to analyse the dynamical system and therefore clarify the behaviour of the discrete neural network~\cite{E2019}.
Many more results have recently appeared in the literature, e.g.~\cite{Bo20,lu20,treister18,ruthotto18,lu17,chen18,grathwohl18AAA}, and we refer to~\cite{celledoni21} for a more detailed overview.
The aim of this paper is to connect discrete neural networks to a dynamical system using (a small modification) of the model presented in~\cite{haber17}.

Classification of data is the task of assigning each element of a data set a label which indicates membership of one of several classes. Each of those classes has some a priori data assigned to it. A neural network approaches this task in two steps. First the a priori classified data is used to \emph{train} the network. Then the trained network is used to classify other data. In this paper we will consider input data $x\in \R^d$ leading to network output $F(x) \in \R^m$ for some function $F: \R^d \to \R^m$. A neural network assigns a classification to some given input datum by performing a series of sequential operations to it, which are known as \emph{layers}. Each layer is said to consist of \emph{neurons}, by which it is meant that the output of each of the operations can be represented as a vector in $\R^d$ (encoding the state of $d$ neurons). In our paper we assume there are $n$ hidden layers\footnote{For simplicity we usually speak of $n$ layers, even though it can be argued there are $n+2$ layers if one includes the $0^{\text{th}}$ input layer and the final classification layer from \eqref{eq:classificationlayer} in the count.} and each layer has the same number, $d$, of neurons. (Note that by making this assumption, the networks we consider cannot be used for dimensionality reduction; the network makes the classification decision based on the final layer, which contains a number of neurons equal to the dimension of the input datum.) We also assume that each input datum can be represented by a vector of the same dimension $d$. Hence an input datum $x\in \R^d$ leads to a response in the first layer, $f_0(x)\in \R^d$, which in turn leads to a response in the second layer $f_1(f_0(x))\in \R^d$, etc. After the response of the final layer $f_{n-1}(f_{n-2}(\ldots f_0(x)\ldots))\in \R^d$ is obtained, a final function $\hat f: \R^d \to \R^m$ can be applied to map that response to the labels of the various pre-defined classes. The final output of the network then becomes $F(x):= \hat f(f_{n-1}(f_{n-2}(\ldots f_0(x)\ldots)))$.

In the training step training data $\{(x_s,y_s)\}_{s=1}^S$ is available, where $\{x_s\}_{s=1}^S\subset \bbR^d$ are inputs with class labels $\{y_s\}_{s=1}^S\subset\bbR^m$. The goal is to learn the form of the functions $f_i$ such that the network's classifictions $F(x_s)$ are close to the corresponding labels $y_s$. In this paper we restrict ourselves to functions $f_i$ from a parametrised family of functions, as described in \eqref{eq:ResNet} below. The choice of cost function which is used to measure this ``closeness'' is one of many choices whose conseqences are being studied, for example for classification \cite{janocha2016loss} and image restoration tasks \cite{zhao2017loss}.
In this paper we consider a cost function with mild conditions, which allow for, for example, a quadratic error term (or loss function) $\sum_{s=1}^S \|F(x_s)-y_s\|^2$, together with regularisation terms which we will discuss later.

%\B{Without further regularisation terms in the cost function, the optimization problem of finding functions $f_i: \R^d \to \R^d$ that minimize the loss function is extremely ill-posed: Many functions $f_i$ exist such that $F(x_s)=y_s$ for the finitely many training data pairs $(x_s, y_s)$. Moreover, there is a risk of overfitting due to the high (infinite) number of degrees of freedom available.
%That is why the set of functions $f_i$ that are being considered is being restricted, which can be accomplished directly by specifying a restricted admissible class of functions $f_i$ and indirectly by adding regularisation terms in the cost function. In this paper we opt for a combination of both. A good choice of admissible class is one that is not so large that the task of finding optimal functions is unfeasible or leads to overfitting, yet one that is large enough to contain the building blocks of suitable classifiers $F$. It should be noted that in general the class is not so small that uniqueness of the optimal functions is ensured.}
Implied in the architecture is the choice of parameterisation for $f_i$.
A typical choice is to let $f_i$ be of the form
%, which still leaves a very general class of functions, \B{parameterises} the functions $f_i$ \B{by}
\begin{equation}\label{eq:CNN}
f_i(x) = \sigma_i(K_i x + b_i),
\end{equation}
where $K_i \in \R^{d\times d}$ is a matrix which determines the \emph{weights} with which neurons in layer $i$ activate neurons in layer $i+1$ and $b_i\in \R^d$ is a \emph{bias} vector. The functions $\sigma_i$ are called the \emph{activation functions}. Many, although not all, activation functions used in practice are continuous approximations of a step function that effectively turn neurons ``on'' or ``off'' depending on the value of the input $K_i x + b_i$. In this paper, we assume every layer uses the same (Lipschitz continuous) activation function, $\sigma_i=\sigma$. Results from recent years have shown that the \emph{rectified linear unit} (ReLU) activation function (or ``positive part'' \cite{jarrett2009best}) performs well in many situations  \cite{nair2010rectified,krizhevsky2012imagenet,dahl2013improving}. It is given by
\begin{equation}\label{eq:ReLU}
\sigma(x) = \begin{cases}
0, &\text{if } x<0,\\
x, &\text{if } x\geq 0,
\end{cases}
\end{equation}
where its action on a vector should be interpreted componentwise (see Subsection~\ref{subsec:Intro:Finite} for details). This, however, is not the only choice that can be made. The impact of the activation function on the performance of a given network is studied in many papers. For example, if ReLU is used the network trains faster than when some of the classical  saturating nonlinear activation functions such as $x \mapsto \tanh x$ and $x \mapsto \frac1{1+e^{-x}}$ are used instead \cite{krizhevsky2012imagenet}. Moreover, ReLU has been observed to lead to sparsity in the resulting weights, with many of them being zero. These are sometimes referred to as ``dead neurons''  \cite{maas2013rectifier,zeiler2013rectified,vidal2017mathematics}.

The activation function(s) are often\footnote{But not always; see for example \cite{he15} for parametric ReLU, which contains a learnable parameter.} specified beforehand for a given network and are not a part of what should be ``learned'' by the network. That still leaves, however, a large number of parameters for the learning problem. Each layer contains $d\times d + d$  parameters in the form of $K_i$ and $b_i$. Different types of networks restrict the admissible sets for the $K_i$ and $b_i$. For example, some networks impose that the biases $b_i$ are completely absent, such as the \emph{Finite Impulse Response} (FIR) networks in  \cite{robinson1987utility,williams1989,hochreiter1991untersuchungen,wan1993finite} or that each layer has the same \emph{shared bias}~\cite{nielsen2015}, while the traditional \emph{convolutional neural networks} (CNN) restrict the choice of $K_i$ to convolution matrices, i.e. matrices in which each row is a shifted version of a \emph{filter} vector $(0, \ldots, 0, v_1, \ldots, v_k, 0, \ldots, 0)$, such that the product $K_i x$ becomes a discrete convolution of the vector $v=(v_1, \ldots, v_k)$ with $x$ \cite{kuo2016understanding,higham2018deep}. In this paper we will not restrict the choice of $K_i$ and $b_i$ by such hard constraints. Instead, we include regularisation terms in the cost function, which penalise $K_i$ and $b_i$ which vary too much between layers or whose entries in the first layer are too large (see Section~\ref{subsec:Intro:Reg} for details).

Oberman and Calder~\cite{oberman2018lipschitz} study, in a variational sense, the data rich limit $S\to \infty$.
In particular, they consider, a sequence of variational problems of the form
\begin{equation} \label{eq:Intro:LargeData}
\text{minimise: } L(F,\mu_S) + R(F),
\end{equation}
where $L$ is a loss term, $\mu_S$ is an empirical measure induced by the training data set $\{x_s,y_s\}_{s=1}^S$, and $R$ a regularisation term; for example,
\[ L(F,\mu_S) = \int_{\R^d\times \R^m} |F(x)-y|^2 \, \dd \mu_S(x,y) =  \frac{1}{S}\sum_{s=1}^S |F(x_s)-y_s|^2. \]
The set of admissible $F$ is determined by a neural network.
The main result of~\cite{oberman2018lipschitz} is to show that minimisers $F_S$ of~\eqref{eq:Intro:LargeData} converge as $S\to \infty$ to a solution of the variational problem
\[ \text{minimise: } L(F,\mu) + R(F) \]
for an appropriate measure $\mu$ obtained as limit of the empirical measures $\mu_S$.

In this paper we study the deep layer limit (i.e. the limit $n\to \infty$) of a \emph{residual neural network} (ResNet)~\cite{He16}, which are related in spirit to the highway networks of \cite{srivastava2015highway}. A crucial way in which ResNet type neural networks differ from other networks such as CNNs, is the form of the functions $f_i$. Instead of assuming a form as in \eqref{eq:CNN}, in ResNet the assumption
\begin{equation}\label{eq:ResNet}
f_i(x) = x + \sigma_i(K_i x + b_i)
\end{equation}
is made. This can be interpreted as the network having \emph{shortcut connections}: The additional term $x$ on the right-hand side represents information from the previous layer ``skipping a layer'' (or, more accurately, skipping the processing associated with the layer) and being transmitted to the next layer without being transformed. The reason for introducing these shortcut connections is to tackle the \emph{degradation} problem \cite{he2015convolutional,He16}: It has been observed that increasing the depth of a network (i.e. its number of layers) can lead to an increase in the error term instead of the expected decrease. Crucially, this behaviour appears while training the network, which indicates that it is not due to overfitting (as that would be an error which would be only present during the testing phase of an already trained network). In \cite{He16} it is argued that, if $\hat f_i(x)$ is the actual desired output in layer $i+1$, the \emph{residual} $\hat f_i(x)-x$ is easier to learn in practice than $\hat f_i(x)$ itself.
Deep networks using the architecture~\eqref{eq:CNN} can suffer from vanishing or exploding gradients during backpropagation \cite{hochreiter1991untersuchungen,bengio1994learning,glorot2010understanding,nielsen2015}, resulting in weights which either do not change much at all during the training phase or which change wildly in each step.
In general, learning the residual does not suffer from vanishing/exploding gradients by approximately preserving the norm
of the gradient between layers~\cite{zaeemzadeh2020norm}.
In \cite{glorot2010understanding} it is shown that these problems might be avoided by choosing a careful initialisation; \cite{maas2013rectifier} argues that using the ReLU activation function also helps in avoiding vanishing gradients.

Crucially for our purposes, the additional term $x$ in \eqref{eq:ResNet} compared to \eqref{eq:CNN} allows us to write
\begin{equation}\label{eq:ResNetEuler}
X_{i+1}^{(n)} - X_i^{(n)} = f_i(X_i^{(n)}) - X_i^{(n)} = \frac1n \sigma_i(K^{(n)}_i X_i^{(n)} + b^{(n)}_i),
\end{equation}
where $X_{i+1}^{(n)}=f_i(X_i^{(n)}) \in \R^d$ is the output in layer $i+1$ and where we have introduced a factor $\frac1n$ with $\sigma$ for scaling purposes. We have also added superscripts $(n)$ to $X_i^{(n)}$, $K^{(n)}_i$ and $b^{(n)}_i$ to indicate that these weights and biases belong to the network with $n$ layers. Remember that, in this paper, we will use the same activation function in each layer: $\sigma_i = \sigma$.
As observed in~\cite{haber2017learning,E2017,lu17}, this setup describes an explicit Euler characterisation of the ordinary differential equation (ODE)
\[ \dot{X}(t) = \sigma(K(t)X(t)+b(t)), \]
with time step $1/n$.
Here $X$, $K$, and $b$ denote real-valued functions on $[0,1]$.
This observation has been used to motivate new neural network architectures based on discretisations of partial/ordinary differential equations, e.g.~\cite{haber17,treister18,ruthotto18,lu17,chen18,grathwohl18AAA}.

Since the \emph{forward pass} through ResNet is given by a discretised ODE in \eqref{eq:ResNetEuler}, a natural question is whether the deep limit ($n\to \infty)$ of ResNet indeed gives us back the ODE.
We need to be a bit more careful, however, when formulating this question, and distinguish between the training step and the use of a trained network.
The latter consists of applying \eqref{eq:ResNetEuler} through all layers (with known $K^{(n)}_i$ and $b^{(n)}_i$ obtained by training the network) with a single given input datum $x$ as initial condition, $X_0=x$.
The deep limit question in this case then becomes whether solutions of this discretised process converge to the solution (Lipschitz continuity of $x\mapsto \sigma(Kx+b)$ guarantees a unique solution, by standard ODE theory) of the ODE.
Our Corollary~\ref{cor:convergenceforwardpass} shows that they do, in a pointwise sense.
In order to derive this corollary
we require the trained weights and biases, $K^{(n)}_i$ and $b^{(n)}_i$, to converge (up to a subsequence) to sufficiently regular weights and biases, $K$ and $b$, which can be used in the ODE.
This requires us to carefully analyse the training step.
The main result of this paper, Theorem~\ref{thm:MainRes:Conv}, does exactly that.

Theorem~\ref{thm:MainRes:Conv} uses techniques from variational methods to show that the trained weights and biases have (up to a subsequence) deep layer limits.
In particular, it uses $\Gamma$-convergence, which is explained in further detail in Section~\ref{subsec:Prelim:Gamma}.
Variational calculus deals with problems which can be formulated in terms of minimisation problems.
In this paper we formulate the training step (or \emph{learning problem}) of an $n$-layer ResNet as a minimisation problem for the function $\cE_n$ in \eqref{eq:learning}, which consists of a quadratic cost function with regularisers for all the coefficients that are to be learned.
We then identify the $\Gamma$-limit of the sequence $\{\cE_n\}_{n=1}^\infty$, which is given by $\cE_\infty$ in \eqref{eq:GammalimitEinfty}.
$\Gamma$-convergence is a type of convergence which (in combination with a compactness result) guarentees that minimisers of $\cE_n$ converge (up to a subsequence) to a minimiser of the $\Gamma$-limit $\cE_\infty$. It has been successfully applied for discrete-to-continuum limits in a machine learning setting, for example in \cite{vanGennipBertozzi12} and the references in the following sentence. The specific tools we use in this paper to obtain the discrete-to-continuum $\Gamma$-limit were developed in \cite{garciatrillos16} and have been succesfully applied in a series of papers since \cite{garciatrillos2016consistency,garciatrillos2020,garciatrillos2018variational,slepcev19,dunlop2018large}.

The impact of this $\Gamma$-convergence result is twofold. On the one hand it is an important ingredient in showing that the output of an already trained network for given input data is, in the sense made precise by Corollary~\ref{cor:convergenceforwardpass},
approximately the output of a dynamical system which has the input data as initial condition. On the other hand, it shows that the training step itself is a discrete approximation of a continuum variational problem. This opens up the possibility of using techniques from partial differential equations (PDEs) to solve the minimisation problem for $\cE_\infty$ in order to obtain (approximate) solutions to the $n$-layer training step; such as Pontryagin's maximum principle~\cite{li18,e19}.
It also opens up the possibility to construct different networks by using different discretisations of the ODE, as in the midpoint network in \cite{chang2017reversible}.

We note that connecting discrete difference equations to a continuum differential equation in the setting of recursive algorithms (i.e. $X_{i+1} = f_i(X_i,\theta)$ where $\theta$ are given parameters) is well studied, for example~\cite{ljung77,chung54}.
However, these results are in the pointwise convergence setting, i.e. the parameter $\theta$ is fixed.
Pointwise convergence is not strong enough to imply convergence of minimisers, i.e. what we want is that the $\theta^*_n$ that minimises a variational problem converges as $n\to\infty$ to some $\theta$ that minimises a variational problem with the constraint $\dot{X} = f_\infty(X,\theta)$.
This is the novelty of our result.
\vspace{\baselineskip}

In the remainder of the introduction we introduce our framework; namely the neural network architecture, the choice and motivation of regularisation of the neural network parameters, and the continuum deep layer limit.
In Section~\ref{sec:MainRes} we state our assumptions and main results connecting the discrete neural network with its continuum limit.
In Section~\ref{sec:Prelim} we give some preliminary material which includes (1) defining the topology we use for convergence of the parameters $\bfK^{(n)}$, $\bfb^{(n)}$, i.e. we make precise $\bfK^{(n)}\to K$ and $\bfb^{(n)}\to b$, and (2) giving a brief background on variational methods and in particular $\Gamma$-convergence.
Section~\ref{sec:ProofConv} is devoted to the proofs of the main results.
We conclude the paper in Section~\ref{sec:Conc} with a brief discussion of open questions.

\subsection{The Finite Layer Neural Network} \label{subsec:Intro:Finite}

We recap a simplified version of ResNet as presented in~\cite{haber17}.
In this model there are $n$ layers and the number of neurons in each layer is $d$.
In particular, we let $X_i^{(n)}\in \bbR^d$ be the state of each neuron in the $i^{\text{th}}$ layer.
For clarity we will denote with a superscript the number of layers, this is to avoid confusion when talking about two versions of the neural network with different numbers of layers.
The relationship between layers is given by
\begin{equation} \label{eq:Intro:ResNet:Rec}
X_{i+1}^{(n)} = X_i^{(n)} + \frac{1}{n}\sigma(K_i^{(n)}X_i^{(n)}+b_i^{(n)}), \quad \quad i = 0,1,\dots, n-1,
\end{equation}
where $\bfK^{(n)} = \{K^{(n)}_i\}_{i=0}^{n-1}\subset \bbR^{d\times d}$, $\bfb^{(n)} = \{b^{(n)}_i\}_{i=0}^{n-1}\subset \bbR^d$ determine an affine transformation at each layer and $\sigma:\bbR^d\to \bbR^d$ is an activation function which characterises the difference between layers.
We will assume that $\sigma$ acts componentwise, i.e. $\sigma(x) = (\tilde\sigma(x_1), \tilde\sigma(x_2), \dots, \tilde\sigma(x_d))^T$, for some $\tilde\sigma: \R \to \R$. For example, a valid, but not necessary choice for $\tilde \sigma$ is the ReLU function from \eqref{eq:ReLU}. With a slight abuse of notation, $\tilde \sigma$ is sometimes also denoted by $\sigma$, as for example in \eqref{eq:ReLU}.
The layers $\{X_i^{(n)}\}_{i=1}^{n-1}$ are called hidden, $X_0^{(n)}$ is the input to the network, and $X_n^{(n)}$ is the output.

In order to apply the neural network~\eqref{eq:Intro:ResNet:Rec} to labelling problems an additional, classification, layer is appended to the network.
For example, one can add a linear regression model, that is we let $Y = WX_n^{(n)}+c$ where $W\in\bbR^{m\times d}$ and $c\in \bbR^m$.
More generally, we assume the classification layer takes the form
\begin{equation}\label{eq:classificationlayer}
Y = h(WX_n^{(n)}+c)
\end{equation}
for a given function $h:\bbR^m\to \bbR^m$.
Given all parameters, the forward model/classifier for input $X_0^{(n)}=x$ is $Y=h(WX_n^{(n)}[x;\bfK^{(n)},\bfb^{(n)}]+c)$ where $X_n^{(n)}[x;\bfK^{(n)},\bfb^{(n)}]$ is given by the recursive formula~\eqref{eq:Intro:ResNet:Rec} with input $X_0^{(n)}=x$.

Given a set of training data $\{(x_s,y_s)\}_{s=1}^S$, where $\{x_s\}_{s=1}^S\subset \bbR^d$ are inputs with labels $\{y_s\}_{s=1}^S\subset\bbR^m$, one wishes to find parameters $\bfK^{(n)}$, $\bfb^{(n)}$, $W$, $c$ that minimise the error of the neural network on the training data.
There are clearly multiple ways to measure the error.
To maximise generality we define
\[ E_n(\bfK^{(n)},\bfb^{(n)},W,c;x,y) = \cL\l h(WX_n^{(n)}[x;\bfK^{(n)},\bfb^{(n)}] + c), y \r, \]
where the function $\cL$ is nonnegative and has to satisfy a continuity condition in its first argument, as detailed in Theorem~\ref{thm:MainRes:Conv} and Proposition~\ref{prop:MainRes:Reg}. A typical allowed choice is $\cL(z,y) = \|z-y\|^2$.
The error $E_n(\bfK^{(n)},\bfb^{(n)},W,c;x,y)$ should be interpreted as the error of the parameters $\bfK^{(n)}$, $\bfb^{(n)}$, $W$, $c$ when predicting $x$ given that the true value is $y$.
Naively, one may wish to minimise the sum of $E_n(\bfK^{(n)},\bfb^{(n)},W,c;x_s,y_s)$ over $s\in\{1,\dots, S\}$.
However this problem is ill-posed once the number of layers, $n$, is large.
In particular, the number of parameters being greater than the number of training data points leading to overfitting.
The solution, as is common in the calculus of variations, is to include regularisation terms, e.g. (applicable to neural networks)~\cite{haber17,goodfellow09,ng04,ranzato08}, on each of $\bfK^{(n)}$, $\bfb^{(n)}$, $W$ and $c$, this is discussed in the next section.

The finite layer objective functional, with regularisation weights $\alpha_1,\dots,\alpha_4$, is given by
\begin{equation} \label{eq:learning}
\begin{aligned}
\cE_n(\bfK^{(n)},\bfb^{(n)},W,c) & = \sum_{s=1}^S E_n(\bfK^{(n)},\bfb^{(n)},W,c;x_s,y_s) + \alpha_1 R^{(1)}_n(\bfK^{(n)}) + \alpha_2 R^{(2)}_n(\bfb^{(n)}) \\
 & \quad \quad \quad \quad + \alpha_3 R^{(3)}(W) + \alpha_4 R^{(4)}(c).
\end{aligned}
\end{equation}

Here the $R^{(i)}$ are regularisation terms, which will be introduced in detail in Subsection~\ref{subsec:Intro:Reg}. The learning problem is to find $(\bfK^{(n)},\bfb^{(n)},W^{(n)},c^{(n)})$ which minimizes $\cE_n$.

The problem we concern ourselves with is the behaviour in the deep layer limit, i.e. what happens to $\bfK^{(n)},\bfb^{(n)},W^{(n)},c^{(n)}$ as $n\to \infty$.
The results of this paper are theoretical and in particular ignore the considerable challenge of finding such minimisers.
However, we do hope that a better understanding of the deep layer limit can aid the development of numerical methods by, for example, allowing PDE approaches to the minimization of $\cE_n$.
Indeed, the authors of~\cite{kovachki18} view neural networks as inverse problems and apply filtering methods such as the ensemble Kalman filter which are gradient free.
We note that theory is often developed for continuum models as it reveals what behaviour will be expected for large discrete problems.
For example, the authors of~\cite{zhang18} analyse stability properties of continuum analogues of neural networks.

In this paper we are not concerned with the actual numerical method used to compute the learning or training step, i.e. the method to compute minimizers of~\eqref{eq:learning}.
However, for completeness we briefly point to some optimisation methods and potential pitfalls.
Currently a variety of different methods are being used to compute the training step; \cite{vidal2017mathematics} gives an overview of various methods.
One of the most popular ones is backpropagation~\cite{hertz1991introduction,zurada1992introduction,kung1993,hassoun1995fundamentals,haykin1999neural} using stochastic gradient descent~\cite{higham2018deep}.
Since the minimization problem is not convex, any gradient descent method risks running into critical points which are not minima.
In~\cite{dauphin2014} it is argued that in certain setups critical points are more likely to be saddle points than local minima and~\cite{lee2016gradient} proves that (under some assumptions on the objective function and the step size) gradient descent does not converge to a saddle point for almost all initial conditions.
Moreover,~\cite{choromanska2015loss} empirically verifies that in deep networks most local minima are close in value to the global minimum and the corresponding minimizers give good results.
In some cases it can even be proven that all local minima are equal to the global minimum~\cite{laurent2018deep}.
These results suggests that the critical points of the non-convex optimization problem are not necessarily a major problem for gradient descent methods.

Variants of gradient descent, such as blended coarse gradient descent, which is not strictly speaking a gradient descent algorithm --- rather it chooses an artificial ascent direction --- have been explored in~\cite{yin18}.
The authors of~\cite{chaudari18} show that the (local entropy) loss function satisfies a Hamilton-Jacobi equation and use this to analyse and develop stochastic gradient descent methods (in continuous time) which converge to gradient descent in the limit of fast dynamics.
Outside of gradient based methods the authors of~\cite{haber18} apply an Ensemble Kalman Filter method to the training of parameters.

Overfitting is also an issue to take into account during training. Techniques such as max pooling \cite{schmidhuber2015deep} (for a PDE-based interpretation of max pooling and ReLU as morphological convolutions in a CNN, see \cite{Smets20}) or Dropout \cite{hinton2012improving} work well in practice to avoid overfitting. The former consists of downsampling a layer by pooling the neurons into groups and assigning to each group the maximum value of all its neurons. The latter consists of randomly omitting neurons on each presentation of each training case. The ReLU activation function works well with Dropout \cite{dahl2013improving}. Recently \cite{mocanu2018scalable} made the case that improvements can be obtained by using sparsely connected layers. Adding regularization terms which encourage some level of smoothness to the cost functional can also help to avoid overfitting\cite{higham2018deep}.

%Another problem that can be encountered during the learning phase of deep networks is that of vanishing or exploding gradients during backpropagation \cite{hochreiter1991untersuchungen,bengio1994learning,glorot2010understanding,nielsen2015}, which results in weights which either do not change much at all during the training phase or which change wildly in each step.  In \cite{glorot2010understanding} it is shown that these problems might be avoided by chosing a careful initialisation; \cite{maas2013rectifier} argues that using the ReLU activation function also helps in avoiding vanishing gradients.

\subsection{Regularisation \label{subsec:Intro:Reg}}

Explicit regularisation in neural networks dates back to at least~\cite{drucker92}, where the authors added a penalty on the rate of change of $E_n$ with respect to the input $x_s$.
Here we approximately follow~\cite{haber17}.
We refer to~\cite{oberman19} for a more in-depth discussion on regularisation in machine learning.

We define regularisation terms $R^{(1)}_n(\bfK^{(n)})$, $R^{(2)}_n(\bfb^{(n)})$, $R^{(3)}(W)$, $R^{(4)}(c)$ by
\begin{align*}
R^{(1)}_n(\bfK^{(n)}) & = n \sum_{i=1}^{n-1} \|K_i^{(n)}-K_{i-1}^{(n)}\|^2 + \tau_1\|K_0^{(n)}\|^2, \\
R^{(2)}_n(\bfb^{(n)}) & = n \sum_{i=1}^{n-1} \|b_i^{(n)}-b_{i-1}^{(n)}\|^2 + \tau_2\|b_0^{(n)}\|^2, \\
R^{(3)}(W) & = \|W\|^2, \\
R^{(4)}(c) & = \|c\|^2
\end{align*}
where $\tau_i>0$. Since all norms on fininite dimensional vector spaces are topologically equivalent, many of the results in this paper do not depend on the specific choices for the norms $\|\cdot\|$ that are used in the definitions above. In some places, such as in Section~\ref{subsec:ProofConv:Reg} however, an inner product structure is assumed on certain norms, while in others, such as Lemma~\ref{lem:convergenceforwardpass} and the derived Corollary~\ref{cor:convergenceforwardpass}, the constants in the estimates will depend on the specific choice of norm.

%\BN
Since we eventually wish to interpret the $n$ layers of the network as a discretisation of (one-dimensional) time with time step $\frac1n$, the scaling by $n$ of the difference terms in $R_n^{(1)}$ and $R_n^{(2)}$ is the correct one to view these terms as discretised integrals of finite-difference approximations to squared gradients. This will be further clarified in Section~\ref{subsubsec:Intro:Reg:NonPar} and follows as a consequence of the approximation
\begin{equation} \label{eq:Intro:Reg:NonPar:DisApprx}
\| \dot{K} \|_{L^2}^2 \approx \frac{1}{n} \sum_{i=0}^{n-1} \|\dot{K}(i/n)\|^2 \approx \frac{1}{n}\sum_{i=1}^{n-1} \lda\frac{K(i/n)-K((i-1)/n)}{1/n}\rda^2.
\end{equation}
%\EN

We emphasise that $R^{(3)}$ and $R^{(4)}$ do not depend on $n$.
We refer to $R^{(1)}_n,R^{(2)}_n$ as the non-parametric regularisers, and $R^{(3)},R^{(4)}$ as the parametric regularisers (we consider $\bfK^{(n)}$ and $\bfb^{(n)}$ to be non-parametric as their complexity grows with the number of layers $n$, whilst $W$ and $c$ are parametric as their complexity is independent of $n$).
The point of including regularisation is to enforce compactness in the minimizers; without compactness we cannot find converging sequences of minimisers which, in particular, can lead to objective functionals that become ill-posed in the deep layer limit.
We justify the regularisation below, however we note that the regularisation is quite strong.
In particular, we are imposing $H^1$ bounds on $\bfK^{(n)}$ and $\bfb^{(n)}$ ---which are also suggested in \cite{haber17}--- as well as norm bounds on $W$ and $c$.
The cost of treating a wide range of activation functions $\sigma$ and classification functions $h$ is to include strong regularisation functions.
In specific cases it may be possible to reduce the regularisation, for example by setting $\tau_i=0$ and/or removing the terms $R^{(3)},R^{(4)}$.
In the next two subsections, we give a discussion on why these terms, in general, are necessary.

It should also be noted that it is sometimes observed that techniques such as stochastic node or layer dropout \cite{huang16} can act as regularizers, without the need for explicitly added regularization terms. A good mathematical understanding of this phenomenon is still missing from the literature, to the current knowledge of the authors, and in this paper we have restricted ourselves to adding explicit regularization terms, as is common in the calculus of variations.

\subsubsection{The Non-Parametric Regularisation \label{subsubsec:Intro:Reg:NonPar}}

By construction the regularisation terms on $\bfK^{(n)}$ and $\bfb^{(n)}$ resemble $H^1$ norms. These terms are used for compactness in order to apply the direct method of the calculus of variations. By standard Sobolev embeddings sequences bounded in $H^1$ are (pre-)compact in $L^2$. There is a little work to be done in order to match discrete sequences $\bfK^{(n)}=\{K^{(n)}_j\}_{j=0}^{n-1}$, $\bfb^{(n)}=\{b^{(n)}_j\}_{j=0}^{n-1}$ with continuum sequences $K^{(n)}:[0,1]\to \bbR^{d\times d}$, $b^{(n)}:[0,1]\to \bbR^d$, but with an appropriate identification we can show that $R^{(1)}_n(\bfK^{(n)}) \approx \|\dot{K}^{(n)}\|_{L^2}^2 + \tau_1\|K^{(n)}(0) \|^2$ and similarly for $\bfb^{(n)}$.  In that sense they are very natural choices from a calculus of variations point of view as they allow us to conclude strong $L^2$ convergence of the parameters.

Of course, given $K:[0,1]\to \bbR^{d\times d}$ we can define $\tilde{K}^{(n)}_i=K(i/n)$ and then $R^{(1)}_n(\tilde{\bfK}^{(n)}) \to \| \dot{K}\|_{L^2}^2 + \|K(0)\|^2$.
This we would call pointwise convergence.
The main result of this paper is stronger, in particular we show variational convergence. Without the $H^1$ semi-norm part of our regularization terms (i.e. without the $L^2$ norm of the gradient) we would a priori only get weak $L^2$ convergence. The question whether this suffices to still derive our results is a very interesting one, but goes beyond the scope of this paper, as it introduces a lot of extra technical difficulties.
We note that $R^{(i)}_n$, $i=1,2$, are very similar to the choice of regularisation in~\cite{haber17}, but we add the terms $\|K_0^{(n)}\|^2$, $\|b_0^{(n)}\|^2$.

The penalty on finite differences is natural; in order to achieve a limit it is necessary to bound oscillations in the parameters between layers.
Physically this relates to imposing the condition that close layers discriminate similar features.
For our analysis this is needed to establish compactness.
It is interesting to note that one can obtain limits without including explicit regularisation terms. The limiting behaviour of the deep network, however, may no longer be given by a deterministic ODE system of the type we will describe in \eqref{eq:Intro:Limit:ODE} and, in particular, could be stochastic~\cite{cohen21}. The coefficients appearing in the limiting equations described in~\cite{cohen21} are obtained through a different limiting procedure than the one we use (and describe in Section~\ref{subsec:Prelim:Top}).

The additional terms, $\|K_0^{(n)}\|^2$, $\|b_0^{(n)}\|^2$, are perhaps less physically reasonable and introduce a bias into the methodology (meaning that preference is given to smaller values of $K^{(n)}_0$ and $b^{(n)}_0$). As examples of why it is necessary to have these additional terms, i.e. to have $\tau_1>0$ and $\tau_2>0$, consider the following. First assume $\tau_1=0$, let $d=m=1$, $h=\Id$, $\sigma = \Id$, $\cL(z,y) = |z-y|^2$, and fix $n\in \N$.
Consider the set $\{(x_s,y_s)\}_{s=1}^S\subset \bbR\times \bbR$, where $y_s=x_s$, and the sequence $\{(\bfK^{(n)}_l, \bfb^{(n)}_l, W_l, c_l)\}_{l\in \N}$, with, for all $i$ and $l$, $(K_i^{(n)})_l = l$, $(b_i^{(n)})_l = 0$, $W_l = (1+l)^{-n}$, $c_l = 0$.
Then,
\[ E_n(\bfK^{(n)}_l,\bfb^{(n)}_l, W_l, c_l;x_s,y_s) = |W_l X^{(n)}_n - y_s|^2  = |(1+l)^{-n} (1+l)^n x_s - y_s|^2 = 0.
\]
Moreover, for all $l$,
$R_n^{(1)}(\bfK^{(n)}_l) = R_n^{(2)}(\bfb^{(n)}_l) = R^{(4)}(c_l) = 0$ and $R^{(3)}(W_l) \to 0$ as $l\to\infty$. Therefore $\{(\bfK^{(n)}_l, \bfb^{(n)}_l, W_l, c_l)\}_{l\in \N}$ is a minimising sequence for $\cE_n$ (as $l\to \infty$ with $n$ fixed) with no converging subsequence. As the elementary example shows, if one were to set $\tau_1=0$ then an additional assumption would be needed to guarantee relative compactness of minimizing sequences. A second example showing a similar necessity to have $\tau_2>0$ is constructed by setting $\tau_1\geq 0$, $\tau_2=0$, $(K_i^{(n)})_l = 0$, and $(b_i^{(n)})_l = l$ in the previous example.

\subsubsection{The Parametric Regularisation \label{subsubsec:Intro:Reg:Par}}

An example showing why $\alpha_3>0$ and $\alpha_4>0$ are necessary, can be constructed in a similar fashion.
Let $d=m=1$, $h=\Id$, $\sigma=\Id$, $\cL(z,y) = |z-y|^2$, and fix $n\in \N$. Let $S=1$, so that we have only one training pair $(x_1, y_1) = (x, y)$.
Define the sequence $\{(\bfK^{(n)}_l, \bfb^{(n)}_l, W_l, c_l)\}_{l\in \N}$, with, for all $i$ and $l$, $(K_i^{(n)})_l = 0$, $(b_i^{(n)})_l = 0$, $W_l = l$, $c_l = y-lx$.
Then,
\[ E_n(\bfK^{(n)}_l,\bfb^{(n)}_l, W_l, c_l;x,y) = |W_l X^{(n)}_n + c_l - y_s|^2  = |lx + y -lx - y|^2 = 0.
\]
Also, for all $l$,
$R_n^{(1)}(\bfK^{(n)}_l) = R_n^{(2)}(\bfb^{(n)}_l) = 0$.
We conclude, as before, that $\{(\bfK^{(n)}_l, \bfb^{(n)}_l, W_l, c_l)\}_{l\in \N}$ is a minimising sequence for $\cE_n$ (as $l\to \infty$ with $n$ fixed) with no converging subsequence.

\subsection{The Deep Layer Differential Equation Limit \label{subsec:Intro:Limit}}

By considering pointwise limits it is not difficult to derive our candidate limiting variational problem.
Although pointwise convergence is not enough to imply convergence of minimisers, it is informative.
Let $X:[0,1]\to \bbR^d$ solve the differential equation
\begin{equation} \label{eq:Intro:Limit:ODE}
\dot{X}(t) = \sigma(K(t)X(t) + b(t)), \quad \quad t\in [0,1]
\end{equation}
for some given parameters $K:[0,1]\to \bbR^{d\times d}$ and $b:[0,1]\to \bbR^d$ (as is usual we understand $\dot{X}(0)$ to be the right-derivative of $X$ at $t=0$ and $\dot{X}(1)$ to be the left-derivative of $X$ at $t=1$).
For shorthand we write $X(t;x,K,b)$ for the solutions of~\eqref{eq:Intro:Limit:ODE} with initial condition $X(0) = x$ and parameters $K$, $b$.
One can see that~\eqref{eq:Intro:ResNet:Rec} is the discrete analogue of~\eqref{eq:Intro:Limit:ODE} with $K_i^{(n)}=K(i/n)$ and $b_i^{(n)}=b(i/n)$.
In fact one can show (under sufficient conditions) that $X_{\lfloor nt\rfloor}^{(n)}[x,\bfK^{(n)},\bfb^{(n)}] \to X(t;x,K,b)$ as $n\to \infty$ (see Lemma~\ref{lem:ProofConv:GammaEn:ConvXn}).

Similarly, the regularisation terms $R^{(i)}_n$, $i=1,2$, are discretisations of the functionals
\begin{align}\label{eq:continuumregularisers}
R^{(1)}_\infty(K) & = \|\dot{K}\|_{L^2}^2 + \tau_1\|K(0)\|^2 \notag \\
R^{(2)}_\infty(b) & = \|\dot{b}\|_{L^2}^2 + \tau_2\|b(0)\|^2
\end{align}
and $R^{(3)}$, $R^{(4)}$ are unchanged.
We note that $R^{(i)}_\infty$, $i=1,2$ are well defined on $H^1$, since by regularity properties of Sobolev spaces any $u\in H^1$ is continuous and therefore pointwise evaluation is well defined; in particular we may define $\|K(0)\|$, $\|b(0)\|$ for $H^1$ functions.
In fact, see the discussion in Section~\ref{subsec:Prelim:Sob}, $R_\infty^{(i)}$, $i=1,2$, are equivalent to the $H^1$ norm whenever $\tau_i>0$.

Once we append the classification layer to the neural network we arrive at the limiting objective functional
\begin{equation}\label{eq:GammalimitEinfty}
\begin{aligned}
\cE_\infty(K,b,W,c) & = \sum_{s=1}^S E_\infty(K,b,W,c;x_s,y_s) + \alpha_1R^{(1)}_\infty(K) \\
 & \qquad + \alpha_2R^{(2)}_\infty(b) + \alpha_3 R^{(3)}(W) + \alpha_4 R^{(4)}(c)
\end{aligned}
\end{equation}
where
\[ E_\infty(K,b,W,c;x,y) = \cL\l h(W X(1;x,K,b)+c), y \r. \]
The main result of the paper is to show that minimisers of $\cE_n$ converge to minimisers of $\cE_\infty$.

%\subsection{Overview} \label{subsec:Intro:Overview}

%In the next section we state our assumptions and main results.
%In Section~\ref{sec:Prelim} we give some preliminary material which includes (1) defining the topology we use for convergence of the parameters $\bfK^{(n)}$, $\bfb^{(n)}$, i.e. we make precise $\bfK^{(n)}\to K$ and $\bfb^{(n)}\to b$, and (2) giving a brief background on variational methods and in particular $\Gamma$-convergence.
%Section~\ref{sec:ProofConv} is devoted to the proofs of the main results.
%We conclude the paper in Section~\ref{sec:Conc}.

\section{Main Results \label{sec:MainRes}}

Our main results concerns the convergence of the variational problem $\min \cE_n$ to $\min \cE_\infty$.
In particular we show
\begin{align*}
\min_{\bfK^{(n)},\bfb^{(n)},W,c} \cE_n(\bfK^{(n)},\bfb^{(n)},W,c) & \to \min_{K,b,W,c} \cE_\infty(K,b,W,c), \\
\argmin_{\bfK^{(n)},\bfb^{(n)},W,c} \cE_n(\bfK^{(n)},\bfb^{(n)},W,c) & \to \argmin_{K,b,W,c} \cE_\infty(K,b,W,c),
\end{align*}
as $n \to \infty$.
At this point we have not specified the topology on which we define the discrete-to-continuum convergence.
For now it is enough to say that the distance is given by a function $d:\Theta^{(n)}\times \Theta\to [0,+\infty)$ where $\Theta^{(n)}$ is the parameter space of $\cE_n$ and $\Theta$ is the parameter space of $\cE_\infty$.
The topology is described in detail in Section~\ref{subsec:Prelim:Top}. We do want to emphasize at this point that although $d$ appears to depend on $n$ ---and in fact appears not to be a distance at all, due to a lack of symmetry between its two arguments--- it is in fact a restriction of an $n$-independent metric on a higher-dimensional space to an $n$-dependent subset.

We will use the following assumptions for our results.

\begin{assumptions}
\label{ass:MainRes:Ass1}
The following assumptions will be used in our main convergence result:
\begin{enumerate}
\item $\alpha_i>0$ for $i=1,2,3,4$ and $\tau_j>0$ for $j=1,2$;
\item $h\in C^0(\bbR^m;\bbR^m)$;
\item $\sigma$ is Lipschitz continuous and acts componentwise;
\item $\sigma(0) = 0$;
\item $\cL\geq 0$ and is continuous in its first argument;
\end{enumerate}
%\item $\sigma\in C^2(\R^d; \R^d)$;
%\item $h\in C^2(\R^m; \R^m)$;
%\item $\cL(\cdot,y) \in C^2(\R^m; \R)$ for all $y\in \bbR^m$;
%\item all norms on $\bbR^d$ and $\bbR^{d\times d}$ are induced by %inner products.
%In particular 1-5 are used in Theorem~\ref{thm:MainRes:Conv} and Corollary~\ref{cor:convergenceforwardpass}, and 1-9 are used in Proposition~\ref{prop:MainRes:Reg}.
\end{assumptions}

We note that the condition on $\cL$ does not restrict it to be a typical loss function.
There is no requirement for $\cL$ to be a norm on the difference between its two arguments.

Our main result is the convergence of optimal parameters.

\begin{theorem}
\label{thm:MainRes:Conv}
Let $\Theta^{(n)}$ and $\Theta$ be given by~\eqref{eq:Prelim:Top:Thetan} and~\eqref{eq:Prelim:Top:Theta} respectively.
Define $\cE_n$, $\cE_\infty$, $E_n$, $E_\infty$, $R^{(i)}_n$, $R^{(i)}_\infty$, $R^{(j)}$ for $i=1,2$, $j=3,4$ as in Section~\ref{subsec:Intro:Finite}-\ref{subsec:Intro:Limit}.
Let Assumptions~\ref{ass:MainRes:Ass1} hold.
%Assume
%\begin{enumerate}
%\item $\alpha_i>0$ for $i=1,2,3,4$ and $\tau_j>0$ for $j=1,2$;
%\item $h$ is continuous;
%\item $\sigma$ is Lipschitz continuous and acts \B{componentwise};
%\item $\sigma(0) = 0$;
%\item $\cL\geq 0$ and is continuous in its first argument.
%\end{enumerate}
Let $\{(x_s,y_s)\}_{s=1}^S$ be any given set of training data ($S\geq 1$).
Then minimizers of $\cE_n$ and $\cE_\infty$ exist in $\Theta^{(n)}$ and $\Theta$ respectively.
Furthermore let $\theta^{(n)} \subset \Theta^{(n)}$ be any sequence of minimisers of $\cE_n$, then
\[ \min_{\Theta^{(n)}} \cE_n = \cE_n(\theta^{(n)}) \to \min_{\Theta} \cE_\infty, \qquad \text{as } n\to \infty,\]
$\{\theta^{(n)}\}_{n\in \bbN}$ is relatively compact, and any limit point of $\{\theta^{(n)}\}_{n\in \bbN}$ is a minimiser of $\cE_\infty$.
\end{theorem}
%We note that the conditions on $\cL$ in Theorem~\ref{thm:MainRes:Conv} do not restrict it to be a typical loss function. There is no requirement for $\cL$ to be a norm on the difference between its two arguments.

Through the following proposition, and under the additional and  stronger assumptions in Assumptions~\ref{ass:MainRes:Ass2}, we remark that one obtains extra regularity on minimisers to the deep limit variational problem (as is to be expected based on elliptic regularity).

\begin{assumptions}
\label{ass:MainRes:Ass2}
The following additional assumptions will be used in our regularity result:
\begin{enumerate}
\item $\sigma\in C^2(\R^d; \R^d)$;
\item $h\in C^2(\R^m; \R^m)$;
\item $\cL(\cdot,y) \in C^2(\R^m; \R)$ for all $y\in \bbR^m$;
\item all norms on $\bbR^d$ and $\bbR^{d\times d}$ are induced by inner products.
\end{enumerate}
\end{assumptions}

\begin{proposition}
\label{prop:MainRes:Reg}
%In addition to the assumptions of Theorem~\ref{thm:MainRes:Conv} we assume that $\sigma\in C^2(\R^d; \R^d)$, $h\in C^2(\R^m; \R^m)$, $\cL(\cdot,y) \in C^2(\R^m; \R)$ for all $y\in \bbR^m$ and all norms on $\bbR^d$ and $\bbR^{d\times d}$ are induced by inner products.
Let Assumptions~\ref{ass:MainRes:Ass1} and~\ref{ass:MainRes:Ass2} hold.
Then any minimiser $\theta=(K,b,W,c)\in \Theta$ of $\cE_\infty$ satisfies $K\in H^2_{\loc}([0,1];\bbR^{d\times d})$ and $b\in H^2_{\loc}([0,1];\bbR^d)$.
\end{proposition}

The proof of the proposition is given in Section~\ref{subsec:ProofConv:Reg}.

Theorem~\ref{thm:MainRes:Conv} states that, up to subsequences, minimisers of $\cE_n$ converge to minimizers of $\cE_\infty$.
If the minimizer of $\cE_\infty$ is unique then we have that the sequence of minimizers converges (without recourse to a subsequence) to the minimizer of $\cE_\infty$.
The proof of the theorem relies on variational methods and is given in Sections~\ref{subsec:ProofConv:Compact}-\ref{subsec:ProofConv:GammaEn}.
We do not prove a convergence rate for the minimizers, but we conjecture a convergence rate of $\frac{1}{n}$.
The conjecture is motivated by considering Taylor expansions for a fixed $\theta=(K,b,W,c)\in\Theta$; indeed one can show that for $K,b\in H^2$ the recovery sequence $\theta^{(n)}$ given by~(\ref{eq:ProofConv:GammaEn:Kn}-\ref{eq:ProofConv:GammaEn:c}) satisfies
\[ |\cE_n(\theta^{(n)}) - \cE_\infty(\theta) | \sim \frac{C(\theta)}{n}, \]
where $C(\theta)$ is a constant that depends on $\|\ddot{K}\|_{L^2}$ and $\|\ddot{b}\|_{L^2}$.
Assuming that this can be extended to minimizing sequences (i.e. the above holds for any sequence of minimizers $\theta^{(n)}\to \theta$) one can conclude that the rate of convergence of the minima is $O(n^{-1})$.
Making another conjecture that one can show a local bound of the form $d(\theta^{(n)},\theta) \leq C \la \cE_n(\theta^{(n)}) - \cE_\infty(\theta) \ra$ whenever $d(\theta^{(n)},\theta)$ is small implies
\[ d(\theta^{(n)},\theta) = O\l \frac{1}{n}\r. \]

We do not prove a rate of convergence for either the minimisers or the minimum here.
However, we are able to show a rate of convergence for the forward pass through the Neural Network; more precisely, the output of the ResNet model is converging to the output of a dynamical system with the rate given by the following corollary.

\begin{corollary}\label{cor:convergenceforwardpass}
Let Assumptions~\ref{ass:MainRes:Ass1} hold.
%Assume that the assumptions in Theorem~\ref{thm:MainRes:Conv} on $\alpha_i$, $\tau_j$, $h$, $\sigma$, and $\cL$ are all satisfied,
We use the matrix operator norm on $\bfK^{(n)}$ in $R^{(1)}_n$ and $R^{(1)}_\infty$, and let $L_\sigma>0$ be a Lipschitz constant for $\sigma$. Let $\{(x_s,y_s)\}_{s=1}^S \subset \R^d \times \R^m$ be a set of training data and, for all $n\in \N$, let $(\bfK^{(n)},\bfb^{(n)},W^{(n)},c^{(n)}) \in \argmin_{(\bfK^{(n)},\bfb^{(n)},W,c)}  \cE_n(\bfK^{(n)},\bfb^{(n)},W,c)$. Let $x\in \R^d$.
Let $(K,b,W,c)$ be the minimiser of $\cE_\infty$ which we assume is unique.
For all $n\in \N$ and for all $i\in \{1, \ldots, n\}$, let $X_i^{(n)}$ be the solution to \eqref{eq:Intro:ResNet:Rec} with $X_0^{(n)} = x$,
and $X:[0,1]\to \bbR^d$ be the solution to the ODE in \eqref{eq:Intro:Limit:ODE} (with coefficients $K$ and $b$) with initial condition $X(0)=x$.
Then, for all $\delta>0$, there exists an $N\in \N$ such that, for all $n\geq N$, there exists an $\eps_n \in \R$ such that, for all $i\in \{0, 1, \dots, n\}$,
\begin{equation}\label{eq:Xln-Xln}
\|X(i/n) - X_i^{(n)}\| \leq
 \frac{n}{L_\sigma (\|K\|_{L^\infty}+\delta)} A_n  \left[\exp\left(\frac{i}n L_\sigma (\|K\|_{L^\infty}+\delta)\right)-1\right],
\end{equation}
where
\[ A_n = \frac1n \left(1+ \|X\|_{L^\infty}\right) L_\sigma \delta + \eps_n.\]
Moreover, $\eps_n = o\left(\frac1n\right)$ as $n\to\infty$.
\end{corollary}
We provide the proof of Corollary~\ref{cor:convergenceforwardpass} in Section~\ref{sec:networkasdynamicalsystem}.
\begin{remark}
In Corollary~\ref{cor:convergenceforwardpass} we made the assumption that the minimizer $(K,b,W,c)$ of $\cE_\infty$ is unique, mainly to keep the notation as simple as possible. If minimizers of $\cE_\infty$ are not unique, we need to be more careful in our statement of the corollary.  In that case, by Theorem~\ref{thm:MainRes:Conv} there exists a minimizer $(K,b,W,c)$ of $\cE_\infty$ such that, {\it up to subsequences}, $K^{(n)}\to K$ and $b^{(n)} \to b$ as $n\to \infty$ in the topology of Section~\ref{subsec:Prelim:Top} (where the same indices can be chosen for both subsequences). Taking $\mathcal{N}$ to be the (infinite) set containing the (common) indices $n$ of these subsequences, the statement of the corollary still holds if we restrict the indices $n$ to the set $\mathcal{N}$ instead of allowing them to vary over all of $\N$.
\end{remark}

\section{Background Material \label{sec:Prelim}}

In this section we give background material necessary to present the proofs of the main results.
In particular, we start by clarifying our notation.
We then give a description on the discrete-to-continuum topology.
Finally, for the convenience of the reader, we give a brief overview on $\Gamma$-convergence.

\subsection{Notation \label{subsec:Prelim:Not}}

Let $\Omega$ be an open subset of a Euclidean space. Given a probability measure $\mu\in \cP(\Omega)$ on $\Omega$ we write $L^p(\mu;\Xi)$ for the set of functions from $\Omega$ to $\Xi$ that are $L^p$ integrable with respect to $\mu$, when appropriate we will shorten notation to $L^p(\mu)$.
The $L^p(\mu)$ norm for a function $f:\Omega\to \Xi$ is denoted by $\|f\|_{L^p(\mu)}$.
When $\mu$ is the Lebesgue measure on $\Omega$ we will also write $L^p$ or $L^p(\Omega)$ for $L^p(\mu;\Xi)$, and $\|f\|_{L^p}$ or $\|f\|_{L^p(\Omega)}$ for $\|f\|_{L^p(\mu)}$.
The $L^2$ inner product with respect to the Lebesgue measure is denoted by $\langle\cdot,\cdot,\rangle_{L^2}$.
The Sobolev space of functions that are $k$-times weakly differentiable and with each weak derivative in $L^2$ is denoted by $H^k$.
In order to make clear the domain $\Omega$ and range $\Xi$ of $H^k$ we will also write $H^k(\Omega;\Xi)$ (in order to avoid complications defining derivatives, the underlying measure in Sobolev spaces when $k>0$ is always the Lebesgue measure).
For functions $f:\Omega\to\bbR$ that are $k$ times continuously differentiable we write $f\in C^k(\Omega)$, and if all its $k$th partial derivatives are H\"older continuous with exponent $\gamma$ ---i.e. if $\max_{\alpha: |\alpha|=k} \sup_{\substack{x,y\in \Omega\\x\neq y}}\frac{|D^\alpha f(x)-D^\alpha f(y)|}{\|x-y\|^\gamma} < \infty$, where the $\alpha$ are multi-indices with sum $|\alpha|$ equal to $k$ and $D^\alpha f$ denotes the corresponding partial derivative of $f$ of order $k$---  we write $f\in C^{k,\gamma}(\Omega)$.

We often do not specify a matrix or vector norm, clearly these are finite dimensional spaces and therefore all norms are topologically equivalent.
If $b\in \bbR^d$ is a vector and $K\in \bbR^{d\times d}$ is a matrix then we will write $\|b\|$ and $\|K\|$ for both the vector norm and the matrix norm.
In particular we point out that we only use subscripts for $L^p$ norms.
Sometimes we will need that the norms are induced by inner products, we will write when we need this additional structure.

We use superscripts on the parameters $\bfK^{(n)}$ and $\bfb^{(n)}$ (later denoted $K^{(n)}$ and $b^{(n)}$) in order to clearly denote the dependence of the number of layers on the parameters themselves (this is particularly important as we take the limit $n\to \infty$).
The parameters $W,c$ are respectively a $m\times d$ matrix and a $m$-dimensional vector and therefore we do not include any reference to $n$ unless we are considering sequences.

Vectors are always column vectors.
For two vectors $A,B\in \bbR^\kappa$ we use $\odot$ to denote componentwise multiplication, i.e. $A\odot B = [A_1B_1, A_2B_2,\dots, A_\kappa B_\kappa]^\top$.
When $A\in\bbR^\kappa$ and $C\in \bbR^{\kappa\times d}$ then $\odot$ represents row-wise multiplication, i.e.
\[ A\odot C = C\odot A = \ls \begin{array}{cccc} A_1 C_{11} & A_1 C_{12} & \cdots & A_1 C_{1d} \\ A_2 C_{21} & A_2 C_{22} & \cdots & A_2 C_{2d} \\ \vdots & \vdots & \ddots & \vdots \\ A_\kappa C_{\kappa1} & A_\kappa C_{\kappa2} & \cdots & A_\kappa C_{\kappa d} \end{array} \rs. \]
We can also interpret this product as $A \odot C = \text{diag}(A) C$, where $\text{diag}$ is the diagonal $\kappa\times \kappa$ matrix with the vector $A$ on its diagonal.

We use the convention that $0\not\in\bbN$.

\subsection{Discrete-to-Continuum Topology \label{subsec:Prelim:Top}}

We introduced the parameters for the ResNet model with $n$ layers $\bfK^{(n)}$ and $\bfb^{(n)}$ as sets of matrices/vectors, i.e. $\bfK^{(n)}=\{K^{(n)}_i\}_{i=0}^{n-1}\subset \bbR^{d\times d}$ and $\bfb^{(n)}=\{b_i^{(n)}\}_{i=0}^{n-1}\subset \bbR^d$.
In fact it is more convenient to think of them as functions with respect to the discrete measure $\mu_n=\frac{1}{n}\sum_{i=0}^{n-1} \delta_{\frac{i}{n}}$ on $[0,1]$.
More precisely, for $\bfK^{(n)}$ we make the identification with $K^{(n)}\in L^0(\mu_n;\bbR^{d\times d})$ by $K^{(n)}(i/n) = K^{(n)}_i$.
In the sequel we will, with a small abuse of notation, write both $K^{(n)}$ and $K_i^{(n)}$, where the former is understood as a function in $L^0(\mu_n;\bbR^{d\times d})$ and the latter as the matrix $K^{(n)}_i = K^{(n)}(i/n)\in \bbR^{d\times d}$.
Similarly for $b^{(n)}$ and $b^{(n)}_i$.

With this notation we can define the finite layer parameter space by
\begin{equation} \label{eq:Prelim:Top:Thetan}
\Theta^{(n)} = L^2(\mu_n;\bbR^{d\times d}) \times L^2(\mu_n;\bbR^d) \times \bbR^{m\times d} \times \bbR^m.
\end{equation}
For any $p,q>0$ the discrete spaces $L^p(\mu_n), L^q(\mu_n)$ are topologically equivalent.
Since we apply a discrete analogue of an $H^1$ regularisation penalty to parameters of the neural network (see also Section~\ref{subsec:Prelim:Sob}), it will transpire that the natural limiting space to work in is given by
\begin{equation} \label{eq:Prelim:Top:Theta}
\Theta = H^1([0,1];\bbR^{d\times d}) \times H^1([0,1];\bbR^d) \times \bbR^{m\times d} \times \bbR^m.
\end{equation}

Given $K\in L^2([0,1];\bbR^{d\times d})$ and $K^{(n)}\in L^2(\mu_n;\bbR^{d\times d})$ we define a distance by extending $K^{(n)}$ to a function on $[0,1]$ by $\tilde{K}^{(n)}(t) = K^{(n)}(t_i)$ (for $t\in (t_{i-1},t_i)$, $t_i=i/n$, $i=1,\dots, n$) and comparing in $L^2$; that is
\[ d_1(K^{(n)},K) = \|\tilde{K}^{(n)} - K\|_{L^2}. \]
We note the distance $d_1$ is closely related to the $TL^2$ distance, see~\cite{garciatrillos16}, when the discrete measure is of the form $\mu_n = \frac{1}{n} \sum_{i=1}^n \delta_{t_i}$
and the domain is $[0,1]$.
The $TL^2$ distance (see~\eqref{eq:Prelim:Top:dTL2} below), or more generally the $TL^p$ distance, is a topology useful for metrising discrete-to-continuum convergence on a general domain $\Omega\subseteq\bbR^d$.
The idea is to think of functions on different domains as the coupling of a function with a measure; that is the $TL^p$ space is the space of pairs $(\mu,K)$ where $K\in L^p(\mu)$ and $\mu$ is a probability measure on $\Omega$ with finite $p$th moment.
The $TL^p$ space is metrised by a Wasserstein distance on the graphs of functions.
To compare a discrete function $K^{(n)}:\{x_i\}_{i=1}^n\to\bbR$ with associated discrete measure $\mu_n=\sum_{i=1}^n\delta_{x_i}$ to a continuum function $K:\Omega\to\bbR$ with the continuum measure $\mu\in\cP(\Omega)$ associated with $\Omega$, we perform the following steps.
Firstly, we  find an ``optimal'' partitioning of equal mass of the underlying state space, $T^{(n)}:\Omega\to\{x_i\}_{i=1}^n$ (where optimal is in the sense of solving an optimal transport problem between the discrete measure $\mu_n$ and the continuum measure $\mu$.).
Secondly, we  extend $K^{(n)}$ to $\Omega$ to be piecewise constant, i.e. $\tilde{K}^{(n)} = K^{(n)}\circ T^{(n)}:\Omega\to\bbR$.
Lastly we  compare $\tilde{K}^{(n)}$ to $K$ in an $L^p$ norm.
The notation $TL^p$ stands for ``transport'' and ``$L^p$''.
We leave further details of the topology that is constructed in this way to~\cite{garciatrillos16}.

To make the connection between $d_1$ and $TL^2$ precise consider the following.
For pairs $(\mu,K),(\nu,L)$ where $\mu,\nu\in\cP(\Omega)$ and $K\in L^2(\mu)$, $L\in L^2(\nu)$ the $TL^2$ distance is defined (in the Kantorovich formulation) by
\begin{equation} \label{eq:Prelim:Top:dTL2}
d_{TL^2}^2((\mu,K),(\nu,L)) = \inf_{\pi\in\Pi(\mu,\nu)} \int_{\Omega\times \Omega} \|x-y\|^2 + \|K(x)-L(y)\|^2 \, \dd \pi(x,y).
\end{equation}
Here $\Pi(\mu,\nu)$ denotes the set of all Borel probability measures on $\Omega \times \Omega$ whose marginals on the first and second variable are $\mu$ and $\nu$, respectively (so-called couplings).
We say that $(\mu_n,K^{(n)})\to (\mu,K)$ in $TL^2$ if $d_{TL^2}^2((\mu_n,K^{(n)}),(\mu,K))\to 0$.
In the Monge formulation we  write
\[ d_{TL^2}^2((\mu,K),(\nu,L)) = \inf_{T\,:\, T_{\#}\mu = \nu} \int_{\Omega} \|x-T(x)\|^2 + \|K(x)-L(T(x))\|^2 \, \dd \mu(x). \]
We  note that the Monge formulation is not always defined as there may not exist transport maps $T$ between $\mu$ and $\nu$.
Comparing the metric in $TL^2$ to the Wasserstein distance
\[ d_W^2(\mu,\nu) = \lb \begin{array}{ll} \inf_{\pi\in\Pi(\mu,\nu)} \int_{\Omega\times\Omega} \|x-y\|^2\, \dd \pi(x,y) & \text{in the Kantorovich formulation} \\ \inf_{T\,:\, T_{\#}\mu=\nu} \int_\Omega \| x-T(x)\|^2 \, \dd \mu(x) & \text{in the Monge formulation,} \end{array} \rd \]
we can see that $d_{TL^2}((\mu,K),(\nu,L)) = d_W((\Id\times K)_{\#}\mu,(\Id\times L)_{\#}\nu)$.

In our case we choose $\mu$ to be the Lebesgue measure on $[0,1]$, $\nu=\mu_n$ the discrete measure defined above, and $L=K^{(n)}$.
It is a consequence of results in~\cite{garciatrillos16} (since $\mu_n$ converges weakly$^*$ to $\mu$ ---  we say weak$^*$ to be consistent with notation in functional analysis rather than weak which is often the notation in statistics) that
\[ d_{TL^2}((\mu,K),(\mu_n,K^{(n)})) \to 0 \quad \Leftrightarrow \quad d_1(K^{(n)},K) \to 0. \]
More precisely, in our setting $d_{TL^2}((\mu,K),(\mu_n,K^{(n)})) \to 0$ is equivalent to $\mu_n\weakstarto\mu$ and the existence of a sequence of transport maps $T^{(n)}$ (between $\mu_n$ and $\mu$) such that $T^{(n)}\to\Id$ in $L^2$ and $K^{(n)}\circ T^{(n)}\to K$ in $L^2$.
The existence of such a sequence $T^{(n)}$ is guaranteed in our case, since we can choose $T^{(n)}(t) = \bar{t}_i$ for $t\in (\bar{t}_{i-1},\bar{t}_i)$ where $\bar{t}_i=\frac{i}{n-1}$, which leads to $K^{(n)} \circ T^{(n)}$ being a piecewise constant interpolation of $K^{(n)}$.
Hence we can use the simpler function $d_1$.
We note that $d_1$ is not a metric (for example $d_1(K,K^{(n)})$ does not make sense hence $d_1$ is not symmetric), however due to the relationship of $d_1$ with $d_{TL^2}$ we can still take advantage of metric properties.

Similarly, we define $d_2(b^{(n)},b) = \|\tilde{b}^{(n)}-b\|_{L^2}$ and the distance between $\theta = (K,b,W,c)$ and $\theta^{(n)} = (K^{(n)},b^{(n)},W^{(n)},c^{(n)})$ is given by
\begin{align}
d: \Theta^{(n)} \times \Theta & \mapsto [0,\infty) \notag \\
d( \theta^{(n)}, \theta ) & = d_1(K^{(n)},K) + d_2(b^{(n)},b) + \|W^{(n)}-W\| + \|c^{(n)} - c\|. \label{eq:distance}
\end{align}

We could also have used a piecewise linear interpolation rather than the piecewise constant interpolation we use, i.e.we could have defined
\[ \bar{K}^{(n)}(t) = \frac{\bar{t}_i-t}{\bar{t}_i-\bar{t}_{i-1}} K^{(n)}_{i-1} + \frac{t-\bar{t}_{i-1}}{\bar{t}_i-\bar{t}_{i-1}} K^{(n)}_{i} \quad \text{for } t\in (\bar{t}_{i-1},\bar{t}_i) \quad i=1,\dots, n-1 \,\, \text{where } \bar{t}_i = \frac{i}{n-1}, \]
and compared $\bar{d}_1(K^{(n)},K):=\| \bar{K}^{(n)} - K\|_{L^2}$.
However, under appropriate conditions (which are satisfied in this paper)
\[ \bar{d}_1(K^{(n)},K)\to 0 \quad \Leftrightarrow \quad d_1(K^{(n)},K)\to 0. \]
Since the piecewise constant and piecewise linear constructions both generate the $TL^2$ topology, we choose the simpler former one. This also gives us a metric space structure that we can use to establish $\Gamma$-limits.

\subsection{\texorpdfstring{$\Gamma$}{Gamma}-Convergence \label{subsec:Prelim:Gamma}}

Recall that we wish to show minimisers of $\cE_n$ converge to minimisers of $\cE_\infty$.
In particular, we want to show that $\cE_\infty$ is the variational limit of $\cE_n$.
To characterise variational convergence we first define the $\Gamma$-limit in a general metric space setting.

\begin{mydef}
\label{def:Prelim:Gamma}
Let $\cE_n:\Omega\to \bbR\cup\{\pm\infty\}$, $\cE_\infty:\Omega\to \bbR\cup\{+\infty\}$ where $(\Omega,d)$ is a metric space.
Then $\cE_n$ $\Gamma$-converges to $\cE_\infty$, and we write $\cE_\infty=\Glim_{n\to \infty} \cE_n$, if for all $u\in \Omega$ the following holds:
\begin{enumerate}
\item (the liminf inequality) for any $u_n\to u$
\[ \liminf_{n\to \infty} \cE_n(u_n) \geq \cE_\infty(u); \]
\item (the recovery sequence) there exists $u_n\to u$ such that
\[ \limsup_{n\to \infty} \cE_n(u_n) \leq \cE_\infty(u). \]
\end{enumerate}
\end{mydef}

For brevity we focus only on the key property of $\Gamma$-convergence, and the property that justifies the term variational convergence.
For a more substantial introduction to $\Gamma$-convergence we refer to~\cite{braides02,dalmaso93}.

\begin{theorem}
\label{thm:Prelim:MinConv}
Let $(\Omega,d)$ be a metric space and $\cE_n$ a proper sequence of functionals on $\Omega$.
Let $u_n$ be a sequence of almost minimizers for $\cE_n$, i.e. $\cE_n(u_n)\leq \max\{\inf_{u\in\Omega} \cE_n(u_n) + \eps_n,-\frac{1}{\eps_n}\}$ for some $\eps_n\to 0^+$.
Assume that $\cE_\infty = \Glim_{n\to\infty} \cE_n$ and $\{u_n\}_{n=1}^\infty$ are relatively compact.
Then,
\[ \inf_{u\in \Omega} \cE_n(u) \to \min_{u \in \Omega} \cE_\infty(u) \]
where the minimum of $\cE_\infty$ exists.
Moreover if $u_{n_m}\to u_\infty$ is a convergent subsequence then $u_\infty$ minimises $\cE_\infty$.
\end{theorem}

Clearly if one assumes that the minimum of $\cE_\infty$ is unique then, by the above theorem, $u_n\to u_\infty$ (without recourse to subsequences) where $u_\infty$ is the unique minimiser of $\cE_\infty$.

Theorem~\ref{thm:Prelim:MinConv} forms the basis for our proof of Theorem~\ref{thm:MainRes:Conv}.
In order to apply Theorem~\ref{thm:Prelim:MinConv} we must show that minimisers are relatively compact and $\cE_\infty=\Glim_{n\to \infty} \cE_n$.

We note that Definition~\ref{def:Prelim:Gamma} and Theorem~\ref{thm:Prelim:MinConv}
are in the context of metric spaces.
As we described in Section~\ref{subsec:Prelim:Top} we can describe the convergence of $K^{(n)}$ in terms of the $TL^2$ distance $d_{TL^2}$ which is a metric on the space $\Omega=\{(\mu,f) \, : \, f\in L^2(\mu;\bbR^{d\times d}), \mu\in\cP([0,1])\}$ (and similarly for $b^{(n)}$).
Hence we can use the distance
\[ \tilde{d}\l (\mu,K,b,W,c), (\nu,L,a,V,d) \r = d_{TL^2}((\mu,K),(\nu,L)) + d_{TL^2}((\mu,b),(\nu,a)) + \|W-V\| + \|c-d\| \]
which is a metric on the space
\[ \lb (\mu,K,b,W,c) \, : \, K\in L^2(\mu;\bbR^{d\times d}), b\in L^2(\mu;\bbR^d), W\in \bbR^{d\times d}, c\in \bbR^d, \mu\in \cP([0,1]) \rb. \]
Since convergence in $\tilde{d}$ is equivalent to convergence in $d$, we can simplify our notation by considering sequences that converge in $d$ whilst still being able to apply Theorem~\ref{thm:Prelim:MinConv}.

\subsection{Sobolev Spaces}\label{subsec:Prelim:Sob}

For readers unfamiliar with Sobolev spaces, in this section we provide some definitions and results that are needed to read the remainder of the current paper. For a more detailed introduction and further in-depth study of these concepts we refer the reader to \cite{adams2003sobolev,leoni09}.

We define the Sobolev space $H^k([0,1])$ recursively: for $k\geq 2$, $f\in H^k([0,1])$ if $\dot{f}\in H^{k-1}([0,1])$ where $\dot{f}$ is the weak derivative of $f$, and for $k=1$, $f\in H^1([0,1])$ if $f\in L^2([0,1])$ and $\dot{f}\in L^2([0,1])$.
We can replace $H^k$ with $H^k_{\loc}$ by replacing $L^2$ with $L^2_{\loc}$ in the previous definition (where $L^2_{\loc}([0,1])$ is the set of functions that are in $L^2([a,b])$ for every $0<a<b<1$).
The Sobolev norm in $H^1([0,1])$ is defined by
\[ \| f\|_{H^1([0,1])} = \| f\|_{L^2([0,1])} + \|\dot{f}\|_{L^2([0,1])}. \]
Of course these definitions extend to $p$-norms and functions of several variables \cite{leoni09}.

Morrey's inequality in one dimension \cite[Theorem 11.34]{leoni09}  implies that there exists a constant $C$ such that $\|f\|_{C^{0,\scriptscriptstyle\frac12}} \leq C\|f\|_{H^1}$  for all $f\in H^1$. We note in particular that such an $f$ has a continuous representative, so that the pointwise evaluation $f(0)$ is well-defined.
Therefore,
\[ |f(0)| + \|\dot{f}\|_{L^2} \leq \| f\|_{C^0} + \|\dot{f}\|_{L^2} \leq (C+1)\|f\|_{H^1}. \]
Moreover, for any $f\in H^1$ we have $|f(x)-f(y)|\leq C\sqrt{|x-y|}\|\dot{f}\|_{L^2}$ \cite[Remark 11.35]{leoni09}  so that $|f(x)| \leq |f(0)| + C\sqrt{|x|}\|\dot{f}\|_{L^2} \leq |f(0)| + C\|\dot{f}\|_{L^2}$ implying
\[ \|f\|_{H^1} \leq |f(0)| + (C+1) \|\dot{f}\|_{L^2}. \]
It follows that $|f(0)| + \|\dot{f}\|_{L^2}$ and $\|f\|_{H^1}$ are equivalent norms.
In particular, our regularisation terms $R_\infty^{(1)}$ and $R_\infty^{(2)}$ in the deep layer limit (see \eqref{eq:continuumregularisers}) are equivalent to $H^1$ norms.
We furthermore have the Rellich--Kondrachov type embedding result that $H^1([0,1])$ is compactly embedded in both $C^{0,\scriptscriptstyle\frac12}([0,1])$ and $L^2([0,1])$ \cite[Section 11.3]{leoni09}.

%The finite layer regularisation is a discrete approximation of the Sobolev norm (see~\eqref{eq:Intro:Reg:NonPar:DisApprx}).
Although the finite layer regularisation uses a finite difference approximation of the derivative (as one cannot use usual derivatives in discrete spaces) one can expect minimizers of $\mathcal{E}_n$ to enjoy similar regularity properties in the deep layer limit when the scale in the discretisation goes to zero, as minimizers of $\mathcal{E}_\infty$ have.

Non-local characterisations of Sobolev spaces are possible, see for example~\cite[Theorem 10.55]{leoni09}, and we utilise such ideas to prove $\Gamma$-convergence.

\section{Proofs \label{sec:ProofConv}}

The proof of Theorem~\ref{thm:MainRes:Conv} is a straightforward application of the following theorem, Theorem~\ref{thm:ProofConv:Gamma&Compact}, with Theorem~\ref{thm:Prelim:MinConv}.
This section is devoted to the proofs of Theorem~\ref{thm:ProofConv:Gamma&Compact}, Proposition~\ref{prop:MainRes:Reg} and Corollary~\ref{cor:convergenceforwardpass}.

\begin{theorem}
\label{thm:ProofConv:Gamma&Compact}
Under the assumptions of Theorem~\ref{thm:MainRes:Conv}, the following holds:
\begin{enumerate}
\item for every $n\in \bbN$ there exists a minimiser of $\cE_n$ in $\Theta^{(n)}$,
\item any sequence $\{(K^{(n)},b^{(n)},W^{(n)},c^{(n)})\}_{n\in \bbN}$ which is bounded in $\cE_n$, i.e.
\[ \sup_{n\in \bbN} \cE_n(K^{(n)},b^{(n)},W^{(n)},c^{(n)})<\infty, \]
is relatively compact, and
\item $\Glim_{n\to \infty} \cE_n = \cE_\infty$.
\end{enumerate}
\end{theorem}

The first three subsections are each dedicated to the proof of one part of the above theorem.
In Section~\ref{subsec:ProofConv:Compact} we show that sequences bounded in $\cE_n$ are relatively compact.
The argument relies on approximating discrete sequences $\theta^{(n)} = (K^{(n)},b^{(n)},W^{(n)},c^{(n)})\in \Theta^{(n)}$ with a continuum sequence $\tilde{\theta}^{(n)} = (\tilde{K}^{(n)},\tilde{b}^{(n)},W^{(n)},c^{(n)})\in \Theta$ and using standard Sobolev embedding arguments to deduce the compactness of $\tilde{\theta}^{(n)}$, and therefore $\theta^{(n)}$.

In Section~\ref{subsec:ProofConv:MinExist} we prove the existence of minimizers.
The strategy is to apply the direct method from the calculus of variations.
That is, we show that $\cE_n$ is lower semi-continuous (in fact continuous).
For compactness of minimizing sequences it is enough to show bounded in norm (since for finite $n$ parameters are finite dimensional).
Compactness plus lower semi-continuity is enough to imply the existence of minimisers.

In the third subsection we prove the $\Gamma$-convergence of $\cE_n$ to $\cE_\infty$.
This relies on a variational convergence of finite differences.

In Section~\ref{subsec:ProofConv:Reg} we analyse the regularity of minimisers of $\cE_\infty$ and prove Proposition~\ref{prop:MainRes:Reg}.
To show this we compute the G\^{a}teaux derivative then apply methods from elliptic regularity theory to infer additional smoothness.
In this section we assume that the norms $\|\cdot\|$ on $\bbR^d$ and $\bbR^{d\times d}$ are induced by an inner product $\langle\cdot,\cdot\rangle$.

Finally, in Section~\ref{sec:networkasdynamicalsystem} we prove the uniform convergence of the parameters of the neural network to parameters of the continuum model (Corollary~\ref{cor:convergenceforwardpass}).

\subsection{Proof of Compactness \label{subsec:ProofConv:Compact}}

We start with a preliminary result which implies that $\|K^{(n)}\|_{L^\infty(\mu_n)} \leq C R^{(1)}_n(K^{(n)})$, this is a discrete analogue of the well known Morrey's inequality.
We include the proof as it is important that the constant $C$ can be chosen independently of $\mu_n$.

In the following, where we write $\R^\kappa$, $\kappa$ can be any integer. Specific choices for $\kappa$ will be made when the result is applied.

\begin{proposition}
\label{prop:ProofConv:Compact:Morrey}
Fix $n\in \bbN$ and let $t_i=\frac{i}{n}$, $\mu_n=\frac{1}{n} \sum_{i=0}^{n-1} \delta_{t_i}$, and $f_n:\{t_i\}_{i=0}^{n-1}\to \bbR^\kappa$.
Then
\[ \| f_n\|_{L^\infty(\mu_n)}^2 \leq 2 \l \|f_n(t_0)\|^2 + n\sum_{j=1}^{n-1} \|f_n(t_j) - f_n(t_{j-1})\|^2 \r. \]
\end{proposition}

\begin{proof}
We note that
\[ \| f_n(t_i) - f_n(t_0) \|^2  \leq \l \sum_{j=1}^i \| f_n(t_j) - f_n(t_{j-1}) \| \r^2 \leq n \sum_{j=1}^{n-1} \| f_n(t_j) - f_n(t_{j-1}) \|^2 \]
by Jensen's inequality.
Hence,
\begin{align*}
\| f_n(t_i)\|^2 & \leq 2 \l \| f_n(t_i) - f_n(t_0) \|^2 + \|f_n(t_0)\|^2 \r \\
 & \leq 2 \l \|f_n(t_0)\|^2 + n \sum_{j=1}^{n-1} \| f_n(t_j) - f_n(t_{j-1})\|^2 \r.
\end{align*}
Taking the supremum over $i\in\{0,1,\dots,n-1\}$ proves the proposition.
\end{proof}

Let $(K^{(n)},b^{(n)},W^{(n)},c^{(n)})\in \Theta^{(n)}$ be a sequence such that $\sup_{n\in \bbN}\cE_n(K^{(n)},b^{(n)},W^{(n)},c^{(n)})<+\infty$.
Then compactness of $\{W^{(n)}\}_{n\in \bbN}$ and $\{c^{(n)}\}_{n\in \bbN}$ is immediate from the regularisation functionals $R^{(3)}$ and $R^{(4)}$.
For $K^{(n)}$ and $b^{(n)}$ we deduce compactness by using a smooth continuum approximation.
In particular, let $f_n:\{t_i\}_{i=0}^{n-1}\to \bbR^\kappa$, where $t_i=\frac{i}{n}$, be a sequence of discrete functions that are bounded in the discrete $H^1$ norm $\cR_n$ given by
\[ \cR_n(f_n)= \sqrt{\|f_n(t_0)\|^2 + n\sum_{j=1}^{n-1} \|f_n(t_j) - f_n(t_{j-1})\|^2}. \]
We compare $f_n$ to a smooth continuum function $g_n:[0,1]\to \bbR^\kappa$ with the property $\|g_n\|_{H^1} \lesssim \cR_n(f_n)$.
By Sobolev embedding arguments we have that $\{g_n\}_{n\in \bbN}$ is relatively compact in $L^2([0.1];\R^\kappa)$.
Compactness of $\{f_n\}_{n\in \bbN}$ follows from $\|f_n\circ T_n - g_n\|_{L^2}\to 0$ where $T_n$ is the map $T_n(t) = t_i$ if $t\in [t_i,t_{i+1})$.

In the following proposition $\Leb\lfloor_{[0,1]}$ is the Lebesgue measure on $\bbR$ restricted to the interval $[0,1]$.

\begin{proposition}
\label{prop:ProofConv:Compact:SobEmb}
For each $n\in \bbN$ let $t_i^{(n)}=\frac{i}{n}$, $\mu_n=\frac{1}{n} \sum_{i=0}^{n-1} \delta_{t_i^{(n)}}$, and $f_n:\{t_i^{(n)}\}_{i=0}^{n-1}\to \bbR^\kappa$.
If
\begin{equation} \label{eq:ProofConv:Compact:RBound}
\sup_{n\in \bbN} \l \|f_n(0)\|^2 + n\sum_{j=1}^{n-1} \|f_n(t_j^{(n)}) - f_n(t_{j-1}^{(n)})\|^2 \r < + \infty
\end{equation}
then $\{(\mu_n,f_n)\}_{n\in \bbN}$ is relatively compact in $TL^2$ and any cluster point $(\mu,f)$ satisfies $\mu=\Leb\lfloor_{[0,1]}$ and $f\in C^{0,\gamma}\left([0,1]; \bbR^\kappa\right)$ for any $\gamma<\frac12$.
Furthermore, for any converging subsequence there exists a further subsequence (which we relabel), and a $f\in C^{0,\gamma}\left([0,1]; \bbR^\kappa\right)$, such that
\begin{equation} \label{eq:ProofConv:Compact:UnifConv}
\max_{i\in \{0,1,\dots, n-1\}} \left\| f_n(t_i^{(n)}) - f(t_i^{(n)}) \right\| \to 0.
\end{equation}
\end{proposition}

\begin{proof}
First note that
\begin{align*}
\|f_n(t_i^{(n)})\| & \leq \|f_n(0)\| + \sum_{j=1}^{i-1} \|f_n(t_j^{(n)}) - f_n(t_{j-1}^{(n)})\| \\
 & \leq \|f_n(0)\| + \sum_{j=1}^{n-1} \|f_n(t_j^{(n)}) - f_n(t_{j-1}^{(n)})\| \\
 & \leq 1 + \frac12\l \|f_n(0)\|^2 + n\sum_{j=1}^{n-1} \|f_n(t_j^{(n)}) - f_n(t_{j-1}^{(n)})\|^2 \r,
\end{align*}
with the last line following from Young's inequality, so by~\eqref{eq:ProofConv:Compact:RBound} $\|f_n\|_{L^\infty(\mu_n)}$ is bounded.
In particular,
%By Proposition~\ref{prop:ProofConv:Compact:Morrey} and~\eqref{eq:ProofConv:Compact:RBound}
there exists $M<+\infty$ such that
\[ \|f_n(0)\|^2 + n\sum_{j=1}^{n-1} \|f_n(t_j^{(n)}) - f_n(t_{j-1}^{(n)})\|^2 \leq M, \quad \quad \| f_n\|_{L^\infty(\mu_n)} \leq M. \]
Let $\tilde{f}_n$ be the continuum extension of $f_n$ defined by
\[ \tilde{f}_n(t) = \lb \begin{array}{ll} f_n(0) & \text{if } t<0 \\ f_n(t_i^{(n)}) & \text{if } t\in [t_i^{(n)},t_{i+1}^{(n)}) \text{ for some } i=0,\dots,n-1 \\ f_n(t_{n-1}^{(n)}) & \text{if } t\geq 1. \end{array} \rd \]
Define $g_n=J_{\eps_n}\ast \tilde{f}_n$ where $J\in C^\infty(\bbR)$ is a standard mollifier \cite[Remark C.18(ii)]{leoni09} with $\|J\|_{L^1} = 1$, $\|J\|_{L^\infty}\leq \beta$, for some $\beta>0$, $J_\eps=\frac{1}{\eps} J(t/\eps)$, for all $\eps>0$, and $\eps_n=\frac{1}{2n}$.
We recall the following facts about mollifiers
(which are stated in the domain $\bbR$, but hold for higher dimensional Euclidean spaces as well):
\begin{enumerate}
\item[(M1)] $\| J_\eps \ast f\|_{L^\infty}\leq \|f\|_{L^\infty}$, for any $f\in L^\infty$ and any $\eps>0$ (by Young's inequality \cite[Theorem C.15]{leoni09});
\item[(M2)] $\frac{\dd}{\dd t} (J_\eps \ast f) = \frac{1}{\eps} (\dot{J})_\eps \ast f$, for any $f\in L^1$ and any $\eps>0$, where $(\dot{J})_\eps(t) = \frac{1}{\eps} \dot{J}(t/\eps)$ \cite[Theorem C.20]{leoni09};
\item[(M3)] $\int_\bbR (\dot{J})_\eps(s)\, \dd s = 0$ (since the order of integration and differentiation can be reversed, as $(\dot J)_\eps$ is continuous and supported on a compact subset of $\R$);
\item[(M4)] $\| (\dot{J})_\eps \|_{L^\infty} \leq \frac{\| J\|_{L^\infty}}{\eps}$, for any $\eps>0$.
\end{enumerate}

We first show that $g_n$ is bounded in $H^1([0,1]; \R^\kappa)$.
% \B{Since
%\begin{align*}
%\|f_n(t_i^{(n)})\| & \leq \|f_n(0)\| + \sum_{j=1}^{i-1} \|f_n(t_j^{(n)}) - f_n(t_{j-1}^{(n)})\| \\
% & \leq \|f_n(0)\| + \sum_{j=1}^{n-1} \|f_n(t_j^{(n)}) - f_n(t_{j-1}^{(n)})\| \\
% & \leq 1 + \frac12\l \|f_n(0)\|^2 + n\sum_{j=1}^{n-1} \|f_n(t_j^{(n)}) - f_n(t_{j-1}^{(n)})\|^2 \r,
%\end{align*}
%with the last line following from Young's inequality, then by \eqref{eq:ProofConv:Compact:RBound} $\|f_n\|_{L^\infty(\mu_n)}$ is bounded.
%Moreover,
As $\|f_n\|_{L^\infty(\mu_n)}$ is bounded then $\|\tilde{f}_n\|_{L^\infty([0,1])}$ is bounded, so by (M1) $g_n$ is bounded in $L^\infty([0,1]; \R^\kappa)$.
It is therefore sufficient to show that $\sup_{n\in \bbN} \|\dot{g}_n\|_{L^2}<+\infty$.
For $t\in [t_i^{(n)},t_i^{(n)}+\eps_n]$ and $i\geq 1$ we have,
%\begin{align*}
%\| \dot{g}_n(t)\| & = \lda \int_{\bbR} \frac{\dd}{\dd t} J_{\eps_n} (s-t) \l \tilde{f}_n(s) - \tilde{f}_n(t) \r \, \dd s \rda \\
% & \leq \frac{\beta}{\eps_n^2} \int_{t_{i-1}^{(n)}}^{t_{i+1}^{(n)}} \| \tilde{f}_n(s) - \tilde{f}_n(t) \| \, \dd s \\
% & = 4n\beta \| f_n(t_i^{(n)}) - f_n(t_{i-1}^{(n)})\|
%\end{align*}
\begin{align*}
\| \dot{g}_n(t)\| & = \frac{1}{\eps_n} \lda (\dot{J})_{\eps_n} \ast \tilde{f}_n(t) \rda \qquad \text{by (M2)} \\
 & = \frac{1}{\eps_n} \lda \int_{\bbR} (\dot{J})_{\eps_n}(t-s) \tilde{f}_n(s) \, \dd s \rda \\
 & = \frac{1}{\eps_n} \lda \int_{\bbR} (\dot{J})_{\eps_n}(t-s) \l \tilde{f}_n(s) - \tilde{f}_n(t) \r \, \dd s \rda \qquad \text{by (M3)} \\
 & \leq \frac{\beta}{\eps_n^2} \int_{\bbR} \lda \tilde{f}_n(s) - \tilde{f}_n(t) \rda \, \dd s \qquad \text{by (M4)} \\
 & = 4\beta n \lda f_n(t_i^{(n)}) - f_n(t_{i-1}^{(n)}) \rda.
\end{align*}
%where in the first line we use that $\int_{\bbR} \frac{\dd}{\dd t} J_{\eps_n} (s-t) \, \dd s = 0$.
Similarly, for $t\in [t_i^{(n)}+\eps_n,t_{i+1}^{(n)}]$ and $i\leq n-2$ we have,
\[ \| \dot{g}_n(t) \| \leq 4n\beta \| f_n(t_{i+1}^{(n)}) - f_n(t_i^{(n)})\|. \]
From the definition of $\tilde{f}_n$ we have that $\dot{g}_n(t)=0$ for all $t\leq \eps_n$ or $t\geq 1-\eps_n$.
Squaring and integrating the above inequality over $t\in[0,1]$ implies %It follows that
\[ \| \dot{g}_n\|_{L^2}^2 \leq 16\beta^2 n \sum_{i=1}^{n-1} \| f_n(t_i^{(n)}) - f_n(t_{i-1}^{(n)}) \|^2. \]
Hence $g_n$ is bounded in $H^1([0,1]; \R^\kappa)$.

By Morrey's inequality~\cite[Theorem 11.34]{leoni09}, %the Rellich--Kondrachov Theorem
$g_n$ is relatively compact in $C^{0,\gamma}([0,1]; \R^\kappa)$ for any $\gamma\in (0,\frac12)$.
In particular $g_n$ is relatively compact in $L^\infty([0,1]; \R^\kappa)$.
Hence, we may assume that there exists a subsequence (which we relabel) and $g\in C^{0,\gamma}$ such that $g_n\to g$ in $L^\infty([0,1]; \R^\kappa)$.
The proposition is proved once we show $\|\tilde{f}_n-g_n\|_{L^\infty}\to 0$.
For $t\in [t_i^{(n)},t_i^{(n)}+\eps_n]$ we have
\begin{align*}
\| \tilde{f}_n(t) - g_n(t) \| & = \lda \int_{\bbR} J_{\eps_n}(s-t) \l \tilde{f}_n(s) - \tilde{f}_n(t) \r \, \dd s \rda \\
 & \leq \frac{\beta}{\eps_n} \int_{t_{i-1}^{(n)}}^{t_{i+1}^{(n)}} \| \tilde{f}_n(s) - \tilde{f}_n(t) \| \, \dd s \\
 & = \lb \begin{array}{ll} 2\beta \|f_n(t_i^{(n)})-f_n(t_{i-1}^{(n)})\| & \text{if } i\geq 1 \\ 0 & \text{if } i=0. \end{array} \rd
\end{align*}
Similarly, for $t\in [t_i^{(n)}+\eps_n,t_{i+1}^{(n)}]$ we have
\[ \| \tilde{f}_n(t) - g_n(t) \| \leq \lb \begin{array}{ll} 2\beta \|f_n(t_{i+1}^{(n)})-f_n(t_i^{(n)})\| & \text{if } i\leq n-2 \\ 0 & \text{if } i=n-1. \end{array} \rd \]
Hence
\begin{align*}
\|\tilde{f}_n - g_n\|_{L^\infty}^2 & \leq 4\beta^2 \sup_{i\in \{1,\dots, n-1\}} \|f_n(t_i^{(n)}) - f_n(t_{i-1}^{(n)})\|^2 \\
 & \leq 4\beta^2 \sum_{i=1}^{n-1} \|f_n(t_i^{(n)})-f_n(t_{i-1}^{(n)}) \|^2 \\
 & = O\l\frac{1}{n}\r.
\end{align*}
It follows that $\tilde{f}_n\to g$ in $L^\infty([0,1]; \R^\kappa)$ (and therefore in $L^2([0,1]; \R^\kappa)$) which proves~\eqref{eq:ProofConv:Compact:UnifConv}.
Clearly $\mu_n\weakstarto \Leb\lfloor_{[0,1]}$, hence $(\mu_n,f_n)\to (\Leb\lfloor_{[0,1]},g)$ in the $TL^2$ topology, from which it follows that $\{(\mu_n,f_n)\}_{n\in\bbN}$ is relatively compact in $TL^2$.
\end{proof}

Compactness of sequences bounded in $\cE_n$ is now a simple corollary of the above proposition.

\begin{corollary}
\label{cor:ProofConv:Compact:Compact}
Let $\Theta^{(n)}$ and $\Theta$ be given by~\eqref{eq:Prelim:Top:Thetan} and~\eqref{eq:Prelim:Top:Theta} respectively.
Define $\cE_n$, $\cE_\infty$, $E_n$, $E_\infty$, $R^{(i)}_n$, $R^{(i)}_\infty$, $R^{(j)}$ for $i=1,2$, $j=3,4$ as in Sections~\ref{subsec:Intro:Finite}-\ref{subsec:Intro:Limit}.
Assume that $\alpha_i>0$ for $i=1,2,3,4$, $\tau_j>0$ for $j=1,2$, $h(x)\in (-\infty,+\infty)$ for all $x\in \bbR^d$, $\cL(z,y)\in [0,+\infty)$ for all $x,z\in \bbR^d$, and $\sigma$ is Lipschitz continuous with $\sigma(0)=0$.
If
\[ \sup_{n\in \bbN} \cE_n(K^{(n)},b^{(n)},W^{(n)},c^{(n)}) < +\infty \]
then there exists a subsequence $n_m$ and $(K,b,W,c)\in \Theta$ such that
\[ d\l (K^{(n_m)},b^{(n_m)},W^{(n_m)},c^{(n_m)}), (K,b,W,c) \r \to 0. \]
Furthermore, %if $\sigma(x)\leq C\|x\|$ for some $C>0$ then
$\cE_\infty(K,b,W,c)<+\infty$.
\end{corollary}

\begin{proof}
Relative compactness in $TL^2$ of $\{(K^{(n)}\}_{n=1}^\infty$ and $\{b^{(n)}\}_{n=1}^\infty$ follows from Proposition~\ref{prop:ProofConv:Compact:SobEmb} and compactness of $\{W^{(n)}\}_{n=1}^\infty$ and $\{c^{(n)}\}_{n=1}^\infty$ is immediate from the bounds on $R^{(3)}(W^{(n)})$ and $R^{(4)}(c^{(n)})$. %$W^{(n)},c^{(n)}$.
To see that $\cE_\infty(K,b,W,c)<+\infty$, we note that, by the bound on $\sigma$ we must have that $X(1;x,K,b)$ is finite for any $x$, hence $E_\infty(K,b,W,c;x,y)<+\infty$ for any $(x,y)$.
\end{proof}

In fact one can obtain compactness in a stronger sense; in particular one can show that, if
\[ \sup_{n\in \bbN} \cE_n(K^{(n)},b^{(n)},W^{(n)},c^{(n)}) < +\infty \]
then there exists a subsequence such that
\[ \max_{i\in\{0,\dots,n_m-1\}} \left\| K\l\frac{i}{n_m}\r-K^{(n_m)}_i\right\| \to 0 \quad \text{and} \quad\max_{i\in\{0,\dots,n_m-1\}} \left\| b\l\frac{i}{n_m}\r-b^{(n_m)}_i\right\| \to 0. \]
See Lemma~\ref{lem:ProofConv:Forward:UniformConv}.

\subsection{Proof of Existence of Minimizers \label{subsec:ProofConv:MinExist}}

The existence of minimizers is a straightforward application of the direct method from the calculus of variations.
In particular, for $n\in \bbN$ all parameters are finite dimensional hence it is enough to show that minimizing sequences are bounded.
For $W,c$ this is clear from the regularisation, for $K^{(n)},b^{(n)}$ this follows from Proposition~\ref{prop:ProofConv:Compact:Morrey}.
Lower semi-continuity then implies that converging minimizing sequences converge to minimizers.

\begin{proposition}
\label{prop:ProofConv:MinExist:MinExist}
Let $n\in \bbN$ and $\Theta^{(n)}$ be given by~\eqref{eq:Prelim:Top:Thetan}.
Define $\cE_n$, $E_n$, $R^{(i)}_n$, $R^{(j)}$ for $i=1,2$, $j=3,4$ as in Section~\ref{subsec:Intro:Finite} and~\ref{subsec:Intro:Reg}.
Assume that $\alpha_i>0$ for $i=1,2,3,4$ and $\tau_j>0$ for $j=1,2$.
Further assume that $\sigma$ and $h$ are continuous, that $\sigma(0)=0$, and that $\cL$ is non-negative and continuous in its first argument.
Then, there exists a minimizer of $\cE_n$ in $\Theta^{(n)}$.
\end{proposition}

\begin{proof}
Let $\theta^{(n)}_m=(K^{(n)}_m,b^{(n)}_m,W_m,c_m)\in\Theta^{(n)}$ be a minimizing sequence, i.e.
\[ \cE_n(\theta_m^{(n)}) \to \inf_{\Theta^{(n)}} \cE_n \quad \text{as } m\to\infty. \]
Since $\cE_n(\underline{0}) = \sum_{s=1}^S \cL\l h(\underline{0}), y_i\r =: C<\infty$, we can assume that $\cE_n(\theta_m^{(n)}) \leq C$ for all $m$.
Hence, $\sup_{m\in \bbN}\max\{R_n^{(1)}(K_m^{(n)}),R_n^{(2)}(b_m^{(n)}),R^{(3)}(W_m),R^{(4)}(c_m)\} \leq C$.
We emphasise that all the parameters are finite dimensional.
Since %$R^{(3)},R^{(4)}$ are square norms, \B{i.e.
$\|W_m\|\leq \sqrt{C}$ and $\|c_m\|\leq \sqrt{C}$,  we immediately have that $\{W_m\}_{m\in \bbN}$ and $\{c_m\}_{m\in \bbN}$ are bounded, hence relatively compact.
By Proposition~\ref{prop:ProofConv:Compact:Morrey} $\{K_m^{(n)}\}_{m\in \bbN}$ and $\{b_m^{(n)}\}_{m\in \bbN}$ are also bounded in the supremum norm, hence relatively compact.

With recourse to a subsequence, we assume that $(K^{(n)}_m,b^{(n)}_m,W_m,c_m)\to (K^{(n)},b^{(n)},W,c)=\theta^{(n)}\in\Theta^{(n)}$.
By induction on $i$ it is easy to see that $X_i[x,K_m^{(n)},b_m^{(n)}]\to X_i[x,K^{(n)},b^{(n)}]$ as $m\to \infty$ (by continuity of $\sigma$).
Hence, by continuity of $h$ and $\cL(\cdot,y_i)$, it follows that $\cE_n(\theta^{(n)}_m)\to \cE_n(\theta^{(n)})$.
Now since,
\[ \cE_n(\theta^{(n)}) = \lim_{m\to \infty} \cE_n(\theta^{(n)}_m) = \inf_{\Theta^{(n)}} \cE_n \]
it follows that $\cE_n(\theta^{(n)}) = \inf_{\Theta^{(n)}} \cE_n$.
\end{proof}

\subsection{\texorpdfstring{$\Gamma$}{Gamma}-Convergence of \texorpdfstring{$\cE_n$}{En} \label{subsec:ProofConv:GammaEn}}

In this section we prove the $\Gamma$-convergence of $\cE_n$ to $\cE_\infty$.
We divide the result into two parts: the liminf inequality is in Lemma~\ref{lem:ProofConv:GammaEn:Liminf}, and the existence of a recovery sequence is given in Lemma~\ref{lem:ProofConv:GammaEn:Limsup}.
Before getting to these results we start with some preliminary results, the first is that, for any $K^{(n)}\to K$ and $b^{(n)}\to b$, the discrete model~\eqref{eq:Intro:ResNet:Rec} converges uniformly to the continuum model~\eqref{eq:Intro:Limit:ODE}.
The next preliminary result uses this to infer the convergence of $E_n(\theta^{(n)};x,y)\to E(\theta;x,y)$.

\begin{lemma}
\label{lem:ProofConv:GammaEn:ConvXn}
Consider sequences $K^{(n)}\in L^2(\mu_n;\bbR^{d\times d})$, $b^{(n)}\in L^2(\mu_n;\bbR^d)$ where $\mu_n=\frac{1}{n}\sum_{i=0}^{n-1} \delta_{t_i}$ and $t_i=\frac{i}{n}$.
Let $d_1(K^{(n)}, K)\to 0$ and $d_1(b^{(n)}, b)\to 0$ where $K\in H^1([0,1];\bbR^{d\times d})$ and $b\in H^1([0,1];\bbR^d)$.
Define $R_n^{(i)}$, $i=1,2$, as in Section~\ref{subsec:Intro:Reg} with $\tau_i>0$.
Assume that $\sigma$ is Lipschitz continuous with constant $L_\sigma$, $\sigma(0)=0$, $\max\{\sup_{n\in \bbN} R_n^{(1)}(K^{(n)}),\sup_{n\in \bbN} R_n^{(2)}(b^{(n)})\}<+\infty$ and $x\in \bbR^d$.
Then $\| X(\cdot;x,K,b)\|_{L^\infty} \leq C$ where $C$ depends only on $L_\sigma$, $\|x\|$, $\|K\|_{L^\infty}$, and $\|b\|_{L^\infty}$ and furthermore
\[ \sup_{i\in\{0,1,\dots,n-1\}} \sup_{t\in [t_i,t_{i+1}]} \lda X(t;x,K,b) - X_i^{(n)}[x;K^{(n)},b^{(n)}] \rda \to 0 \]
where $X(t;x,K,b)$ and $X^{(n)}_i[x,K^{(n)},b^{(n)}]$ are determined by~\eqref{eq:Intro:Limit:ODE} and~\eqref{eq:Intro:ResNet:Rec} respectively.
\end{lemma}

\begin{proof}
Let $X_i^{(n)} = X_i^{(n)}[x;K^{(n)},b^{(n)}]$ and $X(t) = X(t;x,K,b)$.
We have
\begin{align*}
\lda X\l \frac{i}{n}\r - X_i^{(n)} \rda & = \lda X\l\frac{i-1}{n}\r + \int_{\frac{i-1}{n}}^{\frac{i}{n}} \dot{X}(t) \, \dd t - X_{i-1}^{(n)} - \l X_i^{(n)} - X_{i-1}^{(n)} \r \rda \\
 & \leq \lda X\l\frac{i-1}{n}\r - X_{i-1}^{(n)} \rda + \lda \int_{\frac{i-1}{n}}^{\frac{i}{n}} \dot{X}(t) \, \dd t - \l X_i^{(n)} - X_{i-1}^{(n)}\r \rda.
\end{align*}
Using the iterative update for $X_{i}^{(n)}$, i.e.~\eqref{eq:Intro:ResNet:Rec}, the continuum differential equation governing the dynamics of $X(t)$, i.e.~\eqref{eq:Intro:Limit:ODE}, and the Lipschitz bound on $\sigma$ we may bound the second term above by the following:
\begin{align}
\lda \int_{\frac{i-1}{n}}^{\frac{i}{n}} \dot{X}(t) \, \dd t - \l X_i^{(n)} - X_{i-1}^{(n)}\r \rda & = \lda \int_{\frac{i-1}{n}}^{\frac{i}{n}} \sigma(K(t)X(t)+b(t)) - \sigma\l K_{i-1}^{(n)}X_{i-1}^{(n)} + b_{i-1}^{(n)}\r \, \dd t \rda \notag \\
 & \leq L_\sigma \int_{\frac{i-1}{n}}^{\frac{i}{n}} \lda K(t)X(t) - b(t) - K_{i-1}^{(n)} X_{i-1}^{(n)} - b_{i-1}^{(n)} \rda \, \dd t \notag \\
 & \leq L_\sigma \int_{\frac{i-1}{n}}^{\frac{i}{n}} \lda b(t) - b_{i-1}^{(n)} \rda + \lda K(t) X(t) - K_{i-1}^{(n)}X_{i-1}^{(n)} \rda \, \dd t. \label{eq:ProofConv:GammaEn:XdotDiff}
\end{align}

By Proposition~\ref{prop:ProofConv:Compact:Morrey} we can show that $\|K\|_{L^\infty},\|b\|_{L^\infty}$ is finite (since $K^{(n)}\to K$ and $K^{(n)}$ is uniformly bounded in $L^\infty(\mu_n; \R^{d\times d})$, analogously for $b$).
Now we show that $\sup_{n\in \bbN} \|X^{(n)}\|_{L^\infty(\mu_n)}<+\infty$.
We have, by~\eqref{eq:Intro:ResNet:Rec} and the Lipschitz assumption on $\sigma$,
\[ \|X_{i+1}^{(n)} - X_i^{(n)}\| \leq \frac{L_\sigma}{n} \|K_i^{(n)} X_i^{(n)} + b_i^{(n)}\| \leq \frac{L_\sigma M_1}{n} \l \|X_i^{(n)}\| + 1 \r \]
where $M_1=\sup_{n\in \bbN} \max\{\|K^{(n)}\|_{L^\infty(\mu_n)},\|b^{(n)}\|_{L^\infty(\mu_n)}\} < + \infty$ by Proposition~\ref{prop:ProofConv:Compact:Morrey}.
Hence,
\[ \| X_{i+1}^{(n)}\| \leq \l 1+\frac{L_\sigma M_1}{n}\r \|X_i^{(n)}\| + \frac{L_\sigma M_1}{n}. \]
Since $X_j^{(n)}=x$, by induction it follows that, for $j\in\{1,\dots,n\}$,
\begin{align*}
\| X_j^{(n)} \| & \leq \|x\| \l 1+\frac{L_\sigma M_1}{n} \r^j + \frac{L_\sigma M_1}{n} \sum_{i=1}^j \l 1+\frac{L_\sigma M_1}{n}\r^{i-1} \\
 & \leq \l \|x\| + L_\sigma M_1\r \l 1+ \frac{L_\sigma M_1}{n}\r^n \\
 & \to \l \|x\|+L_\sigma M_1\r e^{L_\sigma M_1}, \quad \text{as } n\to \infty.
\end{align*}
Hence $\sup_{n\in \bbN} \|X^{(n)}\|_{L^\infty(\mu_n)}<+\infty$.

Now consider, for $0\leq s_1<s_2\leq 1$,
\begin{align*}
\|X\|_{L^\infty([s_1,s_2])} & = \sup_{s\in [s_1,s_2]} \|X(s)\| \\
 & = \sup_{s\in [s_1,s_2]} \lda \int_{s_1}^{s} \dot{X}(r) \, \dd r + X(s_1) \rda \\
 & = \sup_{s\in [s_1,s_2]} \lda \int_{s_1}^s \sigma\l K(r) X(r) + b(r)\r \, \dd r + X(s_1) \rda \quad \text{by~\eqref{eq:Intro:Limit:ODE}}\\
 & \leq \sup_{s\in [s_1,s_2]} \int_{s_1}^s L_\sigma \lda K(r) X(r) + b(r)\rda \, \dd r + \| X(s_1) \| \, \text{as } \sigma \text{ is Lipschitz and } \sigma(0) = 0\\
 & \leq \sup_{s\in [s_1,s_2]} L_\sigma (s-s_1) \|X\|_{L^\infty([s_1,s])} \|K\|_{L^\infty} + L_\sigma (s-s_1) \|b\|_{L^\infty} + \|X(s_1)\| \\
 & = L_\sigma (s_2-s_1) \|X\|_{L^\infty([s_1,s_2])}\|K\|_{L^\infty} + L_\sigma(s_2-s_1)\|b\|_{L^\infty} + \|X(s_1)\|.
\end{align*}
Therefore if we choose $s_2 = \min\{1,s_1+\frac{1}{2L_\sigma \|K\|_{L^\infty}}\}$ we have $\|X\|_{L^\infty([s_1,s_2])} \leq 2L_\sigma \|b\|_{L^\infty} + 2\|X(s_1)\|$.
Let $s_i = \min\{1,\frac{i}{2L_\sigma \|K\|_{L^\infty}}\}$ and $N = \lceil 2L_\sigma \|K\|_{L^\infty}\rceil$ (we  note that $s_{N-1}<1 = s_N$).
For $i\in\{2,\dots,N\}$ we have
\begin{align*}
\|X\|_{L^\infty([0,s_i])} & \leq \max\lb \|X\|_{L^\infty([0,s_{i-1}])}, \|X\|_{L^\infty([s_{i-1},s_i])} \rb \\
 & \leq \max\lb \|X\|_{L^\infty([0,s_{i-1}])}, 2L_\sigma \|b\|_{L^\infty} + 2\|X(s_{i-1})\| \rb \\
 & \leq 2L_\sigma \|b\|_{L^\infty} + 2\|X\|_{L^\infty([0,s_{i-1}])}.
\end{align*}
For $i=1$ we have $\|X\|_{L^\infty([0,s_1])} \leq 2L_\sigma \|b\|_{L^\infty} +\|x\|$, and by induction for $i\in\{2,\dots, N\}$ we have
\begin{align*}
\|X\|_{L^\infty([0,s_i])} & \leq 2^i \|X\|_{L^\infty([0,s_1])} + 2L_\sigma \|b\|_{L^\infty} \sum_{k=0}^{i-1} 2^k \\
 & \leq 2^{i+1} L_\sigma \|b\|_{L^\infty} + 2^i \|x\| + 2(2^i-1)L_\sigma \|b\|_{L^\infty} \\
 & = 2(2^i-1) L_\sigma \|b\|_{L^\infty} + 2^i \|x\|.
\end{align*}
In particular, for $i=N$ we have
\[ \|X\|_{L^\infty([0,1])} \leq 2(2^N-1) L_\sigma \|b\|_{L^\infty} + 2^N \|x\|. \]
%By induction we have
%\[ \|X\|_{L^\infty([0,1])} \leq 2(2^N-1)L_\sigma \|b\|_{L^\infty} + 2^N\|x\| \]
%where $N = \lceil 2L_\sigma \|K\|_{L^\infty}\rceil$.
Now, using $\sigma(0) = 0$ and the Lipschitz assumption on $\sigma$, we have
\[ \|\dot{X}\|_{L^\infty} = \|\sigma(KX+b)\|_{L^\infty} \leq L_\sigma \l \|K\|_{L^\infty} \|X\|_{L^\infty} + \|b\|_{L^\infty} \r, \]
hence $X$ is Lipschitz.
Let $L_X$ be the Lipschitz constant for $X$.

Returning to~\eqref{eq:ProofConv:GammaEn:XdotDiff} we concentrate on the second term, we bound
\begin{align}
& \int_{\frac{i-1}{n}}^{\frac{i}{n}}\lda K(t)X(t) - K_{i-1}^{(n)}X_{i-1}^{(n)} \rda \, \dd t \notag \\
& \quad \quad \leq \int_{\frac{i-1}{n}}^{\frac{i}{n}} \lda X(t) - X_{i-1}^{(n)} \rda \lda K(t)\rda \, \dd t + \int_{\frac{i-1}{n}}^{\frac{i}{n}} \lda K(t) - K_{i-1}^{(n)} \rda \lda X_{i-1}^{(n)} \rda \, \dd t \notag \\
& \quad \quad \leq M_1 \int_{\frac{i-1}{n}}^{\frac{i}{n}} \lda X(t) - X_{i-1}^{(n)} \rda \, \dd t + \|X^{(n)}\|_{L^\infty(\mu_n)} \int_{\frac{i-1}{n}}^{\frac{i}{n}} \lda K(t) - K_{i-1}^{(n)} \rda \, \dd t \notag \\
& \quad \quad \leq M_2\l \int_{\frac{i-1}{n}}^{\frac{i}{n}} \lda X(t) - X_{i-1}^{(n)} \rda \, \dd t + \int_{\frac{i-1}{n}}^{\frac{i}{n}} \lda K(t) - K_{i-1}^{(n)} \rda \, \dd t \r \label{eq:ProofConv:GammaEn:XKCtstoDis}
\end{align}
where $M_2=\max\{M_1,\|X^{(n)}\|_{L^\infty(\mu_n)}\}$.
Continuing to manipulate the first term on the right hand side of the above expression we have,
\begin{align}
\int_{\frac{i-1}{n}}^{\frac{i}{n}} \lda X(t) - X_{i-1}^{(n)} \rda \, \dd t & \leq \int_{\frac{i-1}{n}}^{\frac{i}{n}} \lda X_{i-1}^{(n)} - X\l\frac{i-1}{n}\r \rda + \lda X\l\frac{i-1}{n}\r - X(t)\rda \, \dd t \notag \\
 & = \frac{1}{n} \lda X_{i-1}^{(n)} - X\l\frac{i-1}{n}\r \rda + \int_{\frac{i-1}{n}}^{\frac{i}{n}} \lda X\l\frac{i-1}{n}\r - X(t) \rda \, \dd t \notag \\
 & \leq \frac{1}{n} \lda X_{i-1}^{(n)} - X\l\frac{i-1}{n}\r \rda + L_X \int_{\frac{i-1}{n}}^{\frac{i}{n}} \l t - \frac{i-1}{n} \r \, \dd t  \notag\\
 & = \frac{1}{n} \lda X_{i-1}^{(n)} - X\l\frac{i-1}{n}\r \rda + \frac{L_X}{2n^2}. \label{eq:ProofConv:GammaEn:XCtstoDis}
\end{align}
Combining the bounds~\eqref{eq:ProofConv:GammaEn:XdotDiff}, \eqref{eq:ProofConv:GammaEn:XKCtstoDis} and~\eqref{eq:ProofConv:GammaEn:XCtstoDis} we have
\begin{align*}
\lda X\l\frac{i}{n}\r - X_i^{(n)} \rda & \leq \l 1+\frac{L_\sigma M_2}{n} \r \lda X\l\frac{i-1}{n}\r - X_{i-1}^{(n)} \rda + L_\sigma\int_{\frac{i-1}{n}}^{\frac{i}{n}} \lda b(t) - b_{i-1}^{(n)} \rda \, \dd t \\
 & \quad \quad \quad \quad + L_\sigma M_2 \int_{\frac{i-1}{n}}^{\frac{i}{n}} \lda K(t) - K_{i-1}^{(n)} \rda \, \dd t + \frac{L_X L_\sigma M_2}{2n^2}.
\end{align*}
By induction, for any $k\in \{0,1,\dots, n\}$, we have
\begin{align*}
& \lda X\l\frac{k}{n}\r - X_k^{(n)} \rda \leq L_\sigma \sum_{i=1}^n \l 1+\frac{L_\sigma M_2}{n}\r^{n-i} \int_{\frac{i-1}{n}}^{\frac{i}{n}} \| b(t) - b_{i-1}^{(n)}\| \, \dd t \notag\\
& \quad \quad \quad \quad + L_\sigma M_2 \sum_{i=1}^n \l 1+\frac{L_\sigma M_2}{n}\r^{n-i} \int_{\frac{i-1}{n}}^{\frac{i}{n}} \| K(t) - K_{i-1}^{(n)}\| \, \dd t \notag\\
& \quad \quad \quad \quad + \frac{L_X L_\sigma M_2}{2n^2} \sum_{i=1}^n \l 1+\frac{L_\sigma M_2}{n}\r^{n-i} \notag \\
& \quad \quad \leq \frac{\eps L_\sigma}{2n} \sum_{i=1}^n \l 1+\frac{L_\sigma M_2}{n}\r^{2(n-i)} + \frac{L_\sigma n}{2\eps} \sum_{i=1}^n \l \int_{\frac{i-1}{n}}^{\frac{i}{n}} \| b(t)-b_{i-1}^{(n)}\| \, \dd t \r^2 \notag \\
& \quad \quad \quad \quad + \frac{\eps L_\sigma M_2}{2n} \sum_{i=1}^n \l 1+\frac{L_\sigma M_2}{n}\r^{2(n-i)} + \frac{L_\sigma M_2 n}{2\eps} \sum_{i=1}^n \l \int_{\frac{i-1}{n}}^{\frac{i}{n}} \|K(t) - K_{i-1}^{(n)}\| \, \dd t\r^2 \notag \\
& \quad \quad \quad \quad + \frac{L_X L_\sigma M_2}{2n^2} \sum_{i=0}^{n-1} \l 1+\frac{L_\sigma M_2}{n}\r^i \notag
\end{align*}
for any $\eps>0$ by employing Young's inequality, $\alpha\beta\leq \frac{\eps\alpha^2}{2}+\frac{\beta^2}{2\eps}$,  for appropriately chosen $\alpha$ and $\beta$, on the first two terms for the final inequality (notice that the right hand side is independent of $k$).
By H\"older's inequality and the assumption that $d_1(K^{(n)}, K) \to 0$  we have
\[ n\sum_{i=1}^n \l \int_{\frac{i-1}{n}}^{\frac{i}{n}} \|K(t) - K_{i-1}^{(n)}\| \, \dd t\r^2 \leq \sum_{i=1}^n \int_{\frac{i-1}{n}}^{\frac{i}{n}} \|K(t) - K_{i-1}^{(n)}\|^2 \, \dd t \to 0 \]
%by the the assumption that $K^{(n)}\to K$
(and similarly for the sequence $b^{(n)}$). Hence, to show
\begin{equation} \label{eq:ProofConv:GammaEn:XEndPts}
\sup_{k\in\{0,1,\dots, n\}} \lda X\l\frac{k}{n}\r - X_k^{(n)} \rda\to 0,
\end{equation}
it is enough to show (i) $\frac{1}{n^2} \sum_{i=0}^{n-1} \l 1+\frac{L_\sigma M_2}{n}\r^i \to 0$ and (ii) $\sup_n \frac{1}{n} \sum_{i=1}^n \l 1+\frac{L_\sigma M_2}{n}\r^{2(n-i)} < \infty$.

For (i) we have that
\[ 0\leq \frac{1}{n^2} \sum_{i=0}^{n-1} \l 1+\frac{L_\sigma M_2}{n}\r^i \leq \frac{1}{n} \l 1+ \frac{L_\sigma M_2}{n}\r^n \to 0. \]
And for (ii) we have
\[ \frac{1}{n} \sum_{i=1}^n \l 1+\frac{L_\sigma M_2}{n}\r^{2(n-i)} \leq \l 1+\frac{L_\sigma M_2}{n}\r^{2n} \to e^{2L_\sigma M_2}. \]
Hence, if we replace $\eps$ by a sequence $\eps_n$ that converges to zero sufficiently slowly and that satisfies
\[ \frac{1}{\eps_n} \sum_{i=1}^n \int_{\frac{i-1}{n}}^{\frac{i}{n}} \|K(t) - K_{i-1}^{(n)}\|^2 \, \dd t \to 0 \qquad \text{and} \qquad \frac{1}{\eps_n} \sum_{i=1}^n \int_{\frac{i-1}{n}}^{\frac{i}{n}} \|b(t) - b_{i-1}^{(n)}\|^2 \, \dd t \to 0, \]
then we have that~\eqref{eq:ProofConv:GammaEn:XEndPts} holds.

Finally,
\begin{align*}
\sup_{t\in [t_k,t_{k+1}]} \|X(t) - X_k^{(n)}\| & \leq \sup_{t\in [t_k,t_{k+1}]} \l \|X(t) - X(t_k)\| + \|X(t_k) - X_k^{(n)}\| \r \\
 & \leq \frac{L_X}{n} + \| X(t_k)-X_k^{(n)}\| \to 0
\end{align*}
where the convergence is uniform over $k\in\{0,1,\dots,n\}$ as required.
\end{proof}

We say that $\Theta\ni\theta_n = (K^{(n)},b^{(n)},W^{(n)},c^{(n)})\to\theta = (K,b,W,c)\in\Theta$ if $d_1(K^{(n)},K)\to0$, $d_2(b^{(n)},b)\to0$ and $W^{(n)}\to W$, $c^{(n)}\to c$ (where, since $W^{(n)}$ and $c^{(n)}$ are sequences in $\bbR^{d\times d}$ and $\bbR^d$,  we choose any norm induced-topology  for the latter).
The above result implies the following lemma.

\begin{lemma}
\label{lem:ProofConv:GammaEn:ConvEn}
In addition to the assumptions of Theorem~\ref{thm:MainRes:Conv} let $\Theta^{(n)}\ni\theta^{(n)}\to \theta\in \Theta$, with $\max\{\sup_{n\in \bbN} R_n^{(1)}(K^{(n)}),\sup_{n\in \bbN} R_n^{(2)}(b^{(n)})\}<+\infty$ and $x\in \bbR^d$, $y\in \bbR^m$, then
\[ \lim_{n\to \infty} E_n(\theta^{(n)};x,y) = E_\infty(\theta;x,y). \]
\end{lemma}

\begin{proof}
By continuity of $h$ and $\cL$ (in its first argument), convergence of $W^{(n)}\to W$, $c^{(n)}\to c$ and $X^{(n)}_n[x,K^{(n)},b^{(n)}]\to X(1;x,K,b)$ (with the latter following from Lemma~\ref{lem:ProofConv:GammaEn:ConvXn}) we can easily conclude the result.
\end{proof}

The following is a small generalisation of Theorem 10.55 in~\cite{leoni09}.
The difference between the results stated here and the result in~\cite{leoni09} is that here we treat sequences of functions $f_n$, whilst in~\cite{leoni09} $f_n=f$.
We also only state the result on the domain $[0,1]$ and for $L^2$ convergence (the result generalises to bounded sets in higher dimensions and $L^p$ convergence where $p>1$).

\begin{proposition}
\label{prop:ProofConv:GammaEn:ConvDer}
Let $f_n\in L^2([0,1]; \R^\kappa)$, $f\in L^2([0,1]; \R^\kappa)$ and $\eps_n \to 0^+$.
Assume that $f_n\to f$ in $L^2([0,1]; \R^\kappa)$.
If
\[ \liminf_{n\to \infty} \frac{1}{\eps_n^2} \int_{\eps_n}^1 \lda f_n(t) - f_n(t-\eps_n)\rda^2 \, \dd t < +\infty, \]
then $f\in H^1([0,1]; \R^\kappa)$ and
\[ \liminf_{n\to \infty} \frac{1}{\eps_n^2} \int_{\eps_n}^1 \|f_n(t) - f_n(t-\eps_n)\|^2 \, \dd t \geq \int_0^1 \|\dot{f}(t)\|^2 \, \dd t. \]
\end{proposition}

\begin{proof}
The strategy is to show the following two inequalities:
\begin{equation} \label{eq:ProofConv:GammaEn:FiniteDiff1}
\int_{\delta^\prime}^{1-\delta^\prime} \lda J_\delta \ast \tilde g (t) - J_\delta \ast \tilde g(t-\eps_n) \rda^2 \, \dd t \leq \int_{\eps_n}^1 \lda \tilde g(t) - \tilde g(t-\eps_n) \rda^2 \, \dd t,
\end{equation}
for any $\tilde g\in L^2([0,1]; \R^\kappa)$ and any $\delta,\delta^\prime>0$ that satisfy $\eps_n+\delta<\delta^\prime$, and where $J_\delta$ is a standard mollifier; and
\begin{equation} \label{eq:ProofConv:GammaEn:FiniteDiff2}
\int_{2\delta^\prime}^{1-2\delta^\prime} \lda \dot{g}(t) \rda^2 \, \dd t \leq \liminf_{n\to \infty} \frac{1}{\eps_n^2} \int_{2\delta^\prime}^{1-2\delta^\prime} \lda g_n(t) - g_n(t-\eps_n) \rda^2 \, \dd t
\end{equation}
for any $g,g_n \in C^\infty([\delta^\prime,1-\delta^\prime]; \R^\kappa)$ with $\dot{g}_n \to \dot{g}$ in $L^\infty([\delta^\prime,1-\delta^\prime]; \R^\kappa)$ and $\sup_n \| \ddot{g}_n \|_{L^\infty([\delta^\prime,1-\delta^\prime])}<\infty$.

Before we prove these two inequalities, we use them to prove the result of the lemma. We note that $\|\frac{\dd}{\dd t} J_\delta\ast f - \frac{\dd}{\dd t} J_\delta \ast f_n\|_{L^\infty([\delta^\prime,1-\delta^\prime])}\leq \|\frac{\dd}{\dd t} J_\delta\|_{L^2(\bbR)}\|f_n-f\|_{L^2([0,1])}$ and $\|\frac{\dd^2}{\dd t^2} J_\delta\ast f_n\|_{L^\infty([\delta^\prime,1-\delta^\prime])}\leq \|\frac{\dd^2}{\dd t^2} J_\delta\|_{L^2(\bbR)} \|f_n\|_{L^2([0,1])}$.
Therefore we may apply~\eqref{eq:ProofConv:GammaEn:FiniteDiff2} to $g=J_\delta\ast f$ and $g_n=J_\delta\ast f_n$.

To show existence of $\dot{f}\in L^2([2\delta^\prime,1-2\delta^\prime]; \R^\kappa)$ we assume the above inequalities hold, then by~\eqref{eq:ProofConv:GammaEn:FiniteDiff1} and the assumptions on $f_n$ there exists $M$ such that
\[ \liminf_{n\to \infty} \frac{1}{\eps_n^2} \int_{\delta'}^{1-\delta^\prime} \lda J_\delta \ast f_n(t) - J_\delta \ast f_n(t-\eps_n) \rda^2 \, \dd t \leq M. \]
Furthermore, by~\eqref{eq:ProofConv:GammaEn:FiniteDiff2} $\int_{2\delta^\prime}^{1-2\delta^\prime} \|\frac{\dd}{\dd t} J_\delta \ast f (t) \|^2 \dd t \leq M$.
In addition to the four standard properties of mollifiers which we recalled in the proof of Proposition~\ref{prop:ProofConv:Compact:SobEmb}, we list a fifth one here:
\begin{itemize}
 \item [(M5)] $J_\delta \ast \tilde g \to \tilde g$ in $L^2([0,1]; \R^\kappa)$, as $\delta \to 0^+$, for any $\tilde g \in L^2([0,1]; \R^\kappa)$ \cite[Theorem C.19]{leoni09}.
\end{itemize}
By (M5) and since $\frac{\dd}{\dd t} J_\delta \ast f$ is bounded in $L^2([2\delta^\prime,1-2\delta^\prime]; \R^\kappa)$, uniformly in $\delta$, there exists an $h\in L^2([2\delta^\prime,1-2\delta^\prime]; \R^\kappa)$ such that, after potentially passing to a subsequence, $J_\delta \ast f \to f$ and $\frac{\dd}{\dd t} (J_\delta \ast f) \weakto h$ in $L^2([0,1]; \R^\kappa)$, as $\delta \to 0^+$.
Therefore, for any differentiable $\varphi$ with compact support in $[2\delta^\prime,1-2\delta^\prime]$,
\[ \int_{2\delta^\prime}^{1-2\delta^\prime} \varphi h \ot \int_{2\delta^\prime}^{1-2\delta^\prime} \varphi \frac{\dd}{\dd t} (J_\delta \ast f) = - \int_{2\delta^\prime}^{1-2\delta^\prime} \frac{\dd}{\dd t} \varphi J_\delta \ast f \to - \int_{2\delta^\prime}^{1-2\delta^\prime} \dot{\varphi} f = \int_{2\delta^\prime}^{1-2\delta^\prime} \varphi \dot{f}. \]
Hence, $h = \dot{f}$ and in particular $\dot{f}\in L^2([2\delta^\prime,1-2\delta^\prime]; \R^\kappa)$.
Since $\frac{\dd}{\dd t} (J_\delta \ast f) = J_{\delta}\ast \dot{f}$, we can use again the convergence of mollifiers to infer $\frac{\dd}{\dd t} (J_\delta \ast f)\to \dot{f}$ (\emph{strongly}) in $L^2([2\delta^\prime,1-2\delta^\prime]; \R^\kappa)$. %\B{$J_\delta \ast f\weakto f$ (\emph{weakly}) in $H^1([2\delta^\prime,1-2\delta^\prime]; \R^\kappa)$ there exists a $L^2([2\delta^\prime,1-2\delta^\prime]; \R^\kappa)$-strongly converging subsequence. By uniqueness of limits} $\frac{\dd}{\dd t} (J_\delta \ast f)\to \dot{f}$ (\emph{strongly}) in $L^2([2\delta^\prime,1-2\delta^\prime]; \R^\kappa)$.

Applying~\eqref{eq:ProofConv:GammaEn:FiniteDiff2} followed by~\eqref{eq:ProofConv:GammaEn:FiniteDiff1} we have
\begin{align*}
\int_{2\delta^\prime}^{1-2\delta^\prime} \lda \frac{\dd}{\dd t} J_\delta \ast f(t) \rda^2 \, \dd t & \leq \liminf_{n\to \infty} \int_{2\delta^\prime}^{1-2\delta^\prime} \lda \frac{J_\delta \ast f_n(t) - J_\delta \ast f_n(t-\eps_n)}{\eps_n} \rda^2 \, \dd t \\
& \leq \liminf_{n\to\infty} \int_{\eps_n}^1 \lda \frac{f_n(t) - f_n(t-\eps_n)}{\eps_n} \rda^2 \, \dd t.
\end{align*}
By $L^2([2\delta^\prime,1-2\delta^\prime]; \R^\kappa)$ convergence of $\frac{\dd}{\dd t} J_\delta \ast f$ as $\delta\to 0$ we have
\[ \int_{2\delta^\prime}^{1-2\delta^\prime} \lda \dot{f}(t) \rda^2 \, \dd t \leq \liminf_{n\to\infty} \int_{\eps_n}^1 \lda \frac{f_n(t) - f_n(t-\eps_n)}{\eps_n} \rda^2 \, \dd t, \]
with the inequality above valid for any $\delta^\prime>0$ (since the additional constraint imposed by $\delta^\prime>\delta+\eps_n$ vanishes when taking $\delta\to0$ and $\eps_n\to0$).
Taking $\delta^\prime \to 0$ proves the lemma under the assumption of~\eqref{eq:ProofConv:GammaEn:FiniteDiff1} and~\eqref{eq:ProofConv:GammaEn:FiniteDiff2}.

To show~\eqref{eq:ProofConv:GammaEn:FiniteDiff1} we have, assuming $\delta +\eps_n< \delta^\prime$,
\begin{align*}
& \l \int_{\delta^\prime}^{1-\delta^\prime} \lda J_\delta \ast \tilde g(t) - J_\delta \ast \tilde g(t-\eps_n) \rda^2 \, \dd t \r^{\frac{1}{2}} \\
& \quad \quad \quad \quad \quad \quad = \l \int_{\delta^\prime}^{1-\delta^\prime} \lda \int_{-\delta}^{\delta} J_\delta(s) \ls \tilde g(t-s) - \tilde g(t-\eps_n-s) \rs \, \dd s \rda^2 \, \dd t \r^{\frac{1}{2}} \\
& \quad \quad \quad \quad \quad \quad \leq \int_{-\delta}^{\delta} J_\delta(s) \l \int_{\delta^\prime}^{1-\delta^\prime} \lda \tilde g(t-s) - \tilde g(t-\eps_n-s) \rda^2 \, \dd t \r^{\frac{1}{2}} \, \dd s \\
& \quad \quad \quad \quad \quad \quad \leq \int_{-\delta}^\delta J_\delta(s) \l \int_{\eps_n}^1 \lda \tilde g(t) - \tilde g(t-\eps_n) \rda^2 \, \dd t \r^{\frac{1}{2}} \, \dd s \\
& \quad \quad \quad \quad \quad \quad = \l \int_{\eps_n}^1 \lda \tilde g(t) - \tilde g(t-\eps_n)\rda^2 \, \dd t \r^\frac12,
\end{align*}
where the antepenultimate line follows from Minkowski's inequality for integrals.

For inequality~\eqref{eq:ProofConv:GammaEn:FiniteDiff2}, by Taylor's theorem we have
\[ g_n(t) - g_n(t-\eps_n) = \eps_n \dot{g}_n(t) + \eps_n^2\ddot{g}_n(z) \quad \quad \text{for some } z \in [t-\eps_n,t]. \]
Therefore, for $t\in [2\delta^\prime,1-2\delta^\prime]$, when $\eps_n<\delta^\prime$,
\[ \frac{\|g_n(t) - g_n(t-\eps_n)\|}{\eps_n} \geq \|\dot{g}_n(t) \| - \eps_n \|\ddot{g}_n\|_{L^\infty([\delta^\prime,1-\delta^\prime])}. \]
For any $\eta>0$ there exists $C_\eta>0$ such that $|a+b|^2\leq (1+\eta)|a|^2 + C_\eta |b|^2$ for any $a,b\in \bbR$ (a consequence of Young's inequality, and one can show that $C_\eta = 1+\frac{1}{\eta}$), hence
\[ \|\dot{g}_n(t)\|^2 \leq (1+\eta) \lda \frac{g_n(t) - g_n(t-\eps_n)}{\eps_n}\rda^2 + C_\eta \eps_n^2 \|\ddot{g}_n\|_{L^\infty([\delta^\prime,1-\delta^\prime])}^2. \]
In particular, by Lebesgue's dominated convergence theorem and $\sup_{n\in\bbN} \|\ddot{g}_n\|_{L^\infty([\delta^\prime,1-\delta^\prime])}<\infty$,
\begin{align*}
\int_{2\delta^\prime}^{1-2\delta^\prime} \|\dot{g}(t)\|^2 \, \dd t & = \lim_{n\to \infty} \int_{2\delta^\prime}^{1-2\delta^\prime} \|\dot{g}_n(t)\|^2 \, \dd t \\
 & \leq (1+\eta) \liminf_{n\to \infty} \int_{2\delta^\prime}^{1-2\delta^\prime} \lda \frac{g_n(t) - g_n(t-\eps_n)}{\eps_n} \rda^2 \, \dd t.
\end{align*}
Taking $\eta\to 0$ proves~\eqref{eq:ProofConv:GammaEn:FiniteDiff2}.
\end{proof}

By application of the preceeding lemma we can now prove the liminf inequality for the $\Gamma$-convergence of $\cE_n$.

\begin{lemma}
\label{lem:ProofConv:GammaEn:Liminf}
Under the assumptions of Theorem~\ref{thm:MainRes:Conv} let $\Theta^{(n)}\ni\theta^{(n)}\to \theta\in \Theta$, then,
\[ \liminf_{n\to \infty} \cE_n(\theta^{(n)}) \geq \cE_\infty(\theta). \]
\end{lemma}

\begin{proof}
Let $\theta^{(n)}=(K^{(n)},b^{(n)},W^{(n)},c^{(n)})$ and $\theta=(K,b,W,c)$.
We only need to consider the case when $\liminf_{n\to \infty} \cE_n(\theta^{(n)})<+\infty$.
Hence we assume that $\cE_n(\theta^{(n)})$ is bounded and therefore by the compactness property (Corollary~\ref{cor:ProofConv:Compact:Compact}) $\cE_\infty(\theta)<+\infty$.
We will show the following
\begin{align*}
\text{(A)} & \quad \lim_{n\to \infty} E_n(\theta^{(n)};x,y) = E_\infty(\theta;x,y), \\
\text{(B)} & \quad \liminf_{n\to \infty} R^{(1)}_n(K^{(n)}) \geq R^{(1)}_\infty(K), \text{ and} \\
\text{(C)} & \quad \liminf_{n\to \infty} R^{(2)}_n(b^{(n)}) \geq R^{(2)}_\infty(b).
\end{align*}

Indeed (A) holds by Lemma~\ref{lem:ProofConv:GammaEn:ConvEn} and since $R^{(1)}_n(K^{(n)})$, $R^{(2)}_n(b^{(n)})$ are uniformly (in $n$) bounded.

Parts (B) and (C) are analogous, so we only show (B).
Let $\tilde{K}^{(n)}(t) = K_i^{(n)}$ for $t\in \l \frac{i}{n},\frac{i+1}{n}\rs$, for $i=0,\dots, n-1$, then
\begin{align*}
\liminf_{n\to \infty} R_n^{(1)}(K^{(n)}) & = \liminf_{n\to \infty} \l n \sum_{i=1}^n \| K_i^{(n)} - K_{i-1}^{(n)}\|^2 + \tau_1\|K_0^{(n)}\|^2 \r \\
 & \geq \liminf_{n\to \infty} n^2 \int_{\frac{1}{n}}^1 \lda \tilde{K}^{(n)}(t) - \tilde{K}^{(n)}\l t-\frac{1}{n}\r \rda^2 \, \dd t + \tau_1 \liminf_{n\to \infty} \| K_0^{(n)}\|^2  \\
 & \geq \int_0^1 \|\dot{K}(t)\|^2 \, \dd t + \tau_1\| K(0)\|^2,
\end{align*}
with the last inequality holding as, by Proposition~\ref{prop:ProofConv:Compact:SobEmb}, $\tilde{K}_n\to K$ in $L^2([0,1])$. Hence  we may apply Proposition~\ref{prop:ProofConv:GammaEn:ConvDer}.
%by Proposition~\ref{prop:ProofConv:GammaEn:ConvDer} and Proposition~\ref{prop:ProofConv:Compact:SobEmb}.
\end{proof}

We now turn our attention to the recovery sequence.
For any $\theta\in \Theta$ we define a sequence $\theta^{(n)}\in \Theta^{(n)}$ by
\begin{align}
K^{(n)}_i & = n\int_{\frac{i}{n}}^{\frac{i+1}{n}} K(t) \, \dd t, & \text{for } i=0,\dots, n-1, \label{eq:ProofConv:GammaEn:Kn} \\
b^{(n)}_i & = n\int_{\frac{i}{n}}^{\frac{i+1}{n}} b(t) \, \dd t, & \text{for } i=0,\dots, n-1, \label{eq:ProofConv:GammaEn:bn} \\
W^{(n)} & = W, \label{eq:ProofConv:GammaEn:W} \\
c^{(n)} & = c. \label{eq:ProofConv:GammaEn:c}
\end{align}
The above sequence is our candidate recovery sequence.
We first show that $\theta^{(n)}\to \theta$ in the $TL^2$ topology.

\begin{lemma}
\label{lem:ProofConv:GammaEn:RecSeqConv}
Under the assumptions of Theorem~\ref{thm:MainRes:Conv} let $\theta=(K,b,W,c)\in \Theta$ and define $\theta^{(n)}=(K^{(n)},b^{(n)},W^{(n)},c^{(n)})\in\Theta^{(n)}$ by {\normalfont (}\ref{eq:ProofConv:GammaEn:Kn}{\normalfont -}\ref{eq:ProofConv:GammaEn:c}{\normalfont )}.
Then $\theta^{(n)}\to \theta$ in the $TL^2$ topology.
\end{lemma}

\begin{proof}
We show that $K^{(n)}\to K$; the argument for $b^{(n)}\to b$ is analogous and $W^{(n)}=W$, $c^{(n)}=c$ so there is nothing to show for these parts.
Let $\tilde{K}^{(n)}(t) = K_i^{(n)}$ for $t\in \ls\frac{i}{n},\frac{i+1}{n}\r$ for $i=0,\dots,n-1$ and $\tilde{K}^{(n)}(1) = K_{n-1}^{(n)}$.
Since $K\in H^1([0,1]; \R^{d\times d})$, by Morrey's inequality we have that $K\in \nobreak C^{0,\frac12}([0,1]; \R^{d\times d})$.
In particular, $\|K(s)-K(t)\|\leq L_K \sqrt{|t-s|}$ for some $L_K$.
So,
\begin{align*}
\| \tilde{K}^{(n)} - K\|_{L^2}^2 & = \sum_{i=0}^{n-1} \int_{\frac{i}{n}}^{\frac{i+1}{n}} \lda K_i^{(n)} - K(t) \rda^2 \, \dd t \\
 & = \sum_{i=0}^{n-1} \int_{\frac{i}{n}}^{\frac{i+1}{n}} \lda n \int_{\frac{i}{n}}^{\frac{i+1}{n}} K(s) - K(t) \, \dd s \rda^2 \, \dd t \\
 & \leq n \sum_{i=0}^{n-1} \int_{\frac{i}{n}}^{\frac{i+1}{n}} \int_{\frac{i}{n}}^{\frac{i+1}{n}} \lda K(s) - K(t) \rda^2 \, \dd s \, \dd t \\
 & \leq L_K^2 n \sum_{i=0}^{n-1} \int_{\frac{i}{n}}^{\frac{i+1}{n}} \int_{\frac{i}{n}}^{\frac{i+1}{n}} | s-t| \, \dd s \, \dd t \\
 & = \frac{L_K^2}{3n} \to 0.
\end{align*}
Therefore $K^{(n)}\to K$.
\end{proof}

We now prove that the sequence from Lemma~\ref{lem:ProofConv:GammaEn:RecSeqConv} is a recovery sequence.

\begin{lemma}
\label{lem:ProofConv:GammaEn:Limsup}
Under the assumptions of Theorem~\ref{thm:MainRes:Conv} for any $\theta\in \Theta$ we define $\theta^{(n)}\in\Theta^{(n)}$ as in Lemma~\ref{lem:ProofConv:GammaEn:RecSeqConv}.
Then $\theta^{(n)}\to \theta$ and
\[ \limsup_{n\to \infty} \cE_n(\theta^{(n)}) \leq \cE_\infty(\theta). \]
\end{lemma}

\begin{proof}
Let $\theta=(K,b,W,c)\in \Theta$ and assume $\cE_\infty(\theta)<\infty$ (else the result is trivial).
By Lemma~\ref{lem:ProofConv:GammaEn:RecSeqConv} we already have that $\theta^{(n)}\to \theta$.

We show that $\theta^{(n)}$ is a recovery sequence.
It is enough to show the following.
\begin{align*}
\text{(A)} & \quad \lim_{n\to \infty} E_n(\theta^{(n)};x,y) = E_\infty(\theta;x,y), \\
\text{(B)} & \quad \limsup_{n\to \infty} R^{(1)}_n(K^{(n)}) \leq R^{(1)}_\infty(K), \text{ and} \\
\text{(C)} & \quad \limsup_{n\to \infty} R^{(2)}_n(b^{(n)}) \leq R^{(2)}_\infty(b).
\end{align*}
Part (A) follows from Lemma~\ref{lem:ProofConv:GammaEn:ConvEn} once we show parts (B) and (C).
Since (B) and (C) are analogous we only show (B).

Let $\eps>0$ and $C_\eps=1+\frac{1}{\eps}$, then $\|a+b\|^2\leq (1+\eps)\|a\|^2 + C_\eps\|b\|^2$ (as a consequence of Young's inequality).
So,
\begin{align*}
R^{(1)}_n(K^{(n)}) & = n\sum_{i=1}^{n-1} \lda n \int_{\frac{i}{n}}^{\frac{i+1}{n}} K(t) - K\l t-\frac{1}{n}\r \, \dd t \rda^2 + \tau_1 \lda n \int_0^{\frac{1}{n}} K(t) \, \dd t \rda^2 \\
 & \leq n^2 \sum_{i=1}^{n-1} \int_{\frac{i}{n}}^{\frac{i+1}{n}} \lda K(t) - K\l t-\frac{1}{n}\r \rda^2 \, \dd t + (1+\eps) \tau_1 \|K(0)\|^2 \\
 & \quad \quad \quad \quad + C_\eps \tau_1 n \int_0^{\frac{1}{n}} \|K(t)-K(0)\|^2 \, \dd t \\
 & \leq n^2 \int_{\frac{1}{n}}^1 \lda K(t) - K\l t-\frac{1}{n}\r \rda^2 \, \dd t + (1+\eps) \tau_1 \|K(0)\|^2 + C_\eps L_K^2 \tau_1 n \int_0^{\frac{1}{n}} t \, \dd t \\
 & \leq \int_0^1 \lda \dot{K}(t) \rda^2 \, \dd t + (1+\eps) \tau_1 \|K(0)\|^2 + \frac{C_\eps L_K^2 \tau_1}{2n},
\end{align*}
where the last line follows from~\cite[Theorem 10.55]{leoni09} (we note that we cannot use Proposition~\ref{prop:ProofConv:GammaEn:ConvDer} here, since it gives the lower bound $\liminf_{n\to \infty} n^2 \int_{\frac{1}{n}}^1 \lda K(t) - K\l t-\frac{1}{n}\r \rda^2 \, \dd t\geq \int_0^1 \lda \dot{K}(t) \rda^2 \, \dd t$, rather than an upper bound).
Taking $n\to \infty$ we have
\[ \limsup_{n\to \infty} R^{(1)}_n(K^{(n)}) \leq \int_0^1 \lda \dot{K}(t) \rda^2 \, \dd t + (1+\eps) \tau_1 \|K(0)\|^2 \leq (1+\eps) R^{(1)}_\infty(K). \]
Taking $\eps\to 0^+$ proves (B).
\end{proof}

\subsection{Regularity of Minimizers \label{subsec:ProofConv:Reg}}

The aim of this section is to show the higher regularity (i.e. $H^2_{\loc}$ rather than $H^1$) of minimisers.
The strategy is to apply elliptic regularity techniques.
For this we need to compute the Euler--Lagrange equation for $\cE_\infty$.
We start by showing how the finite layer model~\eqref{eq:Intro:ResNet:Rec} behaves when the parameters $K^{(n)}$ and $b^{(n)}$ are perturbed.
By taking the limit $n\to \infty$ we can then infer the corresponding result for the ODE limit~\eqref{eq:Intro:Limit:ODE}.

\begin{lemma}
\label{lem:ProofConv:Reg:NNPert}
Let $n\in\bbN$, $t_i=\frac{i}{n}$, $\mu_n=\frac{1}{n} \sum_{i=1}^n \delta_{t_i}$ and $K^{(n)},L^{(n)}\in L^2(\mu_n;\bbR^{d\times d})$ and $b^{(n)},\beta^{(n)}\in L^2(\mu_n;\bbR^{d})$.
Assume
\[ \max\lb R_n^{(1)}(K^{(n)}), R_n^{(1)}(L^{(n)}), R_n^{(2)}(b^{(n)}), R_n^{(2)}(\beta^{(n)}) \rb \leq C \]
where $R^{(j)}_n$, $j=1,2$ are defined in Section~\ref{subsec:Intro:Reg} with $\tau_i>0$.
Furthermore, we assume that $\sigma\in C^2$, $\sigma(0)=0$, and $\sigma$ acts componentwise.
Let $\theta^{(n)} = (K^{(n)},b^{(n)})$ and $\xi^{(n)} = (L^{(n)},\beta^{(n)})$ and define $X_i^{(n)}[x;\theta^{(n)}]$, $i\in\{0,\dots,n-1\}$, as a solution to~\eqref{eq:Intro:ResNet:Rec} with initial condition $X_0^{(n)}=x$.
We define, for $r>0$ and $i\in\{0,\dots,n-1\}$,
\begin{equation} \label{eq:ProofConv:Reg:Dn}
D_{r,i}^{(n)}(x,\theta^{(n)},\xi^{(n)}) = \frac{1}{r} \l X_i^{(n)}[x;\theta^{(n)}+r\xi^{(n)}] - X_i^{(n)}[x;\theta^{(n)}]\r.
\end{equation}
Then,
%\begin{equation} \label{eq:ProofConv:Reg:NNPert}
%\begin{array}{l}
\begin{align}
D_{r,n}^{(n)}(x,\theta^{(n)},\xi^{(n)}) & = %\displaystyle
\frac{1}{n} \sum_{i=0}^{n-1} \Bigg\{ \ls \prod_{j=i+1}^{n-1} \l \Id + \frac{1}{n} \dot{\sigma}\l K_j^{(n)} X_j^{(n)}[x;\theta^{(n)}] + b_j^{(n)} \r \odot K_j^{(n)} \r \rs \notag \\
 & \qquad \times \l \ls L_i^{(n)}X_i^{(n)}[x;\theta^{(n)}]+\beta_i^{(n)}\rs \odot \dot{\sigma}\l K_i^{(n)} X_i^{(n)}[x;\theta^{(n)}] + b_i^{(n)} \r \r \Bigg\} + O(r). \label{eq:ProofConv:Reg:NNPert}
\end{align}
%\end{array}
%\end{equation}
where the $O(r)$ term depends on $K^{(n)},L^{(n)},b^{(n)},\beta^{(n)}$ only through the parameter $C$ and does not depend on $n$ in any other way.
\end{lemma}

\begin{remark}
\label{rem:ProofConv:Reg:Commute}
Not that for vectors $A,C\in \bbR^d$ and a matrix $B\in \bbR^{d\times d}$ we have $[BC]\odot A = A\odot[BC] = [A\odot B] C = [B\odot A] C$ where $A\odot B$ is understood to be taken componentwise in each row, i.e. $(A\odot B)_{ij} = A_i B_{ij}$.
Hence, the usual matrix multiplication $\times$ and componentwise multiplication $\odot$ commute.
\end{remark}

\begin{proof}[Proof of Lemma~\ref{lem:ProofConv:Reg:NNPert}.]
Since $\theta^{(n)}$, $\xi^{(n)}$ and $x$ are fixed, we may shorten our notation by writing
\begin{align}
D_{r,i}^{(n)} & = D_{r,i}^{(n)}(x_s,\theta^{(n)},\xi^{(n)}), \label{eq:ProofConv:Reg:Dnshort} \\
X_i^{(n)}(r) & = X_i^{(n)}[x;\theta^{(n)}+r\xi^{(n)}], \quad \text{and} \label{eq:ProofConv:Reg:Xnrshort} \\
X_i^{(n)} & = X_i^{(n)}(0) \label{eq:ProofConv:Reg:Xn0short}
\end{align}
throughout the proof.

Fix $i\in\{0,1,\dots,n\}$.
Then, where we understand the square of the brackets below to be taken componentwise,
\begin{align}
D_{r,i}^{(n)} & = \frac{1}{rn} \Bigg( \sigma\l \l K_{i-1}^{(n)}+rL_{i-1}^{(n)}\r X_{i-1}^{(n)}(r) + b_{i-1}^{(n)}+r\beta_{i-1}^{(n)} \r - \sigma\l K_{i-1}^{(n)} X_{i-1}^{(n)} + b_{i-1}^{(n)} \r \Bigg) \notag \\
 & \hspace{2cm} + D_{r,i-1}^{(n)} \notag \\
 & = \frac{1}{rn} \ls \l K_{i-1}^{(n)}+rL_{i-1}^{(n)} \r X_{i-1}^{(n)}(r) + r\beta_{i-1}^{(n)} - K_{i-1}^{(n)}X_{i-1}^{(n)} \rs \odot \dot{\sigma}\l K_{i-1}^{(n)} X_{i-1}^{(n)} + b_{i-1}^{(n)} \r \notag \\
 & \hspace{2cm} + \frac{1}{2rn} \ls \l K_{i-1}^{(n)}+rL_{i-1}^{(n)} \r X_{i-1}^{(n)}(r) + r\beta_{i-1}^{(n)} - K_{i-1}^{(n)}X_{i-1}^{(n)} \rs^2 \odot \ddot{\sigma}(\xi_i) + D_{r,i-1}^{(n)} \notag \\
 & = \frac{1}{n} \ls K_{i-1}^{(n)} D_{r,i-1}^{(n)} + L_{i-1}^{(n)} X_{i-1}^{(n)} + \beta_{i-1}^{(n)} \rs \odot \dot{\sigma}\l K_{i-1}^{(n)} X_{i-1}^{(n)} + b_{i-1}^{(n)} \r \notag \\
 & \hspace{2cm} + \frac{r}{2n} \ls K_{i-1}^{(n)} D_{r,i-1}^{(n)} + L_{i-1}^{(n)} X_{i-1}^{(n)}(r) + \beta_{i-1}^{(n)} \rs^2 \odot \ddot{\sigma}(\xi_i) + D_{r,i-1}^{(n)} \label{eq:ProofConv:Reg:DeltaRec}
\end{align}
where the first equality follows from the definitions of $D_{r,i}^{(n)}$, $X_{i-1}^{(n)}(r)$ and $X_{i-1}^{(n)}$, the second equality follows from Taylor's theorem for some $\xi_i\in\bbR^d$, and the third equality follows from the definition of $D_{r,i-1}^{(n)}$.
We can bound $\xi_i$ by
\begin{align*}
\xi_i & \geq \min\lb K_{i-1}^{(n)} X_{i-1}^{(n)} + b_{i-1}^{(n)}, \l K_{i-1}^{(n)} + r L_{i-1}^{(n)} \r X_{i-1}^{(n)}(r) + b_{i-1}^{(n)} + r\beta_{i-1}^{(n)} \rb \\
\xi_i & \leq \max\lb K_{i-1}^{(n)} X_{i-1}^{(n)} + b_{i-1}^{(n)}, \l K_{i-1}^{(n)} + r L_{i-1}^{(n)} \r X_{i-1}^{(n)}(r) + b_{i-1}^{(n)} + r\beta_{i-1}^{(n)} \rb
\end{align*}
where we understand the inequalities, minimum and maximum to hold componentwise.
%for some $t_i\in\bbR^d$ and where we understand the square of the above brackets to be taken componentwise.

By Lemma~\ref{lem:ProofConv:GammaEn:ConvXn} $X^{(n)}$, $X^{(n)}(r)$ are uniformly bounded by a constant depending only on $C$ (for $r\leq 1$ say), and so we can assume $\xi_i$ is uniformly bounded independent of $i$ and $n$.
Hence if we can show that $\sup_{r\in (0,1]} \sup_{i\in\{0,1,\dots, n\}} \|D_{r,i}^{(n)}\|\leq C^\prime$ where $C^\prime$ depends only on $C$, in particular is independent of $n$, then
\[ D_{r,i}^{(n)} = D_{r,i-1}^{(n)} + \frac{1}{n} \ls K_{i-1}^{(n)} D_{r,i-1}^{(n)} + L_{i-1}^{(n)} X_{i-1}^{(n)} + \beta_{i-1}^{(n)} \rs \odot \dot{\sigma}\l K_{i-1}^{(n)} X_{i-1}^{(n)} + b_{i-1}^{(n)} \r + O\l\frac{r}{n}\r. \]
By induction the above implies~\eqref{eq:ProofConv:Reg:NNPert}.

We are left to show that $D^{(n)}_{r,i}$ is uniformly bounded in $i$ and $r$.
From~\eqref{eq:ProofConv:Reg:DeltaRec} we may infer the existence of constants $c_1$ and $c_2$, that are independent of $r$ and $n$ and, given $C$ can also be made independent of $K^{(n)},L^{(n)},b^{(n)},\beta^{(n)}$, such that
\[ \|D_{r,i}^{(n)}\| \leq \l \frac{c_1(1+r)}{n} + 1 \r \|D_{r,i-1}^{(n)}\| + \frac{c_2}{n}. \]
Hence, by induction,
\[ \|D_{r,i}^{(n)}\| \leq \sum_{k=0}^{i-1} \l 1+\frac{c_1(1+r)}{n}\r^k \frac{c_2}{n} \leq c_2\l 1+\frac{c_1(1+r)}{n}\r^n \to c_2 e^{c_1(1+r)} \quad \text{as } n\to \infty. \]
It follows that $\sup_{r\in (0,1]} \sup_{i\in\{0,1,\dots, n\}} \|D_{r,i}^{(n)}\|$ can be bounded as claimed.
\end{proof}

We now use the above result to deduce the behaviour of the output of the ODE model~\eqref{eq:Intro:Limit:ODE} when the parameters $K$ and $b$ are perturbed.

\begin{lemma}
\label{lem:ProofConv:Reg:ODEPert}
Assume $\sigma\in C^2$, $\sigma(0)=0$ and $\sigma$ acts componentwise.
Let $\theta=(K,b)$ and $\xi=(L,\beta)$ where $K,L\in H^1([0,1];\bbR^{d\times d})$ and $b,\beta\in H^1([0,1];\bbR^d)$.
Furthermore, let $X(t;x,\theta)$ be defined as a solution to~\eqref{eq:Intro:Limit:ODE} for the input $\theta$ and initial condition $X(0)=x$.
Define, for $r>0$,
\begin{equation} \label{eq:ProofConv:Reg:D}
D_r(t;x,\theta,\xi) = \frac{1}{r} \l X(t;x,\theta+r\xi) - X(t;x,\theta)\r.
\end{equation}
Then,
\begin{align*}
\lim_{r\to 0^+} D_r(1;x,\theta,\xi) & = \int_0^1 \Bigg[ \exp\l \int_t^1 \dot{\sigma} \l K(s)X(s;x,\theta)+b(s) \r \odot K(s) \, \dd s \r \\
 & \hspace{2cm} \times \l L(t)X(t;x,\theta) + \beta(t) \r \odot \dot{\sigma}\l K(t)X(t;x,\theta)+b(t) \r \Bigg]\, \dd t.
\end{align*}
\end{lemma}

\begin{proof}
Let $K^{(n)},L^{(n)},b^{(n)},\beta^{(n)}$ be any discrete sequences converging to $K,L,b,\beta$ respectively with
\begin{equation} \label{eq:ProofConv:Reg:RegBound}
\sup_{n\in \bbN} \max\lb R^{(1)}_n(K^{(n)}),R^{(1)}_n(L^{(n)}),R^{(2)}_n(b^{(n)}),R^{(2)}_n(\beta^{(n)}) \rb < + \infty
\end{equation}
and where the convergence is uniform:
\[ \max_{i\in \{0,1,\dots,n-1\}} \sup_{t\in [t_i,t_{i+1}]} \max\lb \| K_i^{(n)}-K(t)\|,\| L_i^{(n)}-L(t)\|,\| b_i^{(n)}-b(t)\|,\| \beta_i^{(n)}-\beta(t)\| \rb \to 0. \]
For example, the recovery sequences, as defined by~\eqref{eq:ProofConv:GammaEn:Kn} and~\eqref{eq:ProofConv:GammaEn:bn}, are sufficient.
To shorten notation we write
\begin{align*}
D_r & = D_r(1;x,\theta^{(n)},\xi^{(n)}) \\
X_r(t) & = X(t;x,\theta+r\xi) \\
X(t) & = X_0(t)
\end{align*}
and again use the abbreviations in~\eqref{eq:ProofConv:Reg:Dnshort}-\eqref{eq:ProofConv:Reg:Xn0short} where $D_{r,i}^{(n)}(x,\theta,\xi)$ is defined by~\eqref{eq:ProofConv:Reg:Dn}.

By Lemma~\ref{lem:ProofConv:GammaEn:ConvXn} we have $X_n^{(n)}(r) \to X_r(1)$ as $n\to \infty$ for all $r\geq 0$.
Hence, $\lim_{n\to \infty} D^{(n)}_{r,n} = D_r$.
By Lemma~\ref{lem:ProofConv:Reg:NNPert} (and in particular using that the $O(r)$ term in~\eqref{eq:ProofConv:Reg:NNPert} is independent of $n$ given the bound~\eqref{eq:ProofConv:Reg:RegBound} on $K^{(n)},L^{(n)},b^{(n)},\beta^{(n)}$) we have that
\[ \lim_{r\to 0^+} D_r = \lim_{r\to 0^+} \lim_{n\to \infty} D_{r,n}^{(n)} = \lim_{r\to 0^+} \lim_{n\to\infty} \l \frac{1}{n} \sum_{i=0}^{n-1} A_i^{(n)} B_i^{(n)} + O(r) \r = \lim_{n\to \infty} \frac{1}{n} \sum_{i=0}^{n-1} A_i^{(n)} B_i^{(n)} \]
where
\begin{align*}
A_i^{(n)} & = \prod_{j=i+1}^{n-1} \l \Id + \frac{1}{n} \dot{\sigma}\l K_j^{(n)} X_j^{(n)} + b_j^{(n)} \r \odot K_j^{(n)} \r \text{ and} \\
B_i^{(n)} & =  \ls L_i^{(n)}X_i^{(n)}+\beta_i^{(n)}\rs \odot \dot{\sigma}\l K_i^{(n)} X_i^{(n)} + b_i^{(n)} \r.
\end{align*}
Convergence in $TL^\infty$ implies convergence of the empirical integral, i.e. a relatively standard argument implies that,  if $\max_{i\in\{0,1,\dots,n-1\}} \sup_{t\in [t_i,t_{i+1}]} \|F(t)-F_n(t_i)\| \to 0$,  then $\frac{1}{n} \sum_{i=0}^{n-1} F_n(t_i) \to \int_0^1 F(t) \, \dd t$ (with the result also being true for weaker assumptions, e.g. convergence in $TL^1$).
By assumptions on the sequences $K^{(n)},L^{(n)},b^{(n)},\beta^{(n)}$ we easily have that
\[ \max_{i\in\{0,1,\dots,n-1\}} \sup_{t\in [t_i,t_{i+1}]} \lda B_i^{(n)} - \ls L(t)X(t)+\beta(t)\rs \odot \dot{\sigma}(K(t)X(t)+b(t)) \rda \to 0. \]
We are left to find the uniform limit of $A^{(n)}_i$.

%We make the easily verifiable claim that
If $\max_{i\in\{0,1,\dots,n-1\}} \sup_{t\in [t_i,t_{i+1}]} \|F(t)-F_n(t_i)\|\leq \eps$ and $\|F\|_{L^\infty}\leq M$,  then
\begin{align*}
\lda \frac{1}{n} \sum_{i=\lfloor tn\rfloor +1}^{n-1} F_n(t_i) - \int_t^1 F(s) \, \dd s \rda & \leq \int_t^{t_{\lfloor tn\rfloor +1}} \|F\|\, \dd s + \sum_{i=\lfloor tn\rfloor +1}^{n-1} \int_{t_i}^{t_{i+1}} \| F_n(s) - F(s)\| \, \dd s \\
 & \leq \eps + \frac{M}{n},
\end{align*}
for any $t\in [0,1]$.
Hence,
\begin{align*}
& \lda \prod_{i=\lfloor tn\rfloor+1}^{n-1} \exp\l \frac{1}{n} F_n(t_i)\r - \exp\l \int_t^1 F(s) \, \dd s \r \rda \\
& \hspace{1cm} = \lda \exp\l\frac{1}{n}\sum_{i=\lfloor tn\rfloor+1}^{n-1} F_n(t_i)\r - \exp\l \int_t^1 F(s) \, \dd s \r \rda \\
& \hspace{1cm} \leq \lda \frac{1}{n}\sum_{i=\lfloor tn\rfloor+1}^{n-1} F_n(t_i) - \int_t^1 F(s) \, \dd s \rda e^{\frac{1}{n}\sum_{i=\lfloor tn\rfloor+1}^{n-1} F_n(t_i) - \int_t^1 F(s) \, \dd s} e^{\int_t^1 F(s) \, \dd s} \\
& \hspace{1cm} \leq \l \eps + \frac{M}{n} \r e^{M+\eps+\frac{M}{n}}
\end{align*}
using the inequality $\| e^{X+Y}-e^X\|\leq \|Y\| e^{\|X\|} e^{\|Y\|}$ (for any square matrices $X$, $Y$; see Appendix~\ref{sec:App:Matrices}) applied to $X=\int_t^1 F(s) \, \dd s$ and $Y=\frac{1}{n}\sum_{i=\lfloor tn\rfloor+1}^{n-1} F_n(t_i)-\int_t^1 F(s) \, \dd s$.
We define $F_n:\{t_j\}_{j=0}^{n-1}\to \bbR^{d\times d}$ and $F:[0,]\to \bbR^{d\times d}$ by
\begin{align*}
F_n(t_j) & = \log\l \Id + \frac{1}{n} C_j^{(n)}\r^n, & C_j^{(n)} & = \dot{\sigma}\l K_j^{(n)}X_j^{(n)}+b_j^{(n)}\r \odot K_j^{(n)}, \\
F(s) & = C(s), \text{ and} & C(s) & = \dot{\sigma}\l K(s)X(s)+b(s)\r \odot K(s).
\end{align*}
By construction $\prod_{j=i+1}^{n-1} \exp\l\frac{1}{n}F_n(t_j)\r = A_i^{(n)}$.
The $L^\infty$ bound, $M$, on $F$ is readily verified from the $L^\infty$ bounds on each of $K$, $X$ and $b$.
We show the uniform convergence of $F_n$ to $F$ shortly.
For now we assume this is true, so we can fix an arbitrary $\eps>0$ and have an $N$ such that
\begin{equation} \label{eq:ProofConv:Reg:WantClogeCjnBound}
\max_{i\in\{0,1,\dots,n-1\}} \sup_{t\in [t_i,t_{i+1}]} \|F(t)-F_n(t_i)\| = \max_{i\in\{0,1,\dots, n-1\}} \sup_{t\in [t_i,t_{i+1}]} \lda C(s) - n\log\l \Id+\frac{1}{n}C_j^{(n)}\r \rda \leq \eps,
\end{equation}
for all $n\geq N$.
For $t\in [t_i,t_{i+1}]$ we have $\lfloor tn\rfloor = i$, and so %, assuming the uniform convergence of $F_n$ to $F$,
\[ \max_{i\in\{0,1,\dots, n-1\}} \sup_{t\in [t_i,t_{i+1}]} \lda A_i^{(n)} - \exp\l \int_t^1 F(s) \, \dd s \r \rda \leq \l \eps+\frac{M}{n} \r e^{M+\eps+\frac{M}{n}}. \]
Hence $A_i^{(n)}$ converges uniformly to $\exp\l \int_t^1 F(s) \, \dd s \r$.

To complete the proof, we show that \eqref{eq:ProofConv:Reg:WantClogeCjnBound} holds.
Analogously to when we considered the sequence $B_n$, we can infer the existence of $N$ such that, if $n\geq N$, then
\begin{equation} \label{eq:ProofConv:Reg:CCjnBound}
\max_{i\in\{0,1,\dots, n-1\}} \sup_{t\in [t_i,t_{i+1}]} \lda C(s) - C_j^{(n)} \rda \leq \eps.
\end{equation}
By~\cite[Proposition 2.9]{hall03}, there exists a constant $c$ (independent of all parameters) such that (assuming $\|C_j^{(n)}\|\leq \frac{n}{2}$)
\begin{align}
\lda C(s) - n\log \l\Id+\frac{1}{n} C_j^{(n)}\r \rda & \leq \lda C(s) - C_j^{(n)}\rda + n \lda \frac{1}{n} C_j^{(n)} - \log\l\Id+\frac{1}{n}C_j^{(n)}\r \rda \notag \\
 & \leq \lda C(s) - C_j^{(n)}\rda + \frac{c}{n} \|C_j^{(n)}\|^2. \label{eq:ProofConv:Reg:ClogeCjnBound}
\end{align}
Since $\|C_j^{(n)}\|$ is uniformly bounded in $j$ and $n$, \eqref{eq:ProofConv:Reg:CCjnBound} and~\eqref{eq:ProofConv:Reg:ClogeCjnBound} imply~\eqref{eq:ProofConv:Reg:WantClogeCjnBound}.
\end{proof}

The previous result shows the limit $\lim_{r\to0^+} D_r(t;x,\theta,\xi)$ exists for $t=1$.
Whilst this is all we require in the sequel, we note that a  rescaling argument implies that the limit exists for all $t>0$.
In particular, if we fix $t>0$ and let $\hat{X}(\cdot;x,\hat{\theta})$ satisfy $\frac{\dd}{\dd s}\hat{X}(s) = \hat{\sigma}(\hat{K}(s)\hat{X}(s)+\hat{b}(s))$ where $\hat{\sigma}(\cdot)=t\sigma(\cdot)$ and $\hat{\theta} = (\hat{K}(\cdot),\hat{b}(\cdot)) = (K(\cdot t),b(\cdot t))$, then we can apply the above lemma directly to $\hat{X}$ to deduce the existence of $\lim_{r\to 0^+} \hat{D}_r(1;x,\hat{\theta},\hat{\xi})$, where $\hat{D}_r(s;x,\hat{\theta},\hat{\xi}) = \frac{1}{r} (\hat{X}(s;x,\hat{\theta}+r\xi) - \hat{X}(s;x,\hat{\theta}))$.
Because $\hat{X}(s;x,\hat{\theta}) = X(st;x,\theta)$, we have $\hat{D}_r(1;x,\hat{\theta},\hat{\xi}) = D_r(t;x,\theta,\xi)$.

Using the above result we can compute the G\^{a}teaux derivative of $\cE_\infty$, defined as
\[ \dd \cE_\infty(\theta;\xi) = \lim_{r\to 0^+} \frac{\cE_\infty(\theta+r\xi)-\cE_\infty(\theta)}{r}. \]

\begin{lemma}
\label{prop:ProofConv:Reg:DerEinfty}
Define $\cE_\infty$, $E_\infty$, $R_\infty^{(i)}$, for $i=1,2$, and $R^{(j)}$, for $j=3,4$, as in Sections~\ref{subsec:Intro:Reg}-\ref{subsec:Intro:Limit}.
In addition to the assumptions in Lemma~\ref{lem:ProofConv:Reg:ODEPert} we assume that $h\in C^2(\R^m; \R^m)$, $\cL(\cdot,y) \in C^2(\R^m; \R)$ for all $y\in \R^m$, and all norms $\|\cdot\|$ on $\bbR^d$ and $\bbR^{d\times d}$ are induced by inner products.
Let $\{(x_s,y_s)\}_{s=1}^S\subset\bbR^d\times \bbR^m$, $\theta=(K,b,W,c)\in\Theta$ and $\xi=(L,\beta,V,\gamma)\in\Theta$ where $\Theta$ is given by~\eqref{eq:Prelim:Top:Theta}.
We define $D_r(t;x,\theta,\xi)$ by~\eqref{eq:ProofConv:Reg:D} for $r>0$ and
\[ D_0(t;x,\theta,\xi) = \lim_{r\to 0^+} D_r(t;x,\theta,\xi). \]
Then,
\begin{align*}
& \dd\cE_\infty(\theta;\xi) = \sum_{s=1}^S \nabla_z \cL(h(WX(1;x_s,\theta)+c), y_s) \cdot \Bigg[\dot{h}(WX(1;x_s,\theta)+c) \\
& \hspace{3cm} \odot (WD_0(1;x_s,\theta,\xi) + VX(1;x_s,\theta) + \gamma)\Bigg] \\
& \hspace{3cm} + \alpha_1\dd R_\infty^{(1)}(K;L) + \alpha_2\dd R_\infty^{(2)}(b;\beta) + \alpha_3\dd R^{(3)}(W;V) + \alpha_4\dd R^{(4)}(c;\gamma),
\end{align*}
where with a small abuse of notation we wrote $X(t;x,\theta) = X(t;x,K,b)$,
$\nabla_z$ is the derivative with respect to the first argument, and
\begin{align*}
\dd R_\infty^{(1)}(K;L) & = 2\langle \dot{K},\dot{L}\rangle_{L^2} + 2\tau_1 \langle K(0),L(0)\rangle, & \dd R^{(3)}(W;V) & = 2\langle W,V\rangle, \\
\dd R_\infty^{(2)}(b;\beta) & = 2\langle \dot{b},\dot{\beta}\rangle_{L^2} + 2\tau_2 \langle b(0),\beta(0)\rangle, & \dd R^{(4)}(c;\gamma) & = 2\langle c,\gamma\rangle.
\end{align*}
\end{lemma}

\begin{proof}
We consider the derivative of each term in $\cE_\infty$ separately.
For ease of notation let
\begin{align*}
X_r(t) & = X(t;x_s,\theta+r\xi) \\
X(t) & = X_0(t) \\
\delta h_r & = h((W+rV)X_r(1) + c+r\gamma) - h(WX(1)+c) \\
D_r(t) & = D_r(t,x_s,\theta,\xi).
\end{align*}
Applying this new notation to \eqref{eq:ProofConv:Reg:D}, we get $D_r(t) = \frac1r (X_r(t)-X(t))$.

By Taylor's theorem, for each $r$ there exists $z_r\in\bbR^m$ such that
\begin{align*}
& \cL(h(WX_r(1)+c+r\gamma),y_s) - \cL(h(WX(1)+c),y_s) \\
& \hspace{2cm} = \nabla_z\cL(h(WX(1)+c),y_s) \cdot \delta h_r + \frac12 (\delta h_r)^\top \nabla^2 \cL(z_r,y_s) \delta h_r.
\end{align*}
Similarly, by another application of Taylor's theorem, for each $r$ there exists $t_r\in\bbR^m$ such that
\begin{align*}
\delta h_r & = \dot{h}(WX(1)+c) \odot \left[(W+rV)X_r(1)+r\gamma - WX(1)\right] \\
 & \qquad + \frac12 \left[ (W+rV) X_r(1) + r\gamma - WX(1) \right]^2 \odot \ddot{h}(t_r) \\
 & = \dot{h}(WX(1)+c) \odot \left[ W(X_r(1)-X(1))+r(VX_r(1)+\gamma) \right] \\
 & \qquad + \frac12 \left[ W(X_r(1)-X(1)) + r (VX_r(1)+\gamma) \right]^2 \odot \ddot{h}(t_r) \\
 & = r\dot{h}(WX(1)+c) \odot \left[ WD_r(1)+ VX_r(1)+\gamma \right] + O(r^2),
\end{align*}
where the square is to be evaluated componentwise and we use that $h\in C^2$ and $t_r$ is bounded as function of $r$.
Hence,
\begin{align*}
& \frac{1}{r} \ls \cL(h(WX_r(1)+c+r\gamma),y_s) - \cL(h(WX(1)+c),y_s) \rs \\
& \hspace{1cm} = \nabla_z\cL(h(WX(1)+c),y_s) \cdot \ls \dot{h}(WX(1)+c) \odot \left( WD_r(1)+ VX_r(1)+\gamma \right) \rs + O(r).
\end{align*}

Taking $r\to 0$ and applying Lemma~\ref{lem:ProofConv:Reg:ODEPert} gives
%The first term follows by                                                                                                                                                                                                                                                                                                                                                                         Lemma~\ref{lem:ProofConv:Reg:ODEPert}, indeed, by Taylor's theorem
\begin{align*}
\dd E_\infty(\theta;x_s,y_s;\xi) & = \lim_{r\to 0^+} \frac{1}{r} %\ls E_\infty(\theta+r\xi;x_s,y_s) - E_\infty(\theta;x_s,y_s) \rs \\
\ls \cL(h(WX(1)+c+r\gamma),y_s) - \cL(h(WX(1)+c),y_s)\rs \\
% & = \nabla_z \cL(h(WX(1;x_s,\theta)+c),y_s) \cdot \Bigg[ \dot{h}(WX(1;x_s,\theta)+c) \\
% & \hspace{1cm} \odot \lim_{r\to 0^+} \l WD_r(1;x_s,\theta,\xi) + VX(1;x_s,\theta+r\xi) + \gamma \r \Bigg] \\
 & = \nabla_z \cL(h(WX(1;x_s,\theta)+c),y_s) \cdot \Bigg[ \dot{h}(WX(1;x_s,\theta)+c) \\
 & \hspace{1cm} \odot \l WD_0(1;x_s,\theta,\xi) + VX(1;x_s,\theta) + \gamma \r \Bigg].
\end{align*}

It is straightforward to show that the G\^{a}teaux derivatives of the regularisation functionals $R^{(1)}_\infty$, $R^{(2)}_\infty$, $R^{(3)}$, and $R^{(4)}$ are as claimed.
%This completes the proof.
Summing the individual terms completes the proof.
\end{proof}

Finally we can deduce the regularity of minimisers of $\cE_\infty$ by applying techniques from the study of elliptic differential equations (see for example~\cite[Section 2.2.2]{grisvard85} for the same techniques).

\begin{proof}[Proof of Proposition~\ref{prop:MainRes:Reg}]
Assume that $\theta=(K,b,W,c)\in\Theta$ be a minimiser of $\cE_\infty$.
We will show that $K\in H^2_{\loc}([0,1];\bbR^{d\times d})$ (the argument for $b\in H^2_{\loc}([0,1];\bbR^d)$ is analogous).

Since $\theta$ is a minimiser of $\cE_\infty$,  $\dd \cE_\infty(\theta;\xi) = 0$ for all $\xi\in \Theta$.
Let $\Omega_N = [1/N,1-1/N]$ and $\gamma_N\in C^\infty$ be a cut off function that has support in $\Omega_{2N}$ and is identically one on $\Omega_N$.
Let $K_N = \gamma_N\odot K$.
We extend $K_N$ to the the whole of $\bbR$ by setting $K_N(t) = 0$ for all $t\in \bbR\setminus [0,1]$.
Clearly $K_N\in H^1(\bbR;\bbR^{d\times d})$, $K_N=K$ on $\Omega_N$ and $K_N$ has support in $\Omega_{2N}$.
Let $\xi = (L,0,0,0)$ where $L\in H^1([0,1];\bbR^{d\times d})$ satisfies $L(0) = 0$, then $\dd \cE_\infty(\theta;\xi) = 0$ implies
\[ \llan \dot{K},\dot{L}\rran_{L^2} = -\frac{1}{\alpha_1} \sum_{s=1}^S \nabla_z \cL(h(WX(1;x_s,\theta)+c), y_s) \cdot \Bigg[\dot{h}(WX(1;x_s,\theta)+c) \odot (WD_0(1;x_s,\theta,\xi))\Bigg]. \]
Using the equality above %
with $\gamma_N\odot L$ in place of $L$ implies,
\begin{align}
\langle \dot{K}_N,\dot{L}\rangle_{L^2} & = \langle \dot{\gamma}_N \odot K + \gamma_N\odot \dot{K},\dot{L} \rangle_{L^2} \notag \\
 & = \langle \dot{\gamma}_N \odot K, \dot{L} \rangle_{L^2} + \langle \dot{K},\gamma_N\odot \dot{L} \rangle_{L^2} \notag \\
 & = - \llan \frac{\dd}{\dd t} \l \dot{\gamma}_N\odot K\r, L \rran_{L^2} + \llan \dot{K},\frac{\dd}{\dd t} \l \gamma_N\odot L\r \rran_{L^2} - \langle \dot{K},\dot{\gamma}_N\odot L\rangle_{L^2} \notag \\
 & = - \llan \frac{\dd}{\dd t} \l \dot{\gamma}_N\odot K\r + \dot{\gamma}_N\odot \dot{K}, L \rran_{L^2} + \llan \dot{K},\frac{\dd}{\dd t}(\gamma_N\odot L)\rran_{L^2} \notag \\
 & = - \llan \frac{\dd}{\dd t} \l \dot{\gamma}_N\odot K\r + \dot{\gamma}_N\odot \dot{K}, L \rran_{L^2} -\frac{1}{\alpha_1} \sum_{s=1}^S \nabla_z \cL(h(WX(1;x_s,\theta)+c), y_s) \notag \\
 & \hspace{2cm} \cdot \Bigg[\dot{h}(WX(1;x_s,\theta)+c) \odot (WD_0(1;x_s,\theta,(\gamma_N\odot L,0,0,0)))\Bigg]. \label{eq:ProofConv:Reg:dotKNdotL}
\end{align}
%\begin{align}
%\langle \dot{K}_N,\dot{L}\rangle_{L^2} & = \langle K,\dot{\gamma}_N\odot \dot{L}\rangle_{L^2} + \langle \dot{K},\gamma_N\odot \dot{L}\rangle_{L^2} \notag \\
% & = \langle K,\dot{\gamma}_N\odot \dot{L}\rangle_{L^2} - \langle \dot{K},\dot{\gamma}_N\odot L\rangle_{L^2} + \llan \dot{K},\frac{\dd}{\dd t} (\gamma_N\odot L) \rran_{L^2} \notag \\
% & = \llan \dot{\gamma}_N\odot K - \frac{\dd}{\dd t} (\dot{\gamma}_N\odot K) ,L\rran_{L^2} -\frac{1}{\alpha_1} \sum_{s=1}^S \nabla_z \cL(h(WX(1;x_s,\theta)+c), y_s) \notag \\
%% & \hspace{2cm} + \frac{1}{\alpha_1} \sum_{s=1}^S \llan y_s, \dot{h}(WX(1;x_s,\theta)+c) \odot W\B{D_0(1;x_s,\theta,(\gamma_N\odot L,0,0,0)}) \rran. \\
% & \hspace{2cm} \cdot \Bigg[\dot{h}(WX(1;x_s,\theta)+c) \odot (W\B{D_0(1;x_s,\theta,(\gamma_N\odot L,0,0,0))})\Bigg]. \label{eq:ProofConv:Reg:dotKNdotL}
%\end{align}

We choose $L = L_{N,r}$ where
\[ L_{N,r}(t) = \frac{2K_N(t) - K_N(t+r) - K_N(t-r)}{r^2}. \]
Clearly $L_{N,r}\in H^1(\bbR;\bbR^{d\times d})$ for every $r>0$ and all $N>2$.
Furthermore, $L_{N,r}$ has support in $[\frac{1}{2N}-r,1-\frac{1}{2N}+r]$.
Since the support of $\dot{K}_N(\cdot-r)$ and  $\dot{K}_N$ is contained in $[r,1]$ for $r\leq\frac{1}{N}$,
\begin{align*}
\langle \dot{K}_N,\dot{L}_{N,r}\rangle_{L^2} & = \frac{1}{r^2} \int_0^1 \dot{K}_N(t) \l 2\dot{K}_N(t) - \dot{K}_N(t+r) - \dot{K}_N(t-r) \r \, \dd t \\
 & = \frac{1}{r^2} \int_0^1 \dot{K}_N(t) \l \dot{K}_N(t) - \dot{K}_N(t+r) \r \, \dd t \\
 & \hspace{2cm} + \frac{1}{r^2} \int_0^1 \dot{K}_N(t) \l \dot{K}_N(t) - \dot{K}_N(t-r) \r \, \dd t \\
 & = \frac{1}{r^2} \int_r^{1+r} \dot{K}_N(s-r) \l \dot{K}_N(s-r) - \dot{K}_N(s) \r \, \dd s \\
 & \hspace{2cm} + \frac{1}{r^2} \int_0^1 \dot{K}_N(t) \l \dot{K}_N(t) - \dot{K}_N(t-r) \r \, \dd t \\
 & = \frac{1}{r^2} \int_r^1 \dot{K}_N(s-r) \l \dot{K}_N(s-r) - \dot{K}_N(s) \r \, \dd s \\
 & \hspace{2cm} + \frac{1}{r^2} \int_r^1 \dot{K}_N(t) \l \dot{K}_N(t) - \dot{K}_N(t-r) \r \, \dd t \\
 & = \frac{1}{r^2} \int_r^1 \lda \dot{K}_N(t) - \dot{K}_N(t-r) \rda^2 \, \dd t,
\end{align*}
 where $\hat{\xi}_{N,r} = (\gamma_N\odot L_{N,r},0,0,0)$. From~\eqref{eq:ProofConv:Reg:dotKNdotL}, it follows that
\begin{align}
& \int_r^1 \lda \frac{\dot{K}_N(t) - \dot{K}_N(t-r)}{r} \rda^2 \, \dd t %\B{\leq  \l \| \dot{\gamma}_N\odot K \|_{L^2} + \lda \frac{\dd}{\dd t} (\dot{\gamma}_N\odot K)\rda_{L^2} \r \| L_{N,r} \|_{L^2} } \notag \\
%& \qquad \qquad + \frac{1}{\alpha_1} \sum_{s=1}^S \| \nabla_z \cL(h(WX(1;x_s,\theta)+c), y_s)\| \| \dot{h}(WX(1;x_s,\theta)+c)\| \| W\| \| D_0(1;x_s,\theta,(\gamma_N\odot L,0,0,0)))\| \notag \\
%& \qquad
\leq C_1 \|L_{N,r}\|_{L^2([0,1])} + C_2 \sum_{s=1}^S \|D_0(1;x_s,\theta,\hat{\xi}_{N,r}) \|, \label{eq:ProofConv:Reg:dotKDiffBound}
\end{align}
where
\begin{align*}
C_1 & = \lda \frac{\dd}{\dd t} (\gamma_N\odot K)\rda_{L^2} + \| \dot{\gamma}_N\odot \dot{K} \|_{L^2}\text{ and} \\
C_2 & = \frac{1}{\alpha_1} \lda \dot{h}(WX(1;x_s,\theta)+c)\rda \| W\| \max_{s=1,\dots,S} \lda \nabla_z \cL(h(WX(1;x_s,\theta)+c), y_s)\rda
\end{align*}
(we  note that the constants $C_1, C_2$ may depend on $N$, but do not depend on $r$).

We note  that we can rewrite $L_{N,r} = \frac{2-\tau_r-\tau_{-r}}{r^2} K_N = \frac{(1-\tau_r)(1-\tau_{-r})}{r^2} K_N$, where $\tau_r$ is the shift operator defined by $\tau_r\varphi(x) = \varphi(x+r)$.
By~\cite[Theorem 10.55]{leoni09} for any $\psi\in H^1([0,1]; \R^{d\times d})$ we have $\lda \frac{\tau_r-1}{r}\psi(t+r)\rda_{L^2([r,1-r])} \leq \|\dot{\psi}\|_{L^2([r,1])}$.
Applying this to $\psi = \frac{(1-\tau_{-r})K_N}{r}$ we have, for $r$ sufficiently small,
\begin{equation} \label{eq:ProofConv:Reg:LnhBound}
\|L_{N,r}\|_{L^2([0,1])}^2 = \lda \frac{1-\tau_r}{r}\psi \rda_{L^2([r,1-r])}^2 \leq \|\dot{\psi}\|_{L^2([r,1])}^2 = \int_r^1 \lda \frac{\dot{K}_N(t) - \dot{K}_N(t-r)}{r}\rda^2 \, \dd t.
\end{equation}

We can write $D_0(1;x_s,\theta,\hat{\xi}_{N,r}) = \int_0^1 A_{N,s}(t) \odot L_{N,r}(t) \, \dd t$ where
\begin{align*}
A_{N,s}(t) & = B_s(t) \gamma_N(t), \\
B_s(t) & = \exp\l\int_t^1 \dot{\sigma}\l K(u)X_s(u)+b(u)\r\odot K(u) \, \dd u\r \odot \dot{\sigma}\l K(t)X_s(t)+b(t)\r X_s(t), \text{ and} \\
X_s(t) & = X(t;x_s,\theta).
\end{align*}
Hence,
\begin{equation} \label{eq:ProofConv:Reg:DBound}
\|D_0(1;x_s,\theta,\hat{\xi}_{N,r})\|\leq C_3 \|L_{N,r}\|_{L^2([0,1])}.
\end{equation}

Combining~\eqref{eq:ProofConv:Reg:DBound} with~\eqref{eq:ProofConv:Reg:dotKDiffBound} and~\eqref{eq:ProofConv:Reg:LnhBound} and Young's inequality we obtain
\[ \int_r^1 \lda \frac{\dot{K}_N(t) - \dot{K}_N(t-r)}{r} \rda^2 \, \dd t \leq C_4. \]
Hence by~\cite[Theorem 10.55]{leoni09} $\dot{K}_N\in H^1([0,1];\bbR^{d\times d})$.
Since this is true for all $N$, we have that $\dot{K}\in H^1_{\loc}([0,1];\bbR^{d\times d})$.
Hence, $K\in H^2_{\loc}([0,1];\bbR^{d\times d})$.

The argument for $b\in H^2_{\loc}([0,1];\bbR^d)$ is analogous.
\end{proof}

\subsection{The Forward Pass as a Discretized ODE}\label{sec:networkasdynamicalsystem}

In this section we prove Corollary~\ref{cor:convergenceforwardpass}.

\begin{lemma}\label{lem:convergenceforwardpass}
Let $K \in H^1([0,1];\bbR^{d\times d})$, $b\in H^1([0,1];\bbR^d)$, and let $\sigma: \R^d \to \R^d$ be Lipschitz continuous with Lipschitz constant $L_\sigma>0$.
Let $x\in \R^d$ and suppose that $X:[0,1]\to \bbR^d$ is the solution to the ODE in \eqref{eq:Intro:Limit:ODE} with initial condition $X(0)=x$. Let $n\in \N$ %(for definiteness, we assume $0\not\in \N$)
and let $K^{(n)} \in L^0(\mu_n; \bbR^{d\times d})$, $b^{(n)} \in L^0(\mu_n; \bbR^d)$ be such that there exists a $\delta_n>0$ such that, for all $i\in \{0, 1, \dots, n-1\}$, $\|K_i^{(n)} - K(i/n)\| < \delta_n$ in matrix operator norm and $\|b_i^{(n)}-b(i/n)\| < \delta_n$.
Moreover, let $X_i^{(n)}$ ($i=0, 1, \ldots, n$) be the solutions to \eqref{eq:Intro:ResNet:Rec} with $X_0^{(n)} = x$. Then there exists an $\eps_n \in \R$ such that, for all $i \in \{0, 1, \ldots, n\}$, \eqref{eq:Xln-Xln}
is satisfied with $\delta=\delta_n$ and $A_n = \frac1n \left(1+ \|X\|_{L^\infty}\right) L_\sigma \delta_n + \eps_n$. Moreover, $\eps_n = o\left(\frac1n\right)$ as $n\to\infty$.
\end{lemma}

\begin{proof}
We follow closely standard proofs of the convergence of the explicit Euler scheme for well-posed ODEs, \cite[Theorem 5.9]{burden2016numerical}, \cite[Section 6.3.3]{leveque2007finite}.

First we note that the case $i=0$ is trivial, so we will consider $i\geq1$ from here on.

By the Sobolev embedding theorem \cite{adams2003sobolev} $K$ and $b$ are continuous.
Since $y \mapsto \sigma(K(t)y+b(t))$ is Lipschitz continuous, by standard ODE theory \cite{Hale2009} there is a unique solution $X$ to~\eqref{eq:Intro:Limit:ODE} and this $X$ is continuous.
Moreover, $t\mapsto \sigma\big(K(t) X(t) + b(t)\big)$ is continuous and thus $\dot X$ is continuous. In particular, $\dot X$ is bounded on $[0,1]$.
Let $n, k\in \N$ with $k\leq n$.
We compute, using Taylor's theorem, for $C^1$ functions \cite[(2.22)]{duistermaat2004multidimensional},
\begin{align*}
X(k/n) &= X((k-1)/n) + \frac1n \dot X((k-1)/n) + r_{k,n}\\
&= X((k-1)/n) + \frac1n \sigma(K((k-1)/n) X((k-1)/n) + b((k-1)/n))+ r_{k,n},
\end{align*}
where $r_{k,n}\in \R^d$ is such that $\|r_{k,n}\| = o\left(\frac1n\right)$ as $n\to \infty$.
Moreover
\[ X_k^{(n)} = X_{k-1}^{(n)} + \frac1n \sigma(K_{k-1}^{(n)} X_{k-1}^{(n)} + b_{k-1}^{(n)}) \]
and thus
\begin{align*}
X(k/n)  - X_k^{(n)} & =  X((k-1)/n) - X_{k-1}^{(n)}\\
&\hspace{0cm} + \frac1n \left(\sigma(K((k-1)/n) X((k-1)/n) + b((k-1)/n)) - \sigma(K_{k-1}^{(n)} X_{k-1}^{(n)} + b_{k-1}^{(n)})\right) \\
&\hspace{0cm} + r_{k,n}.
\end{align*}
Using $\|K_i^{(n)} - K(i/n)\| < \delta_n$ and $\|b_i^{(n)}-b(i/n)\| < \delta_n$ and the fact that $L_\sigma >0$ is a Lipschitz constant for $\sigma$, we find
\begin{align*}
&\hspace{0.3cm} \left\|\sigma(K((k-1)/n) X((k-1)/n) + b((k-1)/n)) - \sigma(K_{k-1}^{(n)} X_{k-1}^{(n)} + b_{k-1}^{(n)})\right\|\\
& \leq L_\sigma \left\|K((k-1)/n) X((k-1)/n) + b((k-1)/n) -(K_{k-1}^{(n)} X_{k-1}^{(n)} + b_{k-1}^{(n)})\right\|\\
& \leq L_\sigma \left\|K((k-1)/n)  \left(X((k-1)/n)-  X_{k-1}^{(n)}\right)\right\| + L_\sigma \|\left(K((k-1)/n)-K_{k-1}^{(n)}\right) X_{k-1}^{(n)}\|\\
&\hspace{0.4cm} + L_\sigma \|b((k-1)/n) -  b_{k-1}^{(n)}\|\\
&\leq L_\sigma \|K\|_{L^\infty} \left\|X((k-1)/n)-  X_{k-1}^{(n)}\right\| + L_\sigma \delta_n \|X_{k-1}^{(n)}\| + L_\sigma \delta_n \\
&\leq L_\sigma (\|K\|_{L^\infty}+\delta_n) \left\|X((k-1)/n)-  X_{k-1}^{(n)}\right\| + L_\sigma \delta_n \|X\|_{L^\infty} + L_\sigma \delta_n.
\end{align*}
For the final inequality, we used
\[
\|X_{k-1}^{(n)}\| \leq \left\|X((k-1)/n)-  X_{k-1}^{(n)}\right\| + \|X((k-1)/n)\| \leq \left\|X((k-1)/n)-  X_{k-1}^{(n)}\right\| + \|X\|_{L^\infty}.
\]
Since $K$ is continuous on $[0,1]$, we have $\|K\|_{L^\infty} <\infty$. Similarly, since $X$ is continuous, we have $\|X\|_{L^\infty} < \infty$. Combining the above we get
\begin{equation}\label{eq:X-X}
\begin{aligned}
\left\|X(k/n)  - X_k^{(n)}\right\| &\leq  \left(1+\frac1n L_\sigma (\|K\|_{L^\infty}+\delta_n)\right) \left\|X((k-1)/n)-  X_{k-1}^{(n)}\right\| \\
&\hspace{0.7cm} + \frac1n \left(1+\|X\|_{L^\infty}\right) L_\sigma \delta_n + \eps_n,
\end{aligned}
\end{equation}
where we have defined $\eps_n = \max_{1\leq k\leq n} \|r_{k,n}\|$.
We note that $\eps_n=o\left(\frac1n\right)$ as $n\to \infty$ and recall that $A_n=\frac1n \left(1+\|X\|_{L^\infty}\right) L_\sigma \delta_n + \eps_n$.
Write $a_k = \left\|X(k/n)  - X_k^{(n)}\right\|$ and $C=1+\frac1n L_\sigma (\|K\|_{L^\infty}+\delta_n)$, where we have repressed the depency on $n$ for notational simplicity.
Let $i \in \N$ with $i\leq n$. We claim that
\begin{equation}\label{eq:alclaim}
a_i \leq A_n \sum_{j=0}^{i-1} C^j.
\end{equation}
We prove this claim by induction. Since \eqref{eq:X-X} holds for arbitrary $k$, we have $a_1 \leq C a_0 + A_n$ directly from \eqref{eq:X-X}. Since $a_0 = \|x-x\|=0$, \eqref{eq:alclaim} holds for $i=1$. Now let $k\in \N$ with $k\leq n$ and assume that \eqref{eq:alclaim} holds for $i=k-1$. Then, combining \eqref{eq:alclaim} with \eqref{eq:X-X} we deduce that
\[ a_k \leq C a_{k-1} + A_n \leq C\left(A_n \sum_{j=0}^{k-1} C^j\right) + A_n = A_n \left(1+\sum_{j=0}^{k-1} C^{j+1}\right) = A_n \sum_{j=0}^{k} C^j. \]
Thus claim \eqref{eq:alclaim} is proven. Since $C>1$,  we compute
\[ \sum_{j=0}^{i-1} C^j = \frac{1-C^i}{1-C} = \frac{n}{L_\sigma (\|K\|_{L^\infty}+\delta_n)} \left[\left(1+\frac1n L_\sigma (\|K\|_{L^\infty}+\delta_n)\right)^i-1\right]. \]
Using that $\left(1+\frac1n L_\sigma(\|K\|_{L^\infty}+\delta_n)\right)^i \leq \exp\left(\frac{i}n L_\sigma (\|K\|_{L^\infty}+\delta_n)\right)$, we find that
\[ a_i \leq \frac{n}{L_\sigma(\|K\|_{L^\infty}+\delta_n)} A_n  \left[\exp\left(\frac{i}n L_\sigma (\|K\|_{L^\infty}+\delta_n)\right)-1\right], \]
as required.
\end{proof}

We now check that the conditions of Lemma~\ref{lem:convergenceforwardpass} hold.

\begin{lemma}
\label{lem:ProofConv:Forward:UniformConv}
Let $\Theta^{(n)}$ and $\Theta$ be given by~\eqref{eq:Prelim:Top:Thetan} and~\eqref{eq:Prelim:Top:Theta} respectively.
Define $\cE_n$, $\cE_\infty$, $E_n$, $E_\infty$, $R^{(i)}_n$, $R^{(i)}_\infty$, $R^{(j)}$ for $i=1,2$, $j=3,4$ as in Sections~\ref{subsec:Intro:Finite}-\ref{subsec:Intro:Limit}.
Assume that the assumptions of Theorem~\ref{thm:MainRes:Conv} hold.
If $\{(K^{(n)},b^{(n)},W^{(n)},c^{(n)})\} \subset \Theta^{(n)}$ is a sequence of minimisers of $\cE_n$ and $(K,b,W,c)\in \Theta$ is the minimiser of $\cE_\infty$ which we assume to be unique then we have
\[ \max_{i\in\{0,\dots,n-1\}} \left\| K\l\frac{i}{n}\r-K^{(n)}_i\right\| \to 0 \quad \text{and} \quad\max_{i\in\{0,\dots,n-1\}} \left\| b\l\frac{i}{n}\r-b^{(n)}_i\right\| \to 0, \]
as $n\to\infty$.
\end{lemma}

\begin{proof}
Let $(K^{(n)},b^{(n)},W^{(n)},c^{(n)})$ minimise $\cE_n$.
Choose any subsequence $\{n_m\}_{m\in\bbN}$ of $\bbN$.
By Theorem~\ref{thm:MainRes:Conv} there exists a further subsequence that converges to a minimiser $(K,b,W,c)$ of $\cE_\infty$.
Since the minimiser is unique, we have $(K^{(n_m)},b^{(n_m)},W^{(n_m)},c^{(n_m)})\to (K,b,W,c)$.
Furthermore, since $\cE_{n_m}(K^{(n_m)},b^{(n_m)},W^{(n_m)},c^{(n_m)})<+\infty$, we have, by Proposition~\ref{prop:ProofConv:Compact:SobEmb}, that there exists a further subsequence $\{n_{m_k}\}_{k\in\bbN}$ such that
\[ \max_{i\in\{0,\dots, n_{m_k}-1\}} \left\| K\l\frac{i}{n_{m_k}}\r - K_i^{(n_{m_k})} \right\| \to 0, \qquad \max_{i\in\{0,\dots, n_{m_k}-1\}} \left\| b\l\frac{i}{n_{m_k}}\r - b_i^{(n_{m_k})} \right\| \to 0, \]
as $k\to\infty$.
We have that any subsequence of $(K^{(n)},b^{(n)},W^{(n)},c^{(n)})$ contains a further subsequence that converges uniformly.
Now if we suppose that $(K^{(n)},b^{(n)},W^{(n)},c^{(n)})$ does not converge uniformly to $(K,b,W,c)$, then there exists an $\eps>0$ and a subsequence (which we index by $n_m$) such that the $L^\infty$ norm of $(K^{(n_m)}-K,b^{(n_m)}-b,W^{(n_m)}-W,c^{(n_m)}-c)$ is bounded from below by $\eps$.
But this subsequence cannot now contain a further subsequence that converges uniformly; a contradiction.
It follows that uniform convergence holds across the whole sequence as required.
\end{proof}

The proof of Corollary~\ref{cor:convergenceforwardpass} follows directly from Lemmas~\ref{lem:convergenceforwardpass} and~\ref{lem:ProofConv:Forward:UniformConv}.

\section{Discussion and Conclusions \label{sec:Conc}}

In this paper we proved that the variational limit of the residual neural network is an ODE system, thereby rigorously justifying the observations in~\cite{haber2017learning,E2017}.
These and similar observations have already inspired new architectures for neural networks, e.g.~\cite{haber17,treister18,ruthotto18,lu17} and the hope is that this work can help in the justification and analysis of these new architectures.
In addition, we proved a regularity result for the coefficients obtained by ResNet training.

We left the question of rates of convergence for the minimisers open (see after Proposition~\ref{prop:MainRes:Reg}). We believe this can be approached through a higher order $\Gamma$-convergence argument, for example see~\cite[Theorem 1.5.1]{braides14}, and a coercivity argument, but it falls outside the scope of this current paper.

An interesting open question which the authors intend to field in future work, is to recover partial differential equations by simultaneously taking $d\to\infty$ (where $d$ is the number of neurons per layer) and $n\to \infty$.
This will mean imposing certain restrictions on the inter-layer connections; in particular, the choice of inter-layer connections is expected to alter the continuum partial differential equation limit.

%\BN
Another open question concerns our use of explicit regularisation terms in the cost function. In practice often implicit regularisation techniques are used, such as dropout or stochastic gradient descent~\cite{wan13,srivastava14,haeffele17,mianjy18,neyshabur14}. Incorporating these methods into our setting requires the rigorous mathematical establishment of their regularising effects, which to the best of our knowledge, has not been accomplished as of yet.
However, recent work~\cite{cohen21} shows that, at least in certain circumstances, the deep layer limit in the absence of explicit regularisation results in a stochastic limit.

In this paper we have established convergence at a variational level. A third open question of interest relates to the convergence of the corresponding gradient flow for the parameters. Except in certain special circumstances, gradient flow convergence does not follow directly from $\Gamma$-convergence; in this case an additional difficulty that needs to be taken into account is the ODE constraint. The authors are planning to address this question in future work.
%\EN

\subsection*{Data Availability}

Data sharing is not applicable to this article as no datasets were generated or analysed during the current study.

\subsection*{Acknowledgements}

This project has received funding from the European Research Council (ERC) under the European Union's Horizon 2020 research and innovation programme (grant agreements 777826 (NoMADS) and 647812).
The authors would like to thank Martin Benning, Jeff Calder, Matthias J. Ehrhardt, Nicol\'as Garc\'ia Trillos and Lukas Lang for enlightening exchanges regarding the work in this paper.
We also thank the anynomous referees for their very insightful comments on an earlier version of the manuscript, which have led to improvements in the final paper.
MT is grateful for the support of the Cantab Capital Institute for the Mathematics of Information (CCIMI) and the Cambridge Image Analysis group at the University of Cambridge.
YvG did a significant part of the work which has contributed to this paper at the University of Nottingham.
The authors state that there are no conflicts of interest.

\bibliographystyle{plain}
\bibliography{references}

\appendix

\section{The Matrix Exponential \label{sec:App:Matrices}}

For completeness we include a short proof of the inequality
\[ \lda e^{X+Y} - e^X \rda \leq \| Y\| e^{\|X\|} e^{\|Y\|}, \]
for any square matrices $X,Y$, which is  used in the proof of Lemma~\ref{lem:ProofConv:Reg:ODEPert}.
Let $k\in\bbN\setminus\{0\}$. Then,CA
\[ \lda (X+Y)^k - X^k \rda \leq \sum_{j=0}^{k-1} {k\choose j} \| X\|^j \|Y\|^{k-j} = \l \| X\| + \|Y\|\r^k - \|X\|^k, \]
where we obtained the inequality by expanding $(X+Y)^k$, applying the triangle inequality and using that $\|XY\|\leq \|X\| \|Y\|$. Then the equality follows from the binomial theorem.
Hence,
\begin{align*}
\lda e^{X+Y} - e^X \rda & = \lda \sum_{k=0}^\infty \frac{1}{k!} \l (X+Y)^k - X^k \r \rda \\
 & \leq \sum_{k=1}^\infty \frac{1}{k!} \lda (X+Y)^k - X^k \rda \\
 & \leq \sum_{k=1}^\infty \frac{1}{k!} \l \| X\| + \|Y\|\r^k - \sum_{k=1}^\infty \frac{1}{k!} \|X\|^k \\
 & = e^{\|X\|+\|Y\|} - e^{\|X\|} \\
 & = e^{\|X\|}\l e^{\|Y\|} - 1\r.
\end{align*}
Since  $e^c-1\leq ce^c$ for all $c\geq 0$, we conclude that the inequality holds.
\end{document}